\documentclass[11pt]{article}
\usepackage{authblk}
\usepackage[top=1in,bottom=1in,left=1in,right=1in]{geometry}
\usepackage{amssymb}
\usepackage{longtable}
\usepackage{amsmath}
\usepackage{amsthm}
\usepackage[normalem]{ulem}
\usepackage{mathtools}
\usepackage{xcolor}
\usepackage[page]{appendix}
\usepackage{enumitem}
\usepackage{hyperref}
\usepackage{bm}
\usepackage{caption}
\usepackage{footmisc}
\usepackage{comment}
\usepackage{natbib}

\newcommand{\bbR}{\mathbb{R}}

\renewcommand{\P}{\mathbb{P}}

\newcommand{\U}{\mathcal{U}}

\newcommand{\op}{\mathrm{op}}

\newcommand{\bA}{\bm{A}}
\newcommand{\bB}{\bm{B}}
\newcommand{\bC}{\bm{C}}
\newcommand{\bM}{\bm{M}}
\newcommand{\bQ}{\bm{Q}}
\newcommand{\bE}{\bm{E}}

\newcommand{\bQhat}{\widehat{\bm{Q}}}
\newcommand{\bu}{\bm{u}}
\newcommand{\bv}{\bm{v}}
\newcommand{\ba}{\bm{a}}
\newcommand{\bb}{\bm{b}}

\newcommand{\bg}{\bm{g}}
\newcommand{\opnorm}[1]{\|#1\|_{\mathrm{op}}}

\newcommand{\btheta}{\bm{\theta}}
\newcommand{\bI}{\bm{I}}
\newcommand{\bbW}{\mathbb{W}}
\newcommand{\cR}{\mathcal{R}}
\newcommand{\cB}{\mathcal{B}}
\newcommand{\cX}{\mathcal{X}}
\newcommand{\Var}{\operatorname{Var}}
\newcommand{\bbE}{\mathbb{E}}
\newcommand{\bbP}{\mathbb{P}}
\newcommand{\bX}{\bm{X}}

\newcommand{\bT}{\bm{T}}
\newcommand{\bS}{\bm{S}}
\newcommand{\bD}{\bm{D}}
\newcommand{\bP}{\bm{P}}

\newcommand{\bbI}{\mathbb{I}}
\newcommand{\indicator}[1]{{\operatorname{I} \left \{ #1\right \}}}

\newcommand{\bbN}{\mathbb{N}}
\newcommand{\bbA}{\mathbb{A}}

\newcommand{\bbQ}{\mathbb{Q}}

\newcommand{\constantBandeira}{C_{\text{Pr.\ref{proposition: matrix concentration}}}}
\newcommand{\KL}[2]{\operatorname{KL}\left (#1\Vert #2 \right )}

\newcommand{\regret}{\operatorname{Regret}}
\newcommand{\ttW}{\mathtt{W}}
\newcommand{\ttd}{\mathtt{d}}
\newcommand{\cS}{\mathcal{S}}
\newcommand{\ttg}{\mathtt{g}}
\newcommand{\avgMatrix}{\overline{\bm{E}}}
\newcommand{\avgLessT}{\overline{\bm{E}}^{\le\tau}}
\newcommand{\avgLarge}{\overline{\bm{E}}^{> \tau}}
\newcommand{\ts}{\mathtt{s}}
\newcommand{\ttb}{\mathtt{b}}
\newcommand{\ttD}{\mathtt{D}}
\newcommand{\bj}{\mathbf{j}}

\DeclareMathOperator{\rank}{rank}
\DeclareMathOperator*{\argmax}{argmax}
\DeclareMathOperator*{\argmin}{argmin}

\theoremstyle{plain}
\newtheorem{theorem}{Theorem}
\newtheorem{lemma}{Lemma}
\newtheorem{corollary}{Corollary}
\newtheorem{prop}{Proposition}

\theoremstyle{definition}
\newtheorem{definition}{Definition}

\theoremstyle{remark}
\newtheorem{remark}{Remark}

\providecommand{\keywords}[1]
{
  \textbf{\textit{Keywords:}} #1
}

\title{Low-Rank Graphon Estimation: Theory and Applications to Graphon Games}
\author{Olga Klopp \\\vskip -0.4 cm ESSEC Business School\\
Fedor Noskov \\ HSE University 
        }
\date{}
\begin{document}
\maketitle
\begin{abstract}
We study low-rank estimation of an unknown sparse graphon from sampled network data under operator-norm loss, motivated by targeted interventions in graphon games. Starting from the observed adjacency matrix, we construct low-rank surrogates by singular value thresholding and, for smooth graphons, by block averaging followed by thresholding. We obtain non-asymptotic bounds on both the operator-norm error and the rank of the resulting estimator for stochastic block model, Hölder, and analytic graphons, and we complement these results with minimax lower bounds showing that the rates are essentially sharp for these classes. Our analysis highlights that low rank is valuable here primarily for computation: while it does not improve the minimax operator-norm rate, it yields operator-norm accurate surrogates with substantially smaller rank. We then apply these estimators to linear-quadratic graphon games and derive non-asymptotic stability bounds showing that the welfare loss incurred by using an estimated graphon is controlled by the operator-norm perturbation. This yields near-optimal guarantees for targeted interventions computed from the estimated graphon, together with substantial computational savings. For zero baseline heterogeneity and under a spectral-gap condition, we also establish matching lower bounds for intervention regret. Numerical experiments illustrate the trade-off between statistical accuracy, retained rank, and runtime.

\end{abstract}
\\


\section{Introduction}


Large networks are ubiquitous in today's world, present in various systems such as social networks, epidemic models, power grids, the Internet of Things, and more. These networked systems possess a unique characteristic: the agents within them are interconnected via the network. Consequently, the efficiency of algorithms operating on these systems scales with the size of the network, denoted as $n$.
For instance, consider linear-quadratic network games, where computing a Nash equilibrium of the game exhibits a worst-case complexity of $O(n^3)$ \citep{galeotti_interventions}. In practical scenarios,  $n$  can reach substantial magnitudes; for instance, Facebook boasts over 2 billion users. Consequently, for computationally intensive tasks, such algorithms can become practically intractable.
To address this challenge, the theory of graphons has emerged as a valuable tool. For example, in the context of network games, graphons have been employed to mitigate computational complexities \citep{Parise_Ozdaglar_Econometrica}.

Formally, a graphon is a measurable function $W$ given by $W : [0,1]^2 \rightarrow [0,1]$. A graphon can either be interpreted as a limit of a sequence of finite networks, or as a probabilistic model that generates random finite networks \citep{lovasz}. While, traditionally, networked systems   have been modeled on finite networks, these models have been extended to graphons, e.g., \citep{Parise_Ozdaglar_Econometrica} extended network games to graphon games. These works have shown that modeling networked systems as graphon systems has three benefits: (i) graphons provide a mathematically elegant, unified framework for representing and comparing networked systems of various sizes, (ii) a graphon system can closely approximate the behavior of a networked system if the network has been sampled from the graphon, (iii) under certain conditions on the graphon, the dimension of some computational problems involving the networked system can be significantly reduced, thus leading to computationally efficient methods. 

{Crucially, the works on graphon games mentioned above assumed that the true graphon that generates
finite networks is known to the planner. In practice, however, the underlying graphon is unknown,
and the planner only observes one (or a few) realized networks. This motivates the problem of
estimating a graphon from sampled network data, a topic extensively studied in the literature \citep{chatterjee2012,gao2014rate,klopp2017,klopp2019,xu2017}.

In connection to our motivating application, targeted interventions that maximize social welfare at a Nash
equilibrium, the estimator must be accurate in a metric that controls the intervention loss.
Our non-asymptotic stability result (Theorem~\ref{theorem: graphon perturbed welfare problem}) shows that the achieved welfare is Lipschitz (with explicit constants) in the operator norm distance
$\|W_1-W_2\|_{\mathrm{op}}$. This suggests that, for the downstream objective \eqref{eq: graphon social welfare problem}, an estimator should satisfy:
\textbf{(P1)} it is close to the underlying graphon in operator norm.

At first glance, one might then simply take the empirical graphon associated with the observed adjacency matrix
(or equivalently use the adjacency matrix itself), since in operator norm the adjacency matrix already achieves
the minimax-optimal scaling over standard low-rank/SBM-type families. In fact, our Theorem~\ref{theorem: UHTE performance} recovers this
benchmark: setting the SVT threshold $\lambda=0$ yields $\widehat{\mathbf Q}_{0}=\mathbf A$ and an operator-norm
error of order $\sqrt{n\rho_n}$ (up to logarithmic factors). Therefore, throughout this paper low-rankness
should \emph{not} be interpreted as providing a better minimax \emph{statistical} rate in operator norm.

Instead, the role of low-rank structure is primarily \emph{computational} and is dictated by the intervention
problem \eqref{eq: graphon social welfare problem}. While $\mathbf A$ is statistically optimal in operator norm, the corresponding empirical graphon
typically has rank of order $n$, making the computation of interventions prohibitively expensive at scale.
In contrast, existing results in graphon games show that if the graphon has rank $r$, then the intervention
problem can be reduced to a semidefinite program with only two variables and an LMI constraint whose matrix
dimension scales as $r+2$ \citep{Parise_Ozdaglar_Econometrica}. This motivates the second desired property:
\textbf{(P2)} the estimator should have low (effective) rank, because the rank is the algorithmic bottleneck
in solving \eqref{eq: graphon social welfare problem}.

To achieve (P1)--(P2) simultaneously, we use a singular value thresholding (SVT) estimator. Given a network of
size $n$ sampled from an unknown sparse graphon $\rho_n W$, we form a low-rank approximation of the adjacency
matrix by retaining only those singular components whose singular values exceed a threshold $\lambda$, yielding
$\widehat{\mathbf Q}_\lambda$. We then define the estimated graphon as the empirical (piecewise-constant)
graphon associated with $\widehat{\mathbf Q}_\lambda$.

A key feature is that $\lambda$ acts as a \emph{computational knob}. Theorem~\ref{theorem: UHTE performance} bounds the operator-norm error by
an additive $\lambda$ term plus the intrinsic sampling noise (of order $\sqrt{n\rho_n}$), so taking $\lambda$ at
(or slightly above) the noise level preserves the optimal statistical scaling in operator norm while potentially
shrinking the retained rank
\[
R(\lambda)\;:=\;\mathrm{rank}(\widehat{\mathbf Q}_\lambda)\;=\;\#\{i:\sigma_i(\mathbf A)\ge \lambda\}.
\]
This rank directly governs both (i) the cost of computing $\widehat{\mathbf Q}_\lambda$ via truncated SVD methods
(up to rank $R(\lambda)$) and (ii) the dimension of the reduced SDP used to compute interventions. We then show
that for structured families (SBM, H\"older, analytic graphons) one can obtain explicit bounds on the rank of the
SVT estimator, which grow slowly with $n$, thereby yielding operator-norm accurate surrogates that make the
intervention computation tractable.
}

The rest of the paper is organized as follows. In Section \ref{section: preliminaries}, we introduce some preliminaries and our key estimation problem. In Section \ref{section: motivating applications}, we provide examples of  applications to graphon games. 
 In Section \ref{section: estimation}, we introduce our estimator, and derive bounds on (i) the operator norm of the difference between the estimated graphon and the true graphon (ii) the rank of the estimated graphon. In Section \ref{section:simulations}, we apply our results on graphon estimation to the problem of targeted interventions in network games and provide numerical experiments to illustrate our findings. The proofs of all the results are given in the appendix.

\subsubsection*{Notation}
Let $[n] := \{1, \dots, n\}$. Given two numbers $a, b$, denote their maximum (minimum) by $a \vee b$ ($a \wedge b$).
We use capital letters such as $W$ to denote a graphon. Let $I$ be the indicator function. We use bold letters to denote vectors and matrices, e.g., $\ba$ and $\bA$ respectively. Let $\bI$ be the identity matrix. Let $\bA_{ij}$ denote the $ij$th element of matrix $\bA$, and $[\bA_{ij}]$ denote the matrix whose elements are $\bA_{ij}$. Let $\bA \circ \bB$ denote the Hadamard or the elementwise product of matrices $\bA$ and $\bB$. 
Let $\|\ba\|_2$ denote the Euclidean $2$-norm of vector $\ba$, and $\|\bA\|_{\op}$ the induced $2$-norm, or the operator norm of matrix $\bA$. 
Small letters such as $f$ denote a function. Let $L_2[0,1]$ be the Hilbert space of square-integrable functions on $[0,1]$ equipped with the inner-product
\begin{align*}
    \langle f_1, f_2 \rangle_{L_2} = \int_0^1 f_1(x) f_2(x) dx,
\end{align*}
The associated norm is denoted by $\Vert f \Vert_{L_2} := \sqrt{\langle f, f \rangle_{L_2}}$.
We use $\U_1, \dots, \U_n$ to denote the uniform partition of $[0,1]$. To distinguish between a graphon $W$ and the associated linear operator $\bbW: L_2([0, 1]) \to L_2([0,1])$, we use blackboard bold font for the latter.


\section{Preliminaries}
\label{section: preliminaries}

We start with a brief recap of graphon theory and  its connection with network theory. We refer the interested reader to \cite{lovasz} for more details.


\subsubsection*{Graphons as operators}
Throughout this paper the term ``graphon'' denotes a symmetric measurable function  $W : [0,1]^2 \rightarrow [0,1]$. 
Let $\mathcal{W}$ denote the space of all bounded symmetric measurable functions $W : [0,1]^2 \rightarrow \mathbb{R}$. The elements of $\mathcal{W}$ are sometimes called ``kernels''. 

In the following, it will be useful to  think of graphons as operators on $L_2[0,1]$. Specifically,
we will denote by $\bbW$  the integral operator associated with graphon $W$, that is, 
$[\bbW f](x)= \int_0^1 W(x, y) f(y) dy$ for all $x \in [0,1]$, for all $f \in L_2[0,1]$,
with the associated operator norm 
\begin{align*}
	\opnorm{\bbW} = \sup_{f \in L_2[0,1], \; \Vert f \Vert_{L_2} = 1} \Vert \bbW f \Vert_{L_2}. 
\end{align*}

\subsubsection*{Graphons as random network models}

In relation to graphs, graphons can be seen both as the limit of a sequence of graphs when the number of nodes tends to infinity, as well as a way to generate random networks. Specifically, given a graphon $W$, one can sample a network or an adjacency matrix $\bA \in \{0,1\}^{n \times n}$ from $W$  of any desired size $n$ using the following procedure \citep{lovasz}:
\begin{enumerate}
\item Sample $n$ points  $\xi_1, \dots, \xi_n$  uniformly at random from $[0,1]$. 
\item Fix a sparsity parameter $\rho_n \in [0,1]$.\footnote{Note that the expected number of neighbors of a node in the graph defined by $\bA$ is of the order $\rho_n n$, and hence the expected number of edges in the graph is $\rho_n n^2$. Hence, the density of the graph is of the order $\rho_n n^2/n^2 = \rho_n$. Thus, the sparsity parameter characterizes the density of the graph defined by $\bA$.}
\item Sample a weight matrix $\bQ \in [0,1]^{n \times n}$ from   $W$ such that
\begin{align}
\label{eq: Q sampling}
    \bQ_{ij} = \rho_n W(\xi_{(i)},\xi_{(j)}) \quad \forall (i,j) \in [n]^2, 
\end{align}
where $\xi_{(i)}$ is an $i$-th ordered statistics of $\xi_1, \ldots, \xi_n$.
\item Sample an adjacency matrix $\bA \in \{0,1\}^{n \times n}$ from $\bQ$ such that
\begin{equation}
\label{eq: A sampling}
\begin{aligned}
    \bA_{ij} &= \begin{cases}
        1, & \textrm{w.p. } \bQ_{ij}, \\
        0, & \textrm{w.p. } 1 - \bQ_{ij},
    \end{cases} \quad \forall (i,j) \in [n]^2, \ i \neq j, \\
    \bm{A}_{ii} &= 0 \quad \forall i \in [n]. 
\end{aligned}
\end{equation}
\end{enumerate}

The more common graphon model assumes that $Q_{ij} = \rho_n W(\xi_i, \xi_j)$ instead of~\eqref{eq: Q sampling}. However, it leads to identifiability issues. Ensuring that each player's identity remains distinguishable is crucial in the context of graphon games, for example, for targeted interventions which motivates our choice of sampling model ~\eqref{eq: Q sampling}.
 By allowing the arguments to represent intrinsic characteristics of the agents - such as age or wealth, - we avoid such issues. Note that the results of Sections \ref{section: estimation} can be easily extended to the model given by $Q_{ij} = \rho_n W(\xi_i, \xi_j)$.

\subsubsection*{Empirical graphons}
Let $\bT \in \bbR^{n \times n}$ be an arbitrary matrix  and $\indicator{\cdot}$ be the indicator function on $[0,1]$. We define the empirical graphon or simply the graphon associated with $\bT$ as 
\begin{align*}
	W_{\bT} = \sum_{1 \le i, j \le n} \bT_{ij} \cdot \indicator{x \in \U_i} \cdot \indicator{y \in \U_j},
\end{align*}
where $\U_1, \dots, \U_n$ denote the uniform partition of $[0,1]$. 
The following lemma connects the operator norms of $W_{\bT}$ and $\bT$. The proof of the result is given in Appendix \ref{appendix: proof of empirical graphon lemma}. 
\begin{lemma}
\label{lemma: empirical graphon norm}
    Let $W_{\bT}$ be the graphon defined above and $\bbW_{\bT}$ be the associated linear operator. Then
    \begin{align*}
        \Vert \bbW_{\bT} \Vert_{\op} = \frac{1}{n} \Vert \bT \Vert_{\op}. 
    \end{align*}
\end{lemma}

\subsubsection*{Estimation problem}
In this paper we are interested in the  problem of estimating the graphon $W$ from the observation of a  sample network $\bm{A}$ generated from it. While several works addressed this problem before (see, e.g., \citep{chatterjee2012,gao2014rate,klopp2017,klopp2019, xu2017}), we are  interested in estimation procedures that satisfy the following two key properties:
\begin{itemize}
    \item [(P1)] \label{P1} the estimator should be close to the underlying graphon in operator norm, and
    \item [(P2)] \label{P2} the estimator should have low rank. 
\end{itemize}
Before presenting our main results, we provide some motivation of why these two properties are important for applications.

\section{Targeted Interventions in Network Games}
\label{section: motivating applications}
We provide an example of application that motivate our study, Targeted Interventions in Network Games. 
Network games serve as models to elucidate strategic behaviors of agents within networked environments, finding applications across various domains such as public goods \newline \citep{BRAMOULLE_public_goods}, R\&D, crime \citep{bramoulle_interactions}, and others \citep{galeotti_jackson_network_games,francesca_review}. Recently, \citet{galeotti_interventions} addressed the challenge of a central planner intervening in a network game to enhance agents' welfare, a task presenting challenging optimization hurdles, particularly for large networks.
In response, \cite{Parise_Ozdaglar_Econometrica} introduced the concept of graphon games, formalizing the asymptotic behavior of a network game as the population size approaches infinity. They proposed a method to derive interventions for large network games from solutions in the graphon domain, assuming knowledge of the graphon. Our aim is to extend these findings to scenarios where the graphon is unknown, yet the planner can observe a realized network. We start by providing an overview of graphon games and the associated optimization problem. 


\subsection{Graphon games and interventions}

In a graphon game, the interval $[0,1]$ is the set of all players. The players interact over a graphon $W$. We will consider linear-quadratic (LQ) graphon games. For a general framework, we refer the reader to \cite{Parise_Ozdaglar_Econometrica}. In an LQ graphon game, each player $x \in[0,1]$ chooses a strategy $s(x) \in \bbR$. A strategy profile is a measurable function $s: x \mapsto s(x)$ in $L_{2}[0,1]$. For a given strategy profile $s$, the utility of player $x$ is given by 
\begin{align}\label{def:utility}
    U\left(s(x), z_{W}(x \mid s), \theta(x)\right) = -\frac{1}{2} s(x)^2 + s(x) \cdot \left(\gamma z_W(x \mid s) + \theta(x)\right). 
\end{align}
Here, $\gamma > 0$ is the peer-effect parameter, $z_{W}(x \mid s)$ is the local aggregate sensed by player $x$ via the graphon, defined as
\begin{align}
    z_{W}(x \mid s):=\int_{0}^{1} W(x, y) s(y) \mathrm{d} y, 
\end{align}
and $\theta(x)$ is a parameter modeling heterogeneity among the players. 
We denote a graphon game compactly as $G(W, \theta)$. 
A strategy profile $s^{*}$ of a graphon game is a Nash equilibrium if almost all players have no incentive to unilaterally deviate from their strategy, i.e.,
\begin{align}
    U\left(s^{*}(x), z_{W}\left(x \mid s^{*}\right), \theta(x)\right) \geq \sup _{\tilde{s} \in \bbR} U\left(\tilde{s}, z_{W}\left(x \mid s^{*}\right), \theta(x)\right) \quad \text { a.e. in }[0,1]. 
\end{align}
We  refer to a Nash equilibrium of a graphon game as a graphon equilibrium. Let $\mathrm{NE}(G(W, \theta))$ be the set of all graphon equilibria of $G(W, \theta)$. 

{
For the utility in \eqref{def:utility}, for any fixed local aggregate $z\in\mathbb{R}$ the payoff
$U(s(x),z,\theta(x))=-\frac12 s(x)^2 + s(x)\,(\gamma z+\theta(x))$ is strictly concave in $s(x)$, hence each player
has a unique best response. Writing the local aggregate as $z_W(x\mid s)=(\mathbb{W}s)(x)$, the best-response
mapping takes the linear form
\[
\mathrm{BR}(s)(x)=\gamma(\mathbb{W}s)(x)+\theta(x).
\]
Therefore, any graphon equilibrium $s^\star$ must satisfy the fixed-point equation
\[
s^\star=\gamma\mathbb{W}s^\star+\theta,
\qquad\text{equivalently}\qquad
(I-\gamma\mathbb{W})s^\star=\theta.
\]
Under the standard stability condition $\gamma\|\mathbb{W}\|_{\mathrm{op}}<1$ (assumed throughout our intervention
analysis and in Theorem~\ref{theorem: graphon perturbed welfare problem}), the operator $(I-\gamma\mathbb{W})$ is invertible on $L^2[0,1]$
(e.g., by a Neumann series argument), and the equilibrium exists and is unique, with the closed form
\[
s^\star=(I-\gamma\mathbb{W})^{-1}\theta.
\]
We refer to \cite{Parise_Ozdaglar_Econometrica} for a broader discussion of equilibrium in graphon games and for the
LQ equilibrium characterization used in our setting.
}


Next consider a central planner who wants to maximize the average utility of the players  at a graphon equilibrium. For this, the planner adds an intervention $\hat{\theta}(x)$ to the heterogeneity $\theta(x)$ of each player $x$, subject to a budget constraint. Thus, the central planner wants to solve the problem
\begin{align}
    \label{eq: graphon social welfare problem}
    \max _{\hat{\theta} \in L_2[0,1]} &\quad T_{(W,\theta)}(\hat{\theta}):= \int_0^1 U\left(s(x), z_W(x \mid s), \theta(x) + \hat{\theta}(x) \right) dx \\
    \text{s.t.} &\quad s \in \mathrm{NE}\left(G(W,\theta+\hat{\theta})\right), \quad \|\hat{\theta}\|_{L_2}^{2} \leq B. \nonumber
\end{align}
Here, $\|\hat{\theta}\|_{L_2}^{2} \leq B$ is the budget constraint, where $B>0$. 

Solving problem \eqref{eq: graphon social welfare problem} in practice presents two main challenges: i) the planner typically doesn't know the graphon  and ii) problem \eqref{eq: graphon social welfare problem} is, in general, computationally intractable. 
To obviate problem i), we  assume that the planner can observe one network realized from the graphon model and use this information to infer the underlying graphon. The planner could then solve problem \eqref{eq: graphon social welfare problem} for the estimated graphon, instead of the true graphon. The next theorem shows that it is possible to bound the suboptimality of such a solution for the original graphon intervention problem in terms of the graphon estimation error in the operator norm.

\begin{theorem}
\label{theorem: graphon perturbed welfare problem}
	Let $\bbW_1, \bbW_2: L_2[0, 1] \to L_2[0, 1]$ be two different integral operators such that \\$\Vert \bbW_1 \Vert_{\op}, \Vert \bbW_2 \Vert_{\op} < 1 / \gamma$, and let $W_1, W_2$ be their kernels. Assume that  $W_1, W_2$ are symmetric and
	\begin{align*}
		\int_{[0, 1]^2} W_1^2(x, y) dx dy, \int_{[0, 1]^2} W_2^2(x, y) dx dy < + \infty.
	\end{align*}	
	Let $T_1, T_2$ be corresponding target functions of problem~\eqref{eq: graphon social welfare problem} with the same heterogeneity $\theta(\cdot)$ and $\hat \theta_1, \hat \theta_2$ be the respective arguments of the maxima. Then
	\begin{align*}
		T_1(\hat{\theta}_2) \ge T_1(\hat{\theta}_1) - \frac{4 \gamma \Vert \bbW_1 - \bbW_2 \Vert_{\op} (\Vert \theta \Vert_{L_2}^2 + B)}{(1 - \gamma [\Vert \bbW_1 \Vert_{\op} \vee \Vert \bbW_2 \Vert_{\op}])^3}.
	\end{align*}
\end{theorem}
\begin{remark}[Why the operator norm naturally appears for interventions]
{Appendix~\ref{proof_graphon_perturbed_welfare_problem} shows that the welfare objective can be written as a quadratic form of the squared resolvent:
$T(W,\theta)(\hat\theta)=\frac12\langle (I-\gamma W)^{-2}(\theta+\hat\theta),\theta+\hat\theta\rangle$.
In particular, when $\theta\equiv 0$ the optimal value satisfies
\[
\max_{\|\hat\theta\|_{L^2}^2\le B} T(W,0)(\hat\theta)
=\frac{B}{2}\,\|(I-\gamma W)^{-1}\|_{\mathrm{op}}^2
=\frac{B}{2}\,(1-\gamma\lambda_{\max}(W))^{-2},
\]
and the maximizer aligns with the top eigenfunction of $W$.
Thus, even in this simplest instance, solving the intervention problem requires estimating the leading spectral
features of $W$, which are naturally controlled in operator norm. We discuss this connection in more detail in Section~\ref{section: regret lower bounds}.}
\end{remark}

Our result extends \citep[Theorem 3]{Parise_Ozdaglar_Econometrica}, which was established in an asymptotic setting for the special case where \( W_2 \) is the empirical graphon corresponding to a network sampled from \( W_1 \). In contrast, our theorem applies to a more general, non-asymptotic setting, where \( W_1 \) and \( W_2 \) are arbitrary graphons. This generalization requires a new proof.  {
Although our conclusion is in the spirit of existing stability results for graphon interventions,
Theorem~\ref{theorem: graphon perturbed welfare problem} is proved in a different (fully non-asymptotic) operator-perturbation framework.

}

Theorem \ref{theorem: graphon perturbed welfare problem} suggests that to achieve a desirable level of social welfare, the central planner should use an estimator of the underlying unknown graphon with a small error in the operator norm. 
 A natural choice for an estimator is the empirical graphon associated with the observed network \( \mathbf{A} \). However, this approach is not computationally feasible because the complexity of solving problem \eqref{eq: graphon social welfare problem} depends on the rank of \( W_{\mathbf{A}} \), which is typically very high (see, e.g., \citep{vu_rank_adjencency}).\footnote{A graphon \( W \) is said to have rank \( r \) if there exist orthonormal eigenfunctions \( \phi_i \in L_2[0,1] \) and eigenvalues \( \lambda_i \in \mathbb{R} \) for each \( i \in [r] \) such that \( W(x,y) = \sum_{i = 1}^r \lambda_i \phi_i(x) \phi_i(y) \) for all \( (x,y) \in [0,1]^2 \).}  
To address this issue, we leverage the fact that when the graphon has rank \( r \), problem \eqref{eq: graphon social welfare problem} can be reformulated as a semidefinite program (SDP) with only two variables and a linear matrix inequality (LMI) constraint involving a matrix of size \( r+2 \) (see \citep[footnote 14]{Parise_Ozdaglar_Econometrica}). This insight, together with Theorem \ref{theorem: graphon perturbed welfare problem}, highlights the need for a graphon estimator that simultaneously satisfies both properties P1 and P2, as introduced in the key estimation problem in Section~\ref{section: preliminaries}.

Note that   we can apply Theorem \ref{theorem: graphon perturbed welfare problem} to compute targeted interventions in a network \(\bA\) with agent heterogeneity \(\btheta\) by setting \(W_1 = W_{\bA}\) and \(W_2\) as a low-rank approximation of the adjacency matrix given, for example, by the hard thresholding estimator \eqref{eq: HTE definition}. This approach reduces the computational cost of the optimization problem that the central planner must solve to determine the optimal intervention in \(G(\bA, \btheta)\). Consequently, constructing an effective approximation for the targeted intervention problem does not require explicit graphon estimation, as it suffices to compute optimal interventions based on a low-rank approximation of the adjacency matrix. However, the accuracy of such an approximation depends on the smoothness of the underlying graphon.
 In contrast, graphon estimation becomes particularly useful in the following scenario:  
  two networks, \( \mathbf{A} \) and \( \mathbf{A}' \), are sampled from the same underlying graphon \( W \) and  the network \( \mathbf{A} \) is too large to compute the truncated singular value decomposition (SVD).  
 On the other hand, the network \( \mathbf{A}' \) is moderately small, allowing for the construction of a hard-thresholding estimator \( \widehat{W} \) of \( W \) based on \( \mathbf{A}' \).  
 In this setting, we can use the following strategy to estimate the optimal interventions for a large network \( \mathbf{A} \) with agent heterogeneity \( \boldsymbol{\theta} \) using the estimator \( \widehat{W} \): 
define the following  function 
\begin{align*}
     \theta(x) = \sum_{i = 1}^n \btheta_i \cdot \indicator{x \in [(i - 1)/n; i/n)}.
\end{align*}
Solve the problem~\eqref{eq: graphon social welfare problem} for the graphon game $G(\widehat{W},  \theta)$, and let $\widehat \theta(x)$ be the argument of the corresponding maximum and define the vector $\widehat{\btheta}'_i = \widehat \theta(i/n)$.
Then, Theorem~\ref{theorem: graphon perturbed welfare problem} implies that for the targeted intervention problem for the game $G(\bA, \btheta)$ we have
\begin{align*}
    T(\widehat{\btheta}') & \ge T(\widehat{\btheta}) - \frac{ 2 \gamma \cdot (\Vert \theta \Vert_{L_2} + \sqrt{B})^2 }{(1 - \gamma [\Vert \bA \Vert_{\op} / n \vee \Vert \widehat{\bbW} \Vert_{\op}])^5} \cdot (\Vert \bbW_{\bA} - \bbW \Vert_{\op} + \Vert \bbW - \widehat{\bbW} \Vert_{\op}) \\
    & \ge T(\widehat{\btheta}) - \frac{ 2 \gamma \cdot (\Vert \theta \Vert_{L_2} + \sqrt{B})^2 }{(1 - \gamma [\Vert \bA \Vert_{\op} / n \vee \Vert \widehat{\bbW} \Vert_{\op}])^5} \cdot \left (\frac{1}{n} \Vert \bA - \bQ \Vert_{\op} + \Vert \bbW_{\bQ} - \bbW \Vert_{\op} + \Vert \bbW - \widehat{\bbW} \Vert_{\op} \right ),
\end{align*}
where the second inequality holds by the triangle inequality and Lemma~\ref{lemma: empirical graphon norm}.
The term $\Vert \bA - \bQ \Vert$ can be efficiently bounded via Corollary~\ref{corollary: symmetric matrix concentration}. Bounding \( \Vert \mathbb{W}_{\mathbf{Q}} - \mathbb{W} \Vert \) requires certain assumptions on the graphon \( \mathbb{W} \). In what follows, we focus on two specific cases: when \( \mathbb{W} \) is an SBM (stochastic block model) graphon or a smooth graphon (see Section~\ref{section: graphon estimation}).  
 Note that the estimator \( \widehat{\mathbb{W}} \) can be constructed from a separate network \( \mathbf{A}' \) sampled from the same graphon \( \mathbb{W} \). As discussed earlier, \( \widehat{\mathbb{W}} \) must satisfy properties (P1)–(P2) to ensure both the tractability of the estimation procedure and that the difference \( T(\widehat{\boldsymbol{\theta}}) - T(\widehat{\boldsymbol{\theta}}') \) remains small. We will explore such estimators in Section~\ref{section: graphon estimation}.  

\begin{remark} [Computational cost]
 If \( \boldsymbol{\theta} \) is not orthogonal to the eigenspace corresponding to the eigenvalue \( \lambda_1(\mathbf{A}) \), the problem of finding optimal interventions in the network \( \mathbf{A} \) with agent heterogeneity \( \boldsymbol{\theta} \) can be reformulated as finding a value of \( L \) in the interval \( (-\infty, - (1 - \gamma \|\mathbf{A}\|_{\operatorname{op}} / n)^{-2}) \) such that  
 \begin{align}
\label{eq: hat theta equation}
 \widehat{\boldsymbol{\theta}} = - \left[ \mathbf{I} + L (\mathbf{I} - \gamma \mathbf{A} / n)^2 \right]^{-1} \boldsymbol{\theta}
 \end{align}  
 satisfies \( \| \widehat{\boldsymbol{\theta}} \|_2^2 = B \), see Lemma~\ref{lemma: computational aspects}  in Appendix~\ref{section: computational aspects}. 
 If the network \( \mathbf{A} \) is sparse, an approximate solution can be obtained iteratively using the Conjugate Gradient method. This procedure generates a monotone sequence \( L_1, L_2, \dots \) that converges to the optimal solution \( L \) (see \citep{gander1989constrained, reinsch1967smoothing}). However, each iteration of the algorithm may require \( O(n \cdot |E|) \) operations, where \( |E| \) is the number of edges in \( \mathbf{A} \), since solving the linear equation  
 \begin{align*}
\left[ \mathbf{I} + L (\mathbf{I} - \gamma \mathbf{A} / n)^2 \right] \mathbf{x} = \mathbf{y}
\end{align*}  
 via the Conjugate Gradient method has this complexity. We refer the reader to \cite{gander1989constrained} for the full details and the description of the algorithm.  
 The degenerate case where \( \boldsymbol{\theta} \) is orthogonal to the eigenspace of \( \lambda_1(\mathbf{A}) \) requires a separate treatment. The eigenvalue \( \lambda_1(\mathbf{A}) \) and its corresponding eigenspace can be computed efficiently using the Lanczos algorithm, which has a complexity of \( O(T \cdot |E|) \) for \( T \) iterations.  
 In contrast, the proposed SVT estimator has a significantly lower computational cost of \( O(nR) \), which corresponds to the complexity of performing a randomized truncated singular value decomposition (SVD), where \( R \) is the target rank and \( n \) is the network size.  
 \end{remark}
 \begin{remark}
There are several other scenarios where a low-rank estimator of the underlying graphon function is essential. One notable example is the optimal control of networked dynamical systems. In such systems, the objective is to determine a control input that minimizes a total running cost and a terminal cost, both of which depend on the system state and control input.  
 However, for large-scale networks, solving this optimization problem becomes computationally intractable—even when the system dynamics are linear and the cost function is quadratic. To address this challenge, the theory of \textit{graphon dynamical systems} was introduced in \citep{graphon_LQR_gao, low_rank_LQR_gao}. This framework establishes that as the number of nodes in a networked dynamical system grows indefinitely, the system behavior converges to a limiting \textit{graphon system}, enabling more tractable analysis and control.  
\end{remark}

\section{Low-rank graphon estimation}
\label{section: estimation}
 In this section, we introduce a graphon estimator that satisfies Properties (P1) and (P2).  Our approach consists of two main steps. First, we construct a matrix estimator \( \widehat{\mathbf{Q}}(\mathbf{A}) \) for \( \mathbf{Q} \) using singular value thresholding (see Section \ref{section: weight matrix estimation}). Then, we show that the empirical graphon associated with \( \widehat{\mathbf{Q}}(\mathbf{A}) \) possesses the required properties (see Section \ref{section: graphon estimation}).  
The problem of estimation of $\bQ$ from $\bA$ was considered mainly in Frobenius and cut norms, see, for example, ~\citep{chatterjee2012,gao2014rate,klopp2017,klopp2019,xu2017} and references therein. Our  goal here is different as we are interested in the estimation of the matrix $\bQ$ in spectral norm.

\subsection{Hard thresholding estimator}\label{section: weight matrix estimation}
We follow the classical approach relying on spectral properties of the observed matrix $\bA$. For the adjacency matrix $\bA$, consider its singular value decomposition
\begin{align*}
	\bA = \sum_{i = 1}^n \sigma_i \bu_i \bv_i^\top.
\end{align*}
Then,  the hard thresholding estimator is defined as
\begin{align}
\label{eq: HTE definition}
	\bQhat_\lambda = \sum_{i \in [n] : \sigma_i \ge \lambda} \sigma_i \bu_i \bv_i^\top.
\end{align}

We start by proving a general upper bound on the error of $\bQhat_{\lambda}$ measured in spectral norm.

\begin{theorem}
\label{theorem: UHTE performance}
    Consider a weight matrix $\bQ$ and an adjacency matrix $\bA$ sampled from it. Let $\bQhat_\lambda$ be the hard thresholding estimator of $\bQ$ as defined in \eqref{eq: HTE definition}. Then, for any $\delta \in (0,1]$, with probability at least $1 - \delta$, 
    \begin{align*}
        \Vert \bQ - \bQhat_\lambda \Vert_{\op} \le\lambda + 7 \sqrt{n \rho_n} + \sqrt{\constantBandeira \log (n/\delta)}.
    \end{align*}
    In particular, for $\lambda = 9 \sqrt{n \rho_n}$ we have
	\begin{equation}\label{eq:3}
		\Vert \bQ - \bQhat_\lambda \Vert_{\op} \le 16 \sqrt{n \rho_n} + \sqrt{\constantBandeira \log (n/\delta)},
	\end{equation}
	where $\constantBandeira$ is an absolute constant defined in Appendix~\ref{appendix: tools}.
\end{theorem}

\begin{remark}
In~\citep{klopp2019}, the authors  consider the problem of estimating the matrix $\bQ$ in  cut-norm and show that for $\rho_n > 1/n$, there exists a constant $C$ such that
\begin{align*}
	\inf_{\bQhat} \sup_{\bQ \in \text{SBM}(2, \rho_n)} \bbE \Vert \bQ - \bQhat \Vert_{\square} \ge C \sqrt{\frac{\rho_n}{n}},
\end{align*}
where $\text{SBM}(2, \rho_n)$ denotes a Stochastic Block Model (SBM) with two communities.\footnote{
We say that a matrix $\bQ$ belongs to $\text{SBM}(k, \rho_n)$ if there exists a symmetric matrix $\bS \in [0, 1]^{k \times k}$ such that $\max_{k_1, k_2} \bS_{k_1, k_2} = \rho_n$, and a mapping $z: [n] \to [k]$ such that $\bQ_{i j} = \bS_{z(i) z(j)}$}
Since $\Vert \cdot \Vert_{\op} \ge n \Vert \cdot \Vert_{\square}$ (see Appendix \ref{appendix: tools}), this implies that for $\lambda = O(\sqrt{n \rho_n})$ the estimator $\bQhat_\lambda$ is minimax optimal in spectral norm, since
$$\inf_{\bQhat} \sup_{\bQ } \bbE \Vert \bQ - \bQhat \Vert_{\op} \ge 
n \inf_{\bQhat} \sup_{\bQ } \bbE \Vert \bQ - \bQhat \Vert_{\square} \ge
n \inf_{\bQhat} \sup_{\bQ\in \text{SBM}(2, \rho_n) } \bbE \Vert \bQ - \bQhat \Vert_{\square}\ge C\sqrt{\rho_n n}.$$
\end{remark} 

{\begin{remark}[On the $\lambda$-term in Theorem~2]
The additive $\lambda$ term in Theorem~2 reflects the bias induced by hard thresholding and cannot be removed
uniformly: even in a noiseless setting $\mathbf A=\mathbf Q$, if $\mathbf Q$ has a singular value $\sigma\in(0,\lambda)$
then $\widehat{\mathbf Q}_\lambda$ discards the associated component and $\|\mathbf Q-\widehat{\mathbf Q}_\lambda\|_{\mathrm{op}}
\ge \sigma$.
\end{remark}

\begin{remark}[Operator norm versus Frobenius norm]
Most of the graphon/matrix estimation literature measures error in (normalized) Frobenius norm or related
$L^2$-type metrics. In our setting, these norms lead to a very different fluctuation scale than the operator norm.
Indeed, since the entries of $\mathbf A-\mathbf Q$ are independent, centered, and bounded with
$\mathrm{Var}(A_{ij})\le Q_{ij}\le \rho_n$, we have
\[
\mathbb E\|\mathbf A-\mathbf Q\|_F^2
=\sum_{i\neq j}\mathrm{Var}(A_{ij})
\lesssim n^2\rho_n,
\qquad\text{so typically}\qquad
\|\mathbf A-\mathbf Q\|_F = O_{\mathbb P}(n\sqrt{\rho_n}).
\]
By contrast, non-commutative (matrix) concentration yields the sharper spectral-scale bound
\[
\|\mathbf A-\mathbf Q\|_{\mathrm{op}} = O_{\mathbb P}\!\left(\sqrt{n\rho_n}+\sqrt{\log n}\right),
\]
which is the scale captured (up to logarithmic terms) by Theorem~\ref{theorem: UHTE performance}.

This distinction matters for our application: Theorem~\ref{theorem: graphon perturbed welfare problem} controls the difference of intervention objectives and
maximizers through $\|W_1-W_2\|_{\mathrm{op}}$, hence we require operator-norm accuracy. While one always has
$\|\cdot\|_{\mathrm{op}}\le \|\cdot\|_F$, using a Frobenius guarantee to control operator norm would lead to an
$O_{\mathbb P}(n\sqrt{\rho_n})$ bound, which is looser by a factor $\sqrt n$ compared to the intrinsic spectral noise
scale $\sqrt{n\rho_n}$.

The proof techniques also differ accordingly. Frobenius-norm analyses typically exploit ``energy'' control such as
\[
\|\mathbf A-\widehat{\mathbf Q}_\lambda\|_F^2
=\sum_{\sigma_i(\mathbf A)<\lambda}\sigma_i(\mathbf A)^2,
\]
which can be bounded via tail-energy/approximation arguments (and yields improvements when the signal is
approximately low rank).
In contrast, the operator-norm bound in Theorem~\ref{theorem: UHTE performance} relies on two ingredients tailored to worst-direction control:
(i) a matrix concentration inequality for $\|\mathbf A-\mathbf Q\|_{\mathrm{op}}$, and (ii) the basic spectral property of
hard thresholding that $\|\mathbf A-\widehat{\mathbf Q}_\lambda\|_{\mathrm{op}}\le \lambda$ because the discarded singular
values are all below $\lambda$.
\end{remark}

}

\begin{remark} 
    To simplify results, in our subsequent discussion, we fix $\lambda = 9 \sqrt{n \rho_n}$ so that we can use \eqref{eq:3}. All our results  can be generalized for arbitrary $\lambda > 7 \sqrt{n \rho_n} + \sqrt{\constantBandeira \log (n/\delta)}$.
\end{remark}

\subsubsection{Generalization to missing links}
Theorem~\ref{theorem: UHTE performance} can be extended to the case of missing links. In some applications the network of interactions is not fully observed, or the size of networks is so huge that we need to subsample the network. In such cases, we may consider the matrix $\widetilde \bA = \bM \circ \bA$, where $\bM$ is a symmetric mask with independent entries $\bM_{ij}$, each a Bernoulli random variable with parameter $p$.

In this setting, $\widetilde \bA$ can be considered as a sample from the sparse graphon model with sparsity parameter $\rho_n p$. Thus, the singular value thresholding estimator based on the matrix $\widetilde \bA$ satisfies the following theoretical guarantees:
\begin{align*}
	\Vert p \bQ - \bQhat_\lambda(\widetilde{\bA}) \Vert_{\op} \le \lambda + 7 \sqrt{n p \rho_n} + \sqrt{\constantBandeira \log (n/\delta)}
\end{align*}
with probability at least $1 - \delta$, leading to the following corollary.

\begin{corollary}
\label{corollary: missing links corrolary}
    Let $\widetilde{\bA}$ be obtained as above. Then,
    \begin{align*}
        \left \Vert \bQ - \bQhat_\lambda \left ( \frac{1}{p} \widetilde{\bA} \right ) \right  \Vert_{\op} \le \frac{\lambda}{p} + 7 \sqrt{\frac{n \rho_n}{p}} + \frac{1}{p} \sqrt{\constantBandeira \log (n/\delta)}
    \end{align*}
    with probability at least $1 - \delta$.
\end{corollary}

The corollary above generalizes all subsequent results in the context of estimators to the case of missing links.

\subsection{Graphon estimator}
\label{section: graphon estimation}

In this section we use the empirical graphon of the  hard thresholding estimator, $W_{\bQhat_\lambda}$, as an estimator of the graphon $W$. Note that by triangular inequality, and by Lemma \ref{lemma: empirical graphon norm}

\begin{align*}
    \opnorm{\rho_n \bbW - \bbW_{\bQhat_\lambda}} &\le \opnorm{\rho_n \bbW - \bbW_{\bQ}} +  \opnorm{\bbW_{\bQ}- \bbW_{\bQhat_\lambda}} \\
    &\le \opnorm{\rho_n \bbW - \bbW_{\bQ}} + \frac{1}{n} \Vert \bQ - \bQhat_\lambda \Vert_{\op}. 
\end{align*}
We obtained a bound for the second term above in Theorem~\ref{theorem: UHTE performance}.
Hence, to get a bound on the error of our estimator, it is enough to estimate the approximation error $\opnorm{\rho_n W - W_{\bQ}}$, which measures the distance between the true  sparse graphon $\rho_n W$  and its discretized version, the empirical graphon $W_{\bQ}$,
sampled at the unobserved random  points $\xi_1\le \dots\le \xi_n$.

 It is clear that the approximation error depends on the topology of the space of the considered graphons. We will study  two cases usually considered in the literature,
\textit{step function graphons} (see Sect \ref{sec:SBM}) and \textit{smooth graphons} (see Sect \ref{sec:smooth}). For these two cases, we  derive bounds on the estimation error and show that the rank of our estimator is sublinear in $n$. 

\subsubsection{SBM graphons}\label{sec:SBM}
\begin{definition}
\label{definition: SBM}
Given a positive integer $k$, we say that $W$ is a Stochastic Block Model (SBM) graphon with $k$ communities if there is a partition of $[0, 1]$ into intervals 
$P_1=[0,e_1), \dots, P_k=(e_{k-1},1]$, for some $e_1< \hdots<e_{k-1}$, and a symmetric matrix $\bS \in [0, 1]^{k \times k}$ such that
\begin{align*}
	W(x, y) = \sum_{1 \le i, j \le k} \bS_{ij} \indicator{(x, y) \in P_i \times P_j}.
\end{align*}
\end{definition}
Note that networks generated from the SBM graphon defined above correspond to networks generated from  the classic Stochastic Block Model (SBM) \citep{SBM_holland}.
The following result provides a bound on the error and the rank  of the  estimator $W_{\bQhat_\lambda}$ when $W$  is an SBM graphon.
\begin{theorem}
\label{theorem: SBM convergence rate}
    Consider a graphon $W$ that is an SBM graphon with $k$ communities.  Assume that
    $n \ge k \log (4k/\delta)$ and 
    $n \rho_n \ge \constantBandeira \log (2n/\delta)$. Then, for any $\delta \in (0,1]$, with  probability $1 - \delta$, it holds that
    \begin{align*}
         i) &\quad   \opnorm{\rho_n \bbW - \bbW_{\bQhat_\lambda}} \le
         \rho_n \left (\frac{32 k \sum_{a = 1}^{k - 1} e_a}{n} \log \left(\frac{4k}{\delta}\right)\right )^{1/4} + 11 \sqrt{\frac{\rho_n}{n}}, \\
          ii) &\quad  \rank(\bQhat_{\lambda}) \le k.
    \end{align*}
\end{theorem}
The proof of this result is given in Appendix \ref{appendix: proof of SBM convergence rate}. Note that assumption $n \ge k \log (4k/\delta)$ is not restrictive, as otherwise the bound provided by i) is trivial. The next proposition shows that the upper bound given by Theorem \ref{theorem: SBM convergence rate} is optimal in a minimax sense (up to dependency on $k$). 

\begingroup


\begin{theorem}
\label{theorem: SBM lower bound}
For any sequence $\rho_n>0$, we have
\begin{align*}
\underset{\widehat{W}}{\inf}
\underset{W\in \mathcal W[2]}{\sup}
    {\P}_{W} \left ( \opnorm{\widehat \bbW(\bA) - \rho_n \bbW } \ge \rho_n n^{-1/4} / 16 \right ) \ge \frac{1}{4},
\end{align*}
where $\P_{W}$ denotes the probability distribution of sampled adjacency matrices $\bA$ when the underlying graphon is $\rho_n W$, $\inf_{\widehat{W}} $ is the infimum over all estimators and $\mathcal W[2]$ is the class of SBM graphons with two communities.
\end{theorem}

The proof of this result can be found in Appendix \ref{appendix: proof of SBM lower bound}.

\subsubsection{Smooth graphons}\label{sec:smooth}

\begin{definition}
\label{definition: holder graphon}
Given $\alpha>0$, $H>0$, we say that $W : [0,1]^2 \rightarrow [0,1]$ is an $(H,\alpha)$-H\"{o}lder continuous graphon if
$$
\vert W(x',y')-P_{\lfloor \alpha\rfloor}((x,y), (x'-x,y'-y))   \vert \le H (\vert x'-x\vert^{\alpha-\lfloor \alpha\rfloor}+\vert y'-y\vert^{\alpha-\lfloor \alpha\rfloor})
$$
for all $(x',y'), (x,y) \in [0,1]^2$, where $\lfloor \alpha\rfloor$ is the maximal integer less than $\alpha$, and the function $(x',y')\mapsto P_{\lfloor \alpha\rfloor}((x,y), (x'-x,y'-y))$ is the Taylor expansion of $W$ of degree $\lfloor \alpha\rfloor$ at point $(x,y)$. 
\end{definition}
We denote the class of all $(H,\alpha)$-H\"{o}lder continuous graphons by $\Sigma(\alpha,H)$. 
\begin{remark}
It follows from the standard embedding theorems that, for any $W\in \Sigma(\alpha, L)$ and any $(x',y'), (x,y) \in [0,1]^2$,
\begin{equation}\label{lip}
\vert W(x',y') - W(x,y)\vert \le M (\vert x'-x\vert^{\alpha\wedge 1}+\vert y'-y\vert^{\alpha\wedge 1})
\end{equation}
where $M$ is a constant depending only on $\alpha$ and $H$. In the following, we will use only this last property of $W\in \Sigma(\alpha, L)$. For example, an $L$-Lipschitz function is $(L, 1)$-H\"older continuous.     
\end{remark}

{
In our setting, when labels of an observed graph are defined by the order of latent variables $\xi_1, \ldots, \xi_n$, smoothness of a graphon $W(\cdot, \cdot)$ is naturally inherited by the matrix $\bQ$ (see definition~\eqref{eq: Q sampling}) since $\xi_{(i)}, i \in [n],$ concentrates around $i/n$. Instead of taking truncated SVD of $\bA$ as an estimator of $\bQ$, we adapt a different strategy. Define an averaging operator $\cS_m(\cdot)$ acting on $n\times n$ matrices and parameterized by an integer $m$ as follows. Let $n = m s + \ell$. We set 
\begin{align*}
    \cB_i = \begin{cases}
        \{i' \mid \lfloor (i' - 1)/m \rfloor = \lfloor (i - 1)/ m \rfloor\}, & \text{if } i \le m (s - 1), \\
        \{m(s - 1) + 1, \ldots, m s + \ell \} & \text{if } i > m(s - 1).
    \end{cases}
\end{align*}
For a matrix $\bM \in \bbR^{n \times n}$ define $\cS_m (\bM)$ by
\begin{align*}
    \cS_m(\bM)_{ij} = \frac{1}{|\cB_i||\cB_j|} \sum_{i' \in \cB_i, j' \in \cB_j} \bM_{i' j'}.
\end{align*}
For brevity, we set $\cB(i, j) = \cB_i \times \cB_j $. Then, the following theorem holds, which proof can be found in Appendix~\ref{section: upper bound proof - smooth graphons}.
\begin{theorem}
\label{theorem: Holder averaging}
    Let $W(\cdot, \cdot)$ be an $(H, \alpha)$-H\"older graphon and $\rho_n$ be the sparsity parameter. Set $\beta = \alpha \wedge 1$. Suppose that
    \begin{align}
         n^2 \rho_n  \ge \log^{4 \beta + 1} \left (\frac{9 n}{\delta} \right ) & \quad \text{and} \quad (n^2 \rho_n)^{\beta / (2 \beta + 1)} \ge \log^{3/2} \left (\frac{9n}{\delta} \right ) \label{eq: assumption sqrt-rho}\\
        & n^{2 \beta - 1}  > \rho_n \label{eq: m is larger 1 assumption}
    \end{align}
    Set $m = \left \lceil n^{(2 \beta - 1)/(2 \beta + 1)} \rho_n^{-1/(2 \beta + 1)} \right \rceil.$ 
    Then, there are constants $C_ 1= C_1(H, \alpha)$ and $C_ 2= C_2(H, \alpha)$ such that
    \begin{align*}
        \Vert\cS_m(\bA) - \bQ \Vert_{\op} \le  C_1 \rho_n n \cdot \left ( \frac{\log(4/\delta)}{n} \right )^{\beta/2}  + C_2\rho_n^{\frac{\beta + 1}{2 \beta + 1}} n^{1/(2 \beta + 1) 
    }
    \end{align*}
    holds with probability at least $1 - \delta$.
\end{theorem}
Let us comment conditions of the theorem. Condition~\eqref{eq: m is larger 1 assumption} ensures that $m \ge 2$. Importantly, if this condition does not hold, then 
\begin{align*}
    \rho_n^{\frac{\beta + 1}{2 \beta + 1}} n^{1/(2 \beta + 1)} \ge \sqrt{n \rho_n},
\end{align*}
so Theorem~\ref{theorem: UHTE performance} should be applied in this case.  Note that condition~\eqref{eq: assumption sqrt-rho} is not restricting as it is much weaker than a usual assumption $n \rho_n \gtrsim 1$ that we use below.

In what follows, we fix $m = m(\beta)$ as in the statement of Theorem~\ref{theorem: Holder averaging}. We define a low-rank estimator of $\bQ$ using the SVD decomposition of $\cS_m(\bA) = \sum_{i = 1}^n \sigma_i' \tilde \bu_i \tilde \bv_i^\top$ as follows:
\begin{align*}
    \widehat{\bQ}_{\lambda', m} = \sum_{i\in [n] \mid \sigma_i \ge \lambda'} \sigma_i' \tilde \bu_i \tilde \bv_i^\top, \quad \text{where} \quad \lambda' = 2C_1 \rho_n n \cdot \left ( \frac{\log(8/\delta)}{n} \right )^{\beta/2}  + 2C_2\rho_n^{\frac{\beta + 1}{2 \beta + 1}} n^{1/(2 \beta + 1)},
\end{align*}
where $C_1$ and $C_2$ are the constants from Theorem~\ref{theorem: Holder averaging}. Then, set
\begin{align*}
    \widehat{\bQ}^*_{\beta} = \begin{cases}
        \widehat{\bQ}_{\lambda', m}, & \text{ if } n^{2 \beta - 1} > \rho_n, \\
        \widehat{\bQ}_{\lambda}, & \text{ otherwise.}
    \end{cases}
\end{align*}
}

Next we prove an upper bound on the error and the rank of our estimator in the case of  $(H, \alpha)$-H\"older continuous graphons.

\begin{theorem}
\label{theorem: holder convergence rate}
Consider a graphon $W\in \Sigma (\alpha, H)$. Let $\beta = \min(\alpha, 1)$. Assume that 
    condition~\eqref{eq: assumption sqrt-rho} holds with $\delta/4$ in place of $\delta$ and $n \rho_n \ge \constantBandeira \log (2 n/\delta)$. Then, there exist constants $C$ and $C'$ depending on $\alpha, H$ such that for any $\delta \in (0,1]$, with probability at least $1 - \delta$, it holds that
	\begin{align*}
		i)&\ \opnorm{\rho_n \bbW - \bbW_{\bQhat^*_{\beta}}} \le \rho_n  C \cdot \left ( \frac{\log (8 /\delta)}{n} \right )^{\beta/2} + C \min \left \{\sqrt{\frac{\rho_n}{n}}, \; \left ( \frac{\rho_n^{\beta + 1}}{n^{2 \beta}} \right )^{1/(2 \beta + 1)} \right \}, \\
		ii)&\  \rank (\bQhat_{\beta}^*)  \le C' \cdot \begin{cases}
            \min \Big \{\frac{n^{\beta/(2 \alpha + 1)}}{\log^{\beta/(2 \alpha + 1)}(8/\delta)}, \; (n^2 \rho_n)^{2\beta (2 \beta + 1)^{-1} / (2\alpha + 1)} \Big \}, & \text{ if } n^{2 \beta - 1} > \rho_n, \\
            (n \rho_n)^{1/(2 \alpha + 1)}, & \text{ otherwise.}
        \end{cases}. 
	\end{align*}
\end{theorem}

The proof of this result can be found in Appendix \ref{appendix: proof of Holder convergence rate}.  

{
To see where the rate comes from, note that the block-averaging operator $S_m(\cdot)$ approximates an $(H,\alpha)$-H\"older graphon by an $s\times s$ step function (with $s\simeq n/m$), incurring an approximation error of order $\rho_n H s^{-\alpha}$, while the corresponding block-averaged noise contributes a spectral fluctuation of order $\sqrt{\rho_n}\,\sqrt{s}/n$ (up to logarithmic factors) after rescaling to the graphon operator norm.
Balancing these two terms yields $s\asymp (H^2 n^2\rho_n)^{1/(1+2\alpha)}$ and hence the rate $\big(H\rho_n^{1+\alpha}/n^{2\alpha}\big)^{1/(1+2\alpha)}$.
The matching minimax lower bound follows from a standard $s\times s$ bump-function packing of H\"older graphons together with a $\chi^2$ Fano-type testing argument: when $s^{1+2\alpha}\lesssim H^2 n^2\rho_n$, the hypotheses are separated in operator norm by $\asymp \rho_n H s^{-\alpha}$ yet remain statistically indistinguishable, forcing any estimator to incur error of this order, as the following result shows.
\begin{theorem}
\label{theorem: holder graphon estimation lower bound}
    Let $\Sigma(H, \alpha)$ be the class of $(H, \alpha)$-H\"older graphons for $\alpha \in (0; 1]$. If $H^2 n^2 \rho_n \ge 2^{20}$, $n \rho_n \ge H^{1/(2 \alpha)} / 6$ and $n \ge (1 + H^{1/\alpha})^2$, then
    \begin{align*}
        & \inf_{\widehat{\bbW}} \sup_{W \in \Sigma(H, \alpha)} \bbP_{W} \left (   \Vert \widehat{\bbW} - \rho_n \bbW \Vert_{\op} \ge \frac{ H \rho_n}{2^{14} n^{\alpha/2}} + \frac{1}{60} \left ( \frac{H \rho_n^{1 + \alpha}}{n^{2 \alpha}} \right )^{\frac{1}{(1 + 2 \alpha)}} \right ) \ge \frac{1}{6}.
    \end{align*}
\end{theorem}

The proof of this result can be found in Appendix~\ref{appendix: proof of holder graphon estimation lower bound}.
}

An important property of H\"older graphons is that they can be well approximated by piecewise polynomial functions of degree at most $\lfloor \alpha \rfloor $. More generally, we can consider a class of function that can be approximated  by polynomials with arbitrary precision. 

\begin{definition}
\label{condition: M-r analiticity}
    Given $M > 0$, $r > 0$,  we say that $W : [0,1]^2 \rightarrow [0,1]$ is $(M,r)$-analytic if there exist functions $c_k : [0, 1]^2 \to [-Mr^{-k}, Mr^{-k}]$ for each $k \in \bbN$ such that
    \begin{enumerate}
        \item $	W(x, y) = \sum_{k = 0}^{\infty} c_k(x_0, y) (x - x_0)^k $ for any $x_0, y$ and $x, |x - x_0| < r$, 
        \item $	W(x, y) = \sum_{k = 0}^{\infty} c_k(x, y_0) (y - y_0)^k $ for any $y_0, x$ and $y, |y - y_0| < r$.
    \end{enumerate}

\end{definition}

Roughly speaking, a graphon $W$ is $(M,r)$-analytic if it is analytic on $x$ uniformly with respect to $y$, its Taylor series has radius of convergence $r$, and the coefficients are bounded by $M$. Since the functions $c_k$ are bounded, $W$ satisfies the $(H, 1)$-H\"older condition for some $H$.\footnote{This is shown in Appendix \ref{appendix: proof of analytic convergence rate}.} Hence, we estimate $\bQ$ using the estimator $\widehat{\bQ}^*_{\beta}$ for $\beta = 1$. Using Theorem \ref{theorem: holder convergence rate}, we derive the following result. Its proof can be found in Appendix \ref{appendix: proof of analytic convergence rate}. 

\begin{theorem}
\label{theorem: analityc convergence rate}
Consider a graphon $W$ that is $(M,r)$-analytic. Assume that $n \rho_n \ge \constantBandeira \log (2 n/\delta)$. Let $\widehat{\bQ}_{1}^*$ be the estimator $\widehat{\bQ}^*_{\beta}$ for $\beta = 1$. Then, there are constants $C$ and $C'$ depending on $M$, $r$ such that for any $\delta \in (0,1]$, with probability at least $1 - \delta$, it holds that
    \begin{align*}
        i)&\ \opnorm{\rho_n \bbW - \bbW_{\widehat{\bQ}^*_{1}}}  \le \rho_n  C \cdot \left ( \frac{\log \frac{8}{\delta}}{n} \right )^{1/2} + C \left (\frac{\rho_n}{n} \right )^{2/3}, \\
        ii)&\ \rank (\widehat{\bQ}^*_{1})  \le C' \left (\log (n \rho_n) + \log \log \frac{8}{\delta} \right ).
    \end{align*}
\end{theorem}

\begin{remark}
    Note that our estimator $\bQhat^*_\beta$ is not guaranteed to be symmetric or to take values in $[0,1]$. We can symmetrize it taking $\widetilde {\bQ}^*_\beta=(\bQhat^*_\beta+(\bQhat^*_\beta)^T)/2$. Then, $\rank (\widetilde {\bQ}^*_\beta)\leq 2 \rank (\bQhat^*_\beta)$ and, if $W$ is symmetric, using $W(x,y)=(W(x,y)+W(y,x))/2$, we get 
    $\opnorm{\rho_n W - W_{\widetilde {\bQ}^*_\beta}}\leq \opnorm{\rho_n W - W_{\bQhat^*_\beta}}$.
\end{remark}

{
We would like to highlight that the rates of convergence provided by Theorem~\ref{theorem: analityc convergence rate} are essentially optimal if $\rho_n \gtrsim n^{-1/2}$, as the following result shows.
\begin{theorem}
\label{theorem: analytic graphon estimation lower bound}
For a number $t \in [0; 1/2]$, define a function $\tau_t(x) = x + t x(1 - x)$. Consider graphons $W^{(0)}(x, y)  = (1 + xy)/2$
and $W^{(1)}(x, y) = W^{(0)}(\tau_t(x), \tau_t(y))$ for $t = (16 n)^{-1/2}$. Then, for any estimator $\widehat{\bbW}$, we have
\begin{align*}
    \max_{i \in \{0, 1\}} \bbP_{\bA \sim \rho_n W^{(i)}} \left ( \opnorm{\widehat{\bbW}(\bA) - \rho_n \bbW^{(i)}} \ge \frac{\rho_n}{56} \sqrt{\frac{1}{n}} \right ) \ge \frac{1}{6}.
\end{align*}
\end{theorem}

Clearly, both graphons $W^{(0)}$ and $W^{(1)}$ are $(M, r)$-analytic for some constants $M$ and $r$. The proof of this result can be found in Appendix~\ref{section: analytic graphon lower bound proof}.
}

Finally, note that
    Theorems \ref{theorem: holder convergence rate} and \ref{theorem: analityc convergence rate} can be easily extended to the case when $W$ is piecewise H\"older and piecewise analytic respectively.

\section{Back to Targeted Interventions}\label{section:simulations}
 In this section, we integrate our results to demonstrate how optimal interventions in the network \( \mathbf{A} \) can be estimated using a smaller network \( \mathbf{A}' \) sampled from the same underlying graphon. Notably, \( \mathbf{A}' \) can potentially be obtained by subsampling \( \mathbf{A} \).  
By combining the stability analysis from Theorem~\ref{theorem: graphon perturbed welfare problem} with the convergence results from Theorems~\ref{theorem: SBM convergence rate}, \ref{theorem: holder convergence rate}, and \ref{theorem: analityc convergence rate} , we obtain the following corollary. The proof is provided in Appendix~\ref{section: convergence rate targeted interventions}.
\begin{corollary}
\label{corollary: convergene rate targeted interventions}
   Fix a positive integer $N$. Consider an LQ network game $G(\bA,\btheta)$ with the budget $B \cdot N$ and the peer-effect parameter $\gamma$, $\gamma \Vert \bA \Vert_{\op} < N$, where $\bA$ is a network of size $N$ sampled from a sparse graphon $\rho_N W$.
    Let $T_{(\bA,\btheta)}(\cdot)$ be the objective function of the social welfare problem \eqref{eq: graphon social welfare problem} associated with $G(\bA,\btheta)$. Fix $n$, and let a network $\bA'$ of size $n$ be sampled from the graphon $\rho_n W$ with the matrix of connection probabilities $\bQ'$. Let $\widehat{\bQ}'$ be a symmetric estimator of the matrix of connection probabilities $\bQ'$, and let $\widehat{W} = W_{\widehat{\bQ}'}$ be an estimator of $\rho_n W$ based on $\bA'$.
    
    Assume that $\gamma \Vert \bA' \Vert_{\op} < \rho_N / \rho_n \cdot n$ and $\gamma \Vert \widehat{\bQ}' \Vert_{\op} \le \rho_N / \rho_n \cdot n$.
    Let $\hat{t}' \in \operatorname{Im} \bbW_{\widehat{\bQ}'} + \operatorname{span}(\tilde{\theta})$ be the optimal solution\footnote{One can assume that $\hat{t}' \in \operatorname{Im} \bbW_{\widehat{\bQ}'} + \operatorname{span}(\tilde{\theta})$, since $\bbW_{\widehat{\bQ}'}$ is self-adjoint due to symmetry of $\widehat{\bQ}'$, and so $\operatorname{Im} \bbW_{\widehat{\bQ}'} = \operatorname{Ker}^{\perp} \bbW_{\widehat{\bQ}'}$} of the social welfare problem associated with graphon game $G(\rho_N / \rho_n \cdot \widehat{W},\tilde \theta)$ with
    \begin{align*}
        \tilde \theta(x) = \sum_{i = 1}^N \btheta_i \cdot \indicator{x \in [(i-1)/N, i/N)}
    \end{align*}
    and budget $B$.
    Define a vector $\widehat{\btheta}' \in \bbR^N$ as
    \begin{align*} 
        \widehat{\btheta}'_i = N \cdot \int_{(i - 1)/N}^{i/N} \hat{t}'(x) dx.
    \end{align*}
    Let $\widehat{\btheta}$ be the optimal solution of the network game $G(\bA,\btheta)$. Define 
    \begin{align*}
        \ttb & = \sqrt{\frac{2B}{N}} \cdot \frac{(\Vert \tilde{\theta} \Vert_{L_2} + \sqrt{B})}{(1 - \gamma \Vert \bA \Vert_{\op} / N)^2},\\
        \tau & = 4 \gamma \cdot (1 - \gamma [\Vert \bbW_{\bA} \Vert_{\op} \vee \Vert \rho_N / \rho_n \bbW_{\widehat{\bQ}'} \Vert_{\op}])^{-3} 
    \end{align*}
    and fix $\delta \in (0, 1)$. Suppose that $\rho_N N \ge \constantBandeira \log \frac{6 N }{\delta}$. Then, with probability at least $1 - \delta$, 
    \begin{enumerate}
        \item if $W$ is an SBM with $k$ communities and $\widehat{\bQ}' = (\widehat{\bQ}_{\lambda} + \widehat{\bQ}_\lambda^T)/2$ for the hard thresholding estimator applied to $\bA'$, we have
        \begin{align*}
             T_{(\bA, \btheta)}(\hat{\btheta}) - T_{(\bA, \btheta)}(\hat{\btheta}') \le \ttb + 24\tau \cdot \left (\sqrt{\frac{\rho_N}{n \rho_n}}+  \rho_N \left ( \frac{k \sum_{a = 1}^{k - 1}e_a}{n} \log \left ( \frac{12k}{\delta} \right) \right )^{1/4} \right )
        \end{align*}
        and $\rank(\bbW_{\widehat{\bQ}'}) \le 2k$, assuming conditions of Theorem~\ref{theorem: SBM convergence rate} for $\bA'$ and $\delta/3$ in place of $\delta$;
        \item if $W$ is an $(H, \alpha)$-H\"older graphon and $\widehat{\bQ}' = (\widehat{\bQ}^*_\beta + (\widehat{\bQ}^*_{\beta})^T)/2$ for the estimator $\widehat{\bQ}^*_\beta$ of ${\bQ}'$ based on $\bA'$, we have
        \begin{align*}
            & T_{(\bA, \btheta)}(\hat{\btheta}) - T_{(\bA, \btheta)}(\hat{\btheta}') \le \ttb  \\
    & \qquad \qquad + \tau \cdot \left ( 8 \sqrt{\frac{\rho_N}{N}} + C \rho_N \left ( \frac{\log (24 /\delta)}{n} \right )^{\beta/2} + C \min \left \{\sqrt{\frac{\rho_N}{n \rho_n}}, \; \frac{\rho_N}{\rho_n}\left ( \frac{\rho_n^{\beta + 1}}{n^{2 \beta}} \right )^{1/(2 \beta + 1)} \right \} \right )
        \end{align*}
        and
        \begin{align*}
            \rank(\widehat{\bbW}_{\widehat{\bQ}'}) \le C' \cdot \begin{cases}
            \min \Big \{\frac{n^{\beta/(2 \alpha + 1)}}{\log^{\beta/(2 \alpha + 1)}(24/\delta)}, \; (n^2 \rho_n)^{2\beta (2 \beta + 1)^{-1} / (2\alpha + 1)} \Big \}, & \text{ if } n^{2 \beta - 1} > \rho_n, \\
            (n \rho_n)^{1/(2 \alpha + 1)}, & \text{ otherwise.}
        \end{cases},
        \end{align*}
    provided the conditions of Theorem~\ref{theorem: holder convergence rate} hold for $\bA'$ and $\delta/3$ in place of $\delta$;
    \item if $W$ is $(M, r)$-analytic and $\widehat{\bQ}' = (\widehat{\bQ}^*_1 + (\widehat{\bQ}^*_1)^T)/2$ for the estimator $\widehat{\bQ}_1^*$ of $\bQ'$ based on $\bA'$, we have
    \begin{align*}
        T_{(\bA, \btheta)}(\hat{\btheta}) - T_{(\bA, \btheta)}(\hat{\btheta}') \le \ttb  + \tau \cdot \left ( 8 \sqrt{\frac{\rho_N}{N}} + C \rho_N \left ( \frac{\log (24 /\delta)}{n} \right )^{1/2} + C \frac{\rho_N}{n^{2/3} \rho_n^{1/3}} \right ),
    \end{align*}
    and $\rank(\widehat{\bbW}_{\widehat{\bQ}'}) \le C_2(M, r) (\log (n \rho_n) + \log \log \frac{24}{\delta})$, provided the conditions of Theorem~\ref{theorem: analityc convergence rate} hold for $\bA'$ and $\delta/3$ in place of $\delta$.
    \end{enumerate}

 \end{corollary}

 The quantity $\ttb$ in the statement of Corollary~\ref{corollary: convergene rate targeted interventions} is due to the fact that $N$ may not be divisible by $n$. Otherwise, this term disappears. 

 {
 \subsection{Lower bound for targeted interventions}
 \label{section: regret lower bounds}

While targeted interventions problem for LQR games naturally motivates us to study graphon estimation in the spectral norm, it is an interesting problem how can one  maximize the social welfare $T_{(\rho_n W, \theta)}(\cdot)$ of an unobserved sparse graphon $\rho_n W(\cdot, \cdot)$ in general based on an observed network $\bA$. In other words, setting
\begin{align*}
    \theta^* \in \argmax_{\tilde{\theta} \in L_2[0,1], \|\tilde{\theta}\|_{L_2}^2 \le B} T_{(\rho_n W,\theta)}(\tilde{\theta}), 
\end{align*}
we are interested in minimizing the regret of an intervention $\tilde{\theta}$ defined as
\begin{align*}
    \mathrm{Regret}(\tilde{\theta}) = T_{(\rho_n W,\theta)}(\theta^*) - T_{(\rho_n W,\theta)}(\tilde{\theta}).
\end{align*}
Our paper suggests constructing an intervention $\widehat{\theta}$ based on an estimator $W_{\widehat{\bQ}}$ for some estimator $\widehat{\bQ}$ of the matrix $\bQ$ of connection probabilities, where $W_{\widehat{\bQ}}$ aims to estimate the sparse graphon $\rho_n W$ in the spectral norm. The upper bound on the regret that follows from Theorem~\ref{theorem: graphon perturbed welfare problem} is of the form
\begin{align*}
    \regret(\widehat{\theta}) \le C \cdot  \Vert \rho_n \bbW - \bbW_{\widehat{\bQ}} \Vert_{\op}
\end{align*}
for some constant $C$ depending on $\bbW, \rho_n, B$ and $\theta(x)$. 

In a special case when $\theta(x) \equiv 0$ and $\bbW$ has a spectral gap, this upper bound can be improved through the Davis--Kahan theorem, which leads to the following result.

\begin{lemma}
\label{lemma: regret upper bound spectral gap}
Suppose that $\theta(x) \equiv 0$ and $\bbW_1$ has a spectral gap $\ttg = \lambda_1(\bbW_1) - \lambda_2(\bbW_1) > 0$. Let $T_1(\cdot)$ be the welfare function associated to the graphon $ W_1$, initial heterogeneities $\theta$ and budget $B$, and let $\hat{\theta}_1$ be its maximizer. Let $\hat{\theta}_2$ be the optimal solution of the social welfare problem associated with graphon game $G( W_{2}, \theta)$ with the same initial heterogeneity parameter and budget $B$ for some other graphon $W_2$. Then, we have
\begin{align*}
    T_1(\hat{\theta}_2) \ge T_1(\hat{\theta}_1) - \frac{8 \gamma \sqrt{B} (\Vert \theta \Vert_{L_2} + 2 \sqrt{B})}{(1 - \gamma [\Vert \bbW_1 \Vert_{\op} \vee \Vert \bbW_  2 \Vert_{\op}])^3} \cdot \frac{\Vert \bbW_1 - \bbW_2 \Vert_{\op}^2}{\ttg}.
\end{align*}
\end{lemma}

Note that the above bound is better than one provided by Theorem~\ref{theorem: graphon perturbed welfare problem} only if $\ttg \gtrsim \Vert \bbW_1 - \bbW_2 \Vert_{\op}$. In what follows, we refer to the latter inequality as the spectral gap assumption. Since the spectral gap of a sparse graphon $\rho_n W$ scales linearly in $\rho_n$, the spectral gap assumption implies the following upper bound on the regret of $\widehat{\theta}$:
\begin{align}
    \regret(\widehat{\theta}) \le C' \cdot  \frac{\Vert \rho_n \bbW - \bbW_{\widehat{\bQ}} \Vert_{\op}^2}{\rho_n}, \label{eq: regret squared bound using spectral gap}
\end{align}
for which we are able to find a matching minimax lower bound, see Theorem~\ref{theorem: spectral norms deficit} below.

It is interesting to generalize this result beyond the case $\theta(x) \equiv 0$. It would be enough to derive a Davis--Kahan type result for the generalized eigenvalue problem of the form (see~\citep{gander1989constrained} and Appendix~\ref{section: computational aspects}):
\begin{align*}
    &\text{Find maximal } \lambda \in \bbR \text{ and a corresponding function $u(\cdot)$ such that } \\
    & \qquad  \bbA[u](x) = \lambda u(x) + z(x), \\
    & \qquad \Vert u \Vert_{L_2}^2 = B, 
\end{align*}
when $z(\cdot)$ and bounded self-adjoint operator $\bbA$ are perturbed. Unfortunately, we are not aware of any result of this type in the literature.

We will construct information-theoretic lower bounds for the regret using the following theorem, which relates regret and the deficit of a triangle inequality for some spectral norms. The proof of the theorem is given in Appendix~\ref{appendix: proof of deficit lower bound theorem}.

\begin{theorem}
\label{theorem: spectral norms deficit}
Suppose that for graphons $W_0, W_1$ there exists a number $s > 0$ such that we have
\begin{align}
\label{eq: deficit in resolvents}
    \Vert (\bI - \gamma \rho_n \bbW_0)^{-2} \Vert_{\op} +  \Vert (\bI - \gamma \rho_n \bbW_1)^{-2} \Vert_{\op} \ge s + \Vert (\bI - \gamma \rho_n \bbW_0)^{-2} + (\bI - \gamma \rho_n \bbW_1)^{-2}\Vert_{\op}.
\end{align}
Let $\regret_j(\cdot)$ be the regret in the graphon social welfare problem~\eqref{eq: graphon social welfare problem} with initial heterogeneity parameter $\theta(x) \equiv 0$, budget $B$ and sparse graphon $\rho_n W_j$, $j = 0, 1$. Let $\bbP_j$ be the distribution of $\bA$ generated by the sparse graphon $\rho_n W_j$, $j = 0, 1$. If $\KL{\bbP_1}{\bbP_0} \le \alpha$ 
for some $\alpha > 0$, then
\begin{align*}
    \inf_{\widehat{\theta}} \max_{j = 0, 1} \bbP_j \left ( \regret_j(\widehat{\theta}) \ge B s/2 \right ) \ge \max \left \{ e^{-\alpha}/4, \frac{1 - \sqrt{\alpha}}{2} \right  \}.
\end{align*}
\end{theorem}

 However, it may be hard to deal with condition~\eqref{eq: deficit in resolvents}. Instead, one may use the following simple proposition, which proof is given in Section~\ref{section: graphons deficit power lower bound proof}.

\begin{prop}
\label{proposition: graphons power upper bound}
Suppose that symmetric graphons $W_0, W_1:[0, 1] \to [0, 1]^2$ satisfies
\begin{align*}
    s' = (k_0 + 1) \left ( \lambda_{\max}(\bbW_0)^{k_0} + \lambda_{\max}(\bbW_1 )^{k_0}  - \Vert \bbW_0^{k_0} + \bbW_1^{k_0} \Vert_{\op}\right )
\end{align*}
for some $s' > 0$ and $k_0$,
and $\rho_n \gamma \max\{\Vert \bbW_0 \Vert_{\op}, \Vert \bbW_1 \Vert_{\op} \} < 1$. Then condition~\eqref{eq: deficit in resolvents} holds for $s = (\gamma \rho_n)^{k_0} s'$.
\end{prop}

To see, how Proposition~\ref{proposition: graphons power upper bound} can be applied, suppose that $\bbW_0$ has simple leading positive eigenvalue and $\bbW_0, \bbW_1$ are self-adjoint with finite rank. Then, if $\bbW_0$, $\bbW_1$ have finite rank, we can identify them with finite-dimensional matrices $\ttW_1, \ttW_0$. Setting $\Delta = \ttW_1 - \ttW_0$ and denoting the largest eigenvalue of $\ttW_0$ by $\lambda_0$ and its corresponding eigenvector by $\bu_0$, we get for $k_0 = 1$ in Proposition~\ref{proposition: graphons power upper bound} that
\begin{align*}
    \lambda_{\max}(\bbW_0) + \lambda_{\max}(\bbW_1 )  - \Vert \bbW_0+ \bbW_1 \Vert_{\op} & = 
    \Vert \ttW_0 + \Delta \Vert_{\op} + \Vert \ttW_0 \Vert_{\op} - \Vert 2 \ttW_0 + \Delta \Vert_{\op} \\
    & = \frac{1}{2} \bu_0^\top \Delta (\lambda_0 \bI - \ttW_0)^{+} \Delta \bu_0 + o(\Vert \Delta \Vert_{\op}^2),
\end{align*}
as $\Vert \Delta \Vert_{\op} \to 0$, where $(\cdot)^+$ stands for the Moore--Penrose inverse (see Theorems 7 and 10 \newline from \citep{magnus2019matrix} for the first and second matrix derivatives of eigenvalues). This quadratic behavior is precisely what we need to match the upper bound~\eqref{eq: regret squared bound using spectral gap}! We illustrate this approach for the case of smooth graphons.

\begin{theorem}
\label{theorem: regret lower bound smooth graphons}
Let $W^{(0)}$ and $W^{(1)}$ be graphons from Theorem~\ref{theorem: analytic graphon estimation lower bound}. Let $\regret_i(\cdot)$ be the regret associated with the social welfare problem with budget $B$, initial heterogeneities $\theta(x) \equiv 0$ and graphon $W^{(i)}$. Then there exists an absolute constant $C$ such that for any $n \ge C$ we have
\begin{align*}
     \inf_{\widehat{\theta}} \max_{i = 0, 1} \bbP_{\bA \sim \rho_n W^{(i)}} \left ( \regret_i(\widehat{\theta}) \ge \frac{B\rho_n}{4 \cdot 10^4 \cdot n} \right ) \ge \frac{1}{6}
\end{align*}
and both $W^{(0)}$ and $W^{(1)}$ have the spectral gap at least $\sqrt{13}/6 - 1/3$.
\end{theorem}

Due to Theorem~\ref{theorem: analityc convergence rate} and upper bound~\eqref{eq: regret squared bound using spectral gap}, this lower bound is tight for initial heterogeneities $\theta(x) \equiv 0$ under the spectral gap assumption, provided $\rho_n \gtrsim n^{-1/2}$. However, we postpone tightening the lower bound on the regret in the absence of the spectral gap for future work.
 }

 \subsection{Numerical experiments}

 We start by illustrating our theory by the following example. Consider a graphon $W_1$ defined as follows:
 \begin{align*}
     W_1(x, y) = \sqrt{|x - y|}.
 \end{align*}
 This graphon is $(1, 1/2)$-H\"older, since for any $x_1, x_2, y_1, y_2$, we have
 \begin{align*}
     W_1(x_1, y_1) - W_1(x_2, y_2) & = \sqrt{|x_1 - y_1|} - \sqrt{|x_2 - y_2|} \\
     & \le \sqrt{|x_2 - y_2| + |x_1 - x_2| + |y_1 - y_2|} - \sqrt{|x_2 - y_2|} \\
     & \le |x_1 - x_2|^{1/2} + |y_1 - y_2|^{1/2},
 \end{align*}
 where we used the triangle inequality and the fact that $\sqrt{a + b} \le \sqrt{a} + \sqrt{b}$ for any non-negative $a, b$. For each $n \in \{20, 120, \ldots, 4920\}$, we sample a network $\bA$ and heterogeneity parameters $\btheta$ as follows:
 \begin{align}
 \label{eq: experiment sample scheme}
 \begin{cases}
     \bA_{ij} & = Bern(\rho_n W(\xi_{(i)}, \xi_{(j)})), \quad i < j,\\
     \btheta_i & = \xi_{(i)}^2,
 \end{cases}
 \end{align}
 where $\rho_n = n^{-0.25}$ and $\xi_1, \ldots,\xi_n$ are samples from the uniform distribution on $[0, 1]$. For each $\bA$, we compute optimal interventions $\widehat{\btheta}$ with budget $B = n / 2$ and the peer-effect parameter $\gamma = 0.8$ using the representation~\eqref{eq: hat theta equation} and the algorithm suggested in~\citep{gander1989constrained}. Then, we compute optimal interventions $\widehat{\btheta}'$ for the same $\btheta$ and the empirical graphon of $\widehat{\bQ}_{1/2}^*$ for the hard thresholding parameter
 \begin{align*}
     \lambda = \rho_n n^{3/4} + \rho_n^{3/4} n^{1/2}
 \end{align*}
 and the averaging parameter $m$ defined by Theorem~\ref{theorem: Holder averaging}. 
We also vary $\lambda$ over a power grid $\lambda \cdot 2^{i}, i = -2, \ldots, 2$, to see how the change of $\lambda$ influences the approximation error $T(\widehat{\btheta}) - T(\widehat{\btheta}')$.
 
In the first picture of the first row of Figure~\ref{fig:experiments}, we plot the dependence of the difference between target functions $T(\widehat{\btheta})$ and $T(\widehat{\btheta}')$ computed for interventions on the network $\bA$. In the second picture of the first row of Figure~\ref{fig:experiments}, we plot the rank of $\widehat{\bQ}^*_{1/2}$ for different values of $\lambda$. Finally, in the third picture of the first row of Figure~\ref{fig:experiments}, we plot the CPU time necessary to compute SVD of $\cS_m(\bA)$, truncated at the computed rank. The truncated SVD was computed using the LOBPCG algorithm~\citep{knyazev2001toward}. As one can see, the network $\bA$ can be efficiently approximated by a low-rank matrix, despite the fact that the underlying graphon is not low-rank. Black dashed line in the rightmost picture of the first row of Figure~\ref{fig:experiments} stands for the CPU time required to obtain optimal interventions for the network $\bA$. Note that small hard thresholding parameter $\lambda$ results in a moderate number of noisy SVD components in the estimator $\widehat{\bQ}_{1/2}^*$ for small values of $n$, while large $\lambda$ imposes significant bias, resulting in a suboptimal approximation of the interventions of the initial network. All curves are averaged over 30 trials.

 Then, we consider the case when a network is sampled from an SBM graphon with $K = 4$ communities, i.e.
 \begin{align*}
     W_2(x, y) = \bB_{\lceil 4x \rceil, \lceil 4y \rceil}
 \end{align*}
 and $\bB$ is a symmetric matrix which entries are generated as follows. Let $\{u_{ij}\}_{1 \le i \le j \le K}$ be independent samples from the uniform distribution on $[0, 1]$, then
 \begin{align*}
     \bB_{ij} = \begin{cases}
         u_{ij} / 2, \text{ if } i < j, \\
         (1 + u_{ij}) / 2, \text{ if } i = j.
     \end{cases}
 \end{align*}
 We generated the matrix $\bB$ once, and keep it fixed during all experiments. We sample 100 different pairs of a network $\bA$  and a heterogeneity vector $\btheta$ using the procedure~\eqref{eq: experiment sample scheme} for each $n \in \{20, 40, \ldots, 980\}$ . For each pair, we compute the optimal interventions $\widehat{\btheta}$ with budget $n/2$ and the peer-effect parameter $\gamma = 0.8$ as before. Then, we compute an estimator $\widehat{\btheta}'$ of optimal interventions based on truncated SVD with $4$ components and the true graphon $W_2$. We plot the difference between the target functions for the optimal intervention vector and its estimator in Figure~\ref{fig:experiments}, middle row. As can be seen, truncated SVD outperforms graphon-based interventions. In the second picture of the middle row, we reported CPU time necessary to compute optimal interventions (see ``initial network'' label) and interventions based on either the truncated SVD estimator or the graphon $W_2$. All curves are averaged over 100 trials.

 Finally, we run the following experiment. We generate a network $\bA$ and the heterogeneity vector $\btheta$ of size $N = 10000$ from $W_2$. Then, for each $n \in \{100, 200, \ldots, 1500\}$, we generate a network $\bA'$ of size $n$, construct an estimator $\widehat{W}$ of the sparse graphon $\rho_n W_2$ based on the truncated SVD, and estimate optimal interventions in $\bA$ using $\widehat{W}$. We estimate the ratio of sparsities $\rho_N / \rho_n$ using the ratio between average degrees. We plot the result in Figure~\ref{fig:experiments}, bottom row, left. The dashed line corresponds to interventions based on the true graphon $W_2$. We average the result over $100$ generations of matrices $\bA'$ for each $n$. We also reported CPU time in the right picture of the bottom row of Figure~\ref{fig:experiments}.

 \begin{figure}[h!]
     \centering
     \includegraphics[width=\linewidth]{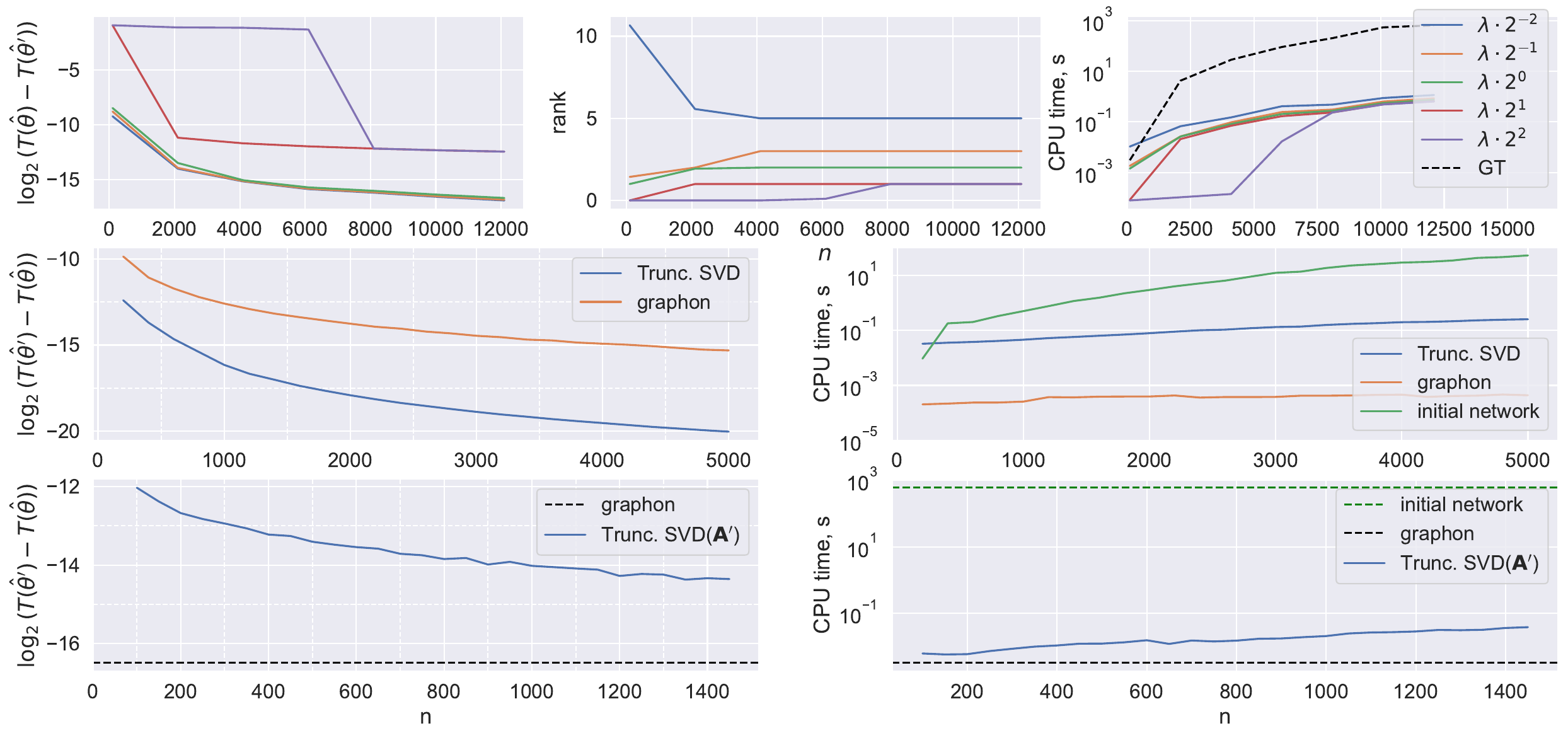}
     \caption{\textit{Upper Left:} difference between target functions for optimal interventions of a network on $n$ vertices sampled from $(1, 1/2)$-H\"older graphon $W_1(x, y) = \sqrt{|x - y|}$ and interventions based on estimator $\widehat{\bQ}^*_{1/2}$. \textit{First Row, Middle:} Rank of the estimator $\widehat{\bQ}^*_{1/2}$ for a network on $n$ vertices sampled from $W_1(x, y)$. \textit{First Row, Right:} CPU time required to obtain optimal interventions (black dashed line) and interventions using truncated SVD of $\cS_m(\bA)$. \textit{Middle Row, Left:} Difference between target functions for optimal interventions of a network on $n$ vertices sampled from SBM with $4$ communities and interventions computed using a) graphon b) hard-thresholding estimator. \textit{Middle Row, Right:} CPU time that is required to compute interventions based on a) initial network b) graphon c) hard-thresholding estimator. \textit{Bottom Row, Left:} The difference between target functions for optimal interventions of a network $\bA$ of size 10000 sampled from SBM with $4$ communities and interventions based on a) graphon b) hard-thresholding estimator computed for other network $\bA'$ of size $n$. \textit{Bottom Row, Right:} CPU time required to compute interventions based on a) initial network b) graphon c) hard-thresholding estimator.}
     \label{fig:experiments}
 \end{figure}

 \section{Conclusion}
\label{section: conclusion}

Driven by applications in graphon games, we developed an estimator for an unknown graphon based on a sampled network. Under standard regularity conditions for the graphon, our estimator possesses two key features: (i) it approximates the unknown graphon in the operator norm, and (ii) it is of low rank. We established upper bounds for the convergence rate and rank of the estimator. Furthermore, we demonstrated that the estimator yields near-optimal solutions for the social welfare problem in linear-quadratic graphon games, with efficient computation. We also quantified the convergence rates of the associated costs and the computational complexity relative to the network size.

\bibliographystyle{apalike}
\bibliography{refs.bib}
\newpage


\appendix

\section{Proof of Theorem~\ref{theorem: graphon perturbed welfare problem}}\label {proof_graphon_perturbed_welfare_problem}

\noindent \textbf{Step 1. Reduction to the difference of resolvents.} The problem of optimal interventions in graphon can be restated as follows:
\begin{align*}
    \max_{\hat{\theta}} T(\hat{\theta}) = \frac{1}{2} \left \Vert  \left ( \bI - \gamma \bbW \right )^{-1} (\theta + \hat \theta) \right \Vert_{L_2}^2 \\
    \text{subject to } \Vert \hat \theta\Vert^2_{L_2} \le B,
\end{align*}
see the proof of Proposition 1 in~\citep{Parise_Ozdaglar_Econometrica}. From the self-adjointness of the operator $\bbW$, we have
\begin{align*}
    T_i(\hat{\theta}) = \frac{1}{2} \left \langle (\bI - \gamma \bbW_i)^{-2} [\theta + \hat{\theta}], \theta + \hat \theta \right \rangle, \quad i = 1, 2.
\end{align*}

Let $\hat{\theta}_i$ be the argument of the maximum of $T_i(\hat{\theta})$. Then, we have
\begin{align}
    T_{1}(\hat{\theta}_2) & = T_1(\hat{\theta}_1) + T_{1}(\hat{\theta}_2) - T_{1}(\hat{\theta}_1) = T_1(\hat{\theta}_1) + T_1(\hat{\theta}_2) - T_2(\hat{\theta}_2) + T_2(\hat{\theta}_2) - T_1(\hat{\theta}_1) \nonumber \\
    & \ge  T_1(\hat{\theta}_1) + T_1(\hat{\theta}_2) - T_2(\hat{\theta}_2) + T_2(\hat{\theta}_1) - T_1(\hat{\theta}_1), \label{eq: lower bound on T with approximate theta}
\end{align}
where we used $T_2(\hat{\theta}_2) \ge T_2(\hat{\theta}_1)$ by the definition of the argument of the maximum in the last inequality. For brevity, we denote the resolvent by
\begin{align*}
    \cR(\gamma, \bbW) = (\bbI - \gamma \bbW)^{-1},
\end{align*}
so
\begin{align*}
    T_1(\widehat{\theta}_2) - T_2(\widehat{\theta}_2) & = \frac{1}{2} \left \langle (\cR^2(\gamma, \bbW_1) - \cR^2(\gamma, \bbW_2)) [\theta + \hat{\theta}_1], \theta + \hat{\theta}_2 \right \rangle \\
    & \ge -  \frac{1}{2} \Vert \cR^2(\gamma, \bbW_1) - \cR^2(\gamma, \bbW_2) \Vert_{\op} \Vert \theta + \theta_2 \Vert_{L_2}^2 \\
    & \ge -  \Vert \cR^2(\gamma, \bbW_1) - \cR^2(\gamma, \bbW_2) \Vert_{\op} (\Vert \theta \Vert^2_{L_2} + B).
\end{align*}
Similarly, we have
\begin{align*}
    T_2(\hat{\theta}_1) - T_1(\hat{\theta}_1) \ge - \Vert \cR^2(\gamma, \bbW_1) - \cR^2(\gamma, \bbW_2) \Vert_{\op} (\Vert \theta \Vert^2_{L_2} + B).
\end{align*}
Combining the above with~\eqref{eq: lower bound on T with approximate theta}, we obtain
\begin{align}
\label{eq: T stability throw resolvents difference}
    T_1(\hat{\theta}_2) \ge T_1(\hat{\theta}_1) - 2 \Vert \cR^2(\gamma, \bbW_1) - \cR^2(\gamma, \bbW_2) \Vert_{\op} \cdot (\Vert \theta \Vert_{L_2}^2 + B).
\end{align}

\noindent \textbf{Step 2. Analysis of squared resolvents.} To bound the operator norm of the difference between $\cR^2(\gamma, \bbW_i)$, we use decomposition
\begin{align*}
	\cR^2(\gamma, \bbW) = \sum_{i = 0}^\infty (i + 1) \gamma^i \bbW^i.
\end{align*}
The right-hand side converges because $\gamma \Vert \bbW_i \Vert <1$ for $i = 1, 2$.
Hence, we have
\begin{align*}
	\cR^2(\gamma, \bbW_1) - \cR^2(\gamma, \bbW_2) = \sum_{i = 1}^\infty (i + 1) \gamma^i (\bbW_1^i - \bbW_2^i).
\end{align*}
Finally, we bound $\Vert \bbW_1^i - \bbW_2^i \Vert_{\op}$ by induction on $i$ as follows:
\begin{align*}
	\Vert \bbW_1^i - \bbW_2^i \Vert_{\op} \le i \cdot \Vert \bbW_1 - \bbW_2 \Vert_{\op} \cdot (\Vert \bbW_1 \Vert_{\op} \vee \Vert \bbW_2 \Vert_{\op})^{i - 1}.
\end{align*}
The base $i = 1$ is trivial. Next, we obtain 
\begin{align*}
	{\Vert \bbW_1^i - \bbW_2^i \Vert}_{\op} & = \Vert \bbW^{i - 1}_1 (\bbW_1 - \bbW_2) + \bbW_2 (\bbW^{i - 1}_1 - \bbW_2^{i - 1}) \Vert_{\op} \\
	& \le \Vert \bbW_1 \Vert_{\op}^{i - 1}  \Vert \bbW_1 - \bbW_2 \Vert_{\op} + \Vert \bbW_2 \Vert_{\op} \Vert \bbW^{i - 1}_1 - \bbW_2^{i - 1} \Vert_{\op} \\
	& \overset{\text{I.H.}}{\le} \Vert \bbW_1 \Vert_{\op}^{i - 1} \Vert \bbW_1 - \bbW_2 \Vert_{\op} + (i - 1) \Vert \bbW_2 \Vert_{\op} \Vert \bbW_1 - \bbW_2 \Vert_{\op} (\Vert \bbW_1 \Vert_{\op} \vee \Vert \bbW_2 \Vert_{\op})^{i - 2} \\
	& \le i \cdot \Vert \bbW_1 - \bbW_2 \Vert_{\op} \cdot (\Vert \bbW_1 \Vert_{\op} \vee \Vert \bbW_2 \Vert_{\op})^{i - 1}.
\end{align*}
Hence, we have
\begin{align}
	\Vert \cR^2(\gamma, \bbW_1) - \cR^2(\gamma, \bbW_2) \Vert_{\op} & \le \Vert \bbW_1 - \bbW_2 \Vert_{\op} \sum_{i = 0}^{\infty} (i + 1)(i + 2) \gamma^{i + 1} (\Vert \bbW_1 \Vert_{\op} \vee \Vert \bbW_2 \Vert_{\op})^{i} \nonumber \\
	& = \frac{2 \gamma \Vert \bbW_1 - \bbW_2 \Vert_{\op}}{(1 - \gamma [\Vert \bbW_1 \Vert_{\op} \vee \Vert \bbW_2 \Vert_{\op}])^3}. \label{eq: resolvent different upper bound}
\end{align}
Then, the above, together with~\eqref{eq: T stability throw resolvents difference}, implies
\begin{align*}
    T_1(\hat{\theta}_2) \ge T_1(\hat{\theta}_1) - \frac{4 \gamma \Vert \bbW_1 - \bbW_2 \Vert_{\op} (\Vert \theta \Vert_{L_2}^2 + B)}{(1 - \gamma [\Vert \bbW_1 \Vert_{\op} \vee \Vert \bbW_2 \Vert_{\op}])^3},
\end{align*}
which completes the proof.

\section{Proof of Theorem~\ref{theorem: UHTE performance}}
\label{appendix: proof of HTE estimator error}

Define $\tilde{\bQ} = \bQ - \bD_{\bQ}$, where $\bD_{\bQ}$ consists of diagonal entries of $\bQ$. Now, $\Vert \bD_{\bQ} \Vert_{\op} = \max_i |(\bD_{\bQ})_{ii}| = \max_i \rho_n |W(\xi_i, \xi_j)| \le \rho_n$. 
We apply the triangle inequality, and derive
\begin{align*}
	\Vert \bQ - \widehat{\bQ}_\lambda \Vert_{\op} \le \Vert \bQ - \bA \Vert_{\op}+\Vert  \bA - \widehat{\bQ}_\lambda\Vert_{\op} =  \Vert \bD_{\bQ} \Vert_{\op}+ \Vert \tilde{\bQ}  - \bA \Vert_{\op} + \lambda \le \rho_n + \lambda + \Vert \tilde{\bQ} - \bA \Vert_{\op}.
\end{align*}
Conditioned on $\xi_1, \ldots, \xi_n$, we bound the difference $\Vert \tilde{\bQ} - \bA \Vert_{\op}$ in probability by Corollary~\ref{corollary: symmetric matrix concentration} in Appendix \ref{appendix: tools}. For $\bX=\tilde{\bQ} - \bA$, we have $\sigma \leq \sqrt{n \rho_n}$ since
\begin{align*}
    \sigma &= \max_i \sqrt{\sum_{j} \Var(\bA_{ij})} = \max_i \sqrt{\sum_j \left(\bbE(\bA_{ij}^2) - \bbE(\bA_{ij})^2\right)} = \max_i \sqrt{\sum_j \left(\rho_n W(\xi_i, \xi_j) - \rho_n^2 W(\xi_i, \xi_j)^2\right)} \\
    &= \max_i \sqrt{\sum_j \rho_n W(\xi_i, \xi_j)\left(1 - \rho_n W(\xi_i, \xi_j)\right)} \leq \max_i \sqrt{\sum_j \rho_n W(\xi_i, \xi_j)} \leq \max_i \sqrt{\sum_j \rho_n} = \sqrt{n \rho_n}. 
\end{align*}
and $\sigma_* \le 1$ since $\tilde{\bQ}_{ij}, \bA_{ij} \in [0,1]$ for all $(i,j)$. Selecting $t=\sqrt{\constantBandeira \log \frac{n}{\delta}}$, we obtain $\Vert \tilde{\bQ} - \bA \Vert_{\op} \le 6 \sqrt{n \rho_n} + \sqrt{\constantBandeira \log \frac{n}{\delta}}$ with probability at least $1 -\delta$. Since the bound does not depend on $\xi_1, \ldots, \xi_n$, the conditioning can be removed. Finally, we have $\rho_n \le \sqrt{n \rho_n}$, and, thus, the result follows.

\section{Proof of Theorem~\ref{theorem: SBM convergence rate}}
\label{appendix: proof of SBM convergence rate}

{\textbf{i) Upper bound on the convergence rate.} 
We build the proof on that of Proposition 3.2 in~\citep{klopp2017}. However, we provide bounds in probability instead of bounds in expectation and we use the operator norm.
Define a function $\phi: [0, 1] \to [k]$, such that $\phi(x) = a$ iff $x \in P_a$ and a function $\psi: [0, 1] \to [k]$ such that for any $x\in \left(\frac{i-1}{n},\frac{i}{n}\right), \psi(x)=\phi(\xi_{i})$. 
Thus,
\begin{align*}
	\left\Vert  \rho_n \bbW - {\bbW}_{\bQ} \right\Vert_{\op}& \le \Vert  \rho_n \bbW - {\bbW}_{\bQ} \Vert_{F} \\
 & \le \rho_n \sqrt{\int_{[0,1]^2} (\bS_{\phi(x)\phi(y)} - \bS_{\psi(x)\psi(y)})^2 dx dy} \\
	& \le \rho_n \sqrt{\mu \bigl (\psi(x) \neq \phi(x) \text{ or } \psi(y) \neq \phi(y) \bigr )} \\
	& \le \rho_n \sqrt{2 \mu \bigl (\psi(x) \neq \phi(x) \bigr )}.
\end{align*}
Define $\widehat{e}_a = \frac{1}{n} \sum_{i = 1}^n \indicator{\xi_i \le e_a}.$
Note that $\{x \mid \phi(x) = \psi(x)\} \subseteq \bigcup_{a = 1}^{k - 1} [\max(e_{a - 1},\hat{e}_{a - 1}), \min(e_{a},\hat{e}_a)]$, where we assume that the segment \newline $[\max(e_{a - 1},\hat{e}_{a - 1}), \min(e_{a},\hat{e}_a)]$ is empty if $\min(e_{a}, \hat{e}_a) < \max(e_{a - 1}, \hat{e}_{a - 1})$. Hence, we have
\begin{align*}
    \mu\{\phi(x) \neq \psi(x)\} \le \sum_{a = 1}^{k - 1} |\widehat{e}_a - e_a|.
\end{align*}
We are left to bound $|e_a - \widehat{e}_a|$ for each $a \in [k-1]$. Fix any $a \in [k-1]$. Applying the Bernstein inequality, we obtain
\begin{align*}
        \P \left (\left |\widehat e_a - e_a \right| \ge t \right ) = \P \left( \left| \sum_{i = 1}^n (\indicator{\xi_i \le e_a} - e_a) \right| \ge n t\right) \le 2 \exp \left( - \frac{n^2 t^2 / 2}{\sum_{i = 1}^n  \Var(\indicator{\xi_i \le e_a})  + n t/3}\right ). 
\end{align*}
Since $\Var(\indicator{\xi_i \le e_a}) =  e_a(1 - e_a) \leq e_a$, we have
\begin{align*}
    \P \left(|e_a - \widehat e_a | \ge t \right ) \le 2 \exp \left( - \frac{n t^2}{2(e_a + t/3)}\right )
\end{align*}
for each $a = [k - 1]$. We choose $t = t_a$ to ensure that the right-hand side is at most $\delta/(2 k)$. For our purposes, it is sufficient to set
\begin{align*}
    t_a = \frac{2}{3n} \log \left( \frac{4k}{\delta} \right) + \sqrt{\frac{2 e_a}{n} \log \left ( \frac{4 k}{\delta}\right )}.
\end{align*}
Hence, by the union bound, with probability at least $1 - \delta/2$, it holds
\begin{align}
\label{eq: mu psi minus phi bound}
    \mu \{ \psi(x) \neq \phi(x)\} \le \sum_{a = 1}^{k-1} t_a \leq \frac{2k}{n} \log \frac{4k}{\delta} + \sum_{a = 1}^{k - 1} \sqrt{\frac{2e_a}{n} \log \frac{4k}{\delta}}.
\end{align}
The second term in the right-hand side can be bounded using the Cauchy-Schwartz inequality to obtain
\begin{align*}
    \sum_{a = 1}^{k - 1} \sqrt{\frac{2e_a}{n} \log \frac{4k}{\delta}} \leq \sqrt{\frac{2k }{n} \sum_{a = 1}^{k-1} e_a \log \frac{4k}{\delta}}.
\end{align*}
Assumption $n \ge 2 k \log \frac{4k}{\delta}$ implies
\begin{align*}
    \frac{2k}{n} \log \frac{4k}{\delta} \leq \sqrt{\frac{2 k}{n} \log \frac{4k}{\delta}}. 
\end{align*}
Substituting the above two displays into~\eqref{eq: mu psi minus phi bound}, we get
\begin{align*}
    \mu \{ \psi(x) \neq \phi(x)\} \leq \sqrt{\frac{8k \sum_{a = 1}^{k - 1} e_a}{n} \log \frac{4k}{\delta}}. 
\end{align*}
Hence,  
\begin{align}
\label{eq: error of connection probabilities graphon}
    \Vert \rho_n \bbW - \bbW_{\bQ} \Vert_{\op} \le \rho_n \sqrt{2 \mu(\psi(x)) \neq \rho(x))} \leq \rho_n \left (\frac{32 k \sum_{a = 1}^{k - 1} e_a}{n} \log \frac{4k}{\delta}\right )^{1/4}
\end{align}
with probability at least $1 - \delta/2$.
Bounding $\Vert \bQ - \widehat{\bQ}_\lambda \Vert$ via Theorem~\ref{theorem: UHTE performance} with probability at least $1 - \delta/2$ and applying Lemma \ref{lemma: empirical graphon norm} we obtain the result.

\textbf{ii) Upper bound on the rank.} For arbitrary matrix of connection probabilities $\bQ$, let $\widehat{\bQ}$ be its arbitrary estimator and $\operatorname{HT}_{\lambda}(\widehat{\bQ})$ be truncated SVD with threshold $\lambda$. Then, the rank of $\widehat{\bQ}_\lambda$ can be bounded as follows.
\begin{prop}
\label{proposition: singular value rank}
	Let $\sigma_1 \ge \ldots \ge \sigma_n$ be singular values of $\bQ$. Then, we have
	\begin{align*}
		\rank(\operatorname{HT}_{\lambda}(\widehat{\bQ})) & \le \max \{i \mid \sigma_i \ge \lambda -\Vert \bQ -\operatorname{HT}_{\lambda}(\widehat{\bQ}) \Vert_{\op}\}.
	\end{align*}	
	
\end{prop}

\begin{proof}
	If $i \le \rank(\operatorname{HT}_{\lambda}(\widehat{\bQ}))$, then $\sigma_i(\widehat{\bQ}) \ge \lambda$. Consequently, $\lambda \le \sigma_i(\bQ) + \Vert \widehat{\bQ} - \bQ \Vert_{\op} $ for such $i$ by Proposition~\ref{proposition: weyl inequality}, and $\rank(\widehat{\bQ}_\lambda) \le \max \{ i \mid \sigma_i(\bQ) \ge \lambda - \Vert \widehat{\bQ} - \bQ \Vert_{\op} \}$. 
\end{proof}

In our case, $\lambda = 9 \sqrt{n \rho_n}$ and, as in the proof of Theorem~\ref{theorem: UHTE performance}, we have 
\begin{align}
    \Vert \bQ - \bA \Vert_{\op} \le \Vert \bQ - \bA - \bD_{\bQ}\Vert_{\op} + \Vert \bD_{\bQ} \Vert_{\op} \le   7 \sqrt{n \rho_n} + \sqrt{\constantBandeira \log \frac{2n}{\delta}} \label{eq: Bandeira bound for Q minus A}
\end{align}
with probability at least $1 - \delta/2$ by Corollary~\ref{corollary: symmetric matrix concentration}. Since $n \rho_n \ge \constantBandeira \log \frac{2n}{\delta}$, we have $\lambda - \Vert \bQ - \widehat{\bQ}_\lambda \Vert_{\op} \ge  \sqrt{n \rho_n}$. Since $\bQ$ has rank at most $k$, we have $\sigma_i(\bQ) = 0$ for $i > k$ and
\begin{align*}
    \rank(\widehat{\bQ}_\lambda) \le \max \{i \mid \sigma_i(\bQ) \ge \sqrt{n \rho_n}\} \le k.
\end{align*}

\section{Proof of Theorem~\ref{theorem: SBM lower bound}}
\label{appendix: proof of SBM lower bound}

We use the standard framework based on hypotheses testing and adapt the proof technique from \citep{klopp2017}. As discussed in Section~\ref{sec:minimax bounds}, we should construct two SBM graphons $U$ and $V$ such that
\begin{enumerate}
    \item $\opnorm{\rho_n U - \rho_n V} \ge \rho_n n^{-1/4} / 8$;
    \item the corresponding distributions $\P_U$ and $\P_V$ satisfy $\KL{\P_V}{\P_U} \le 1/2$.
\end{enumerate}
We introduce a family of graphons parametrized by some real number $m \in (-1/2; 1/2)$ as follows:
\begin{align*}
    W_{m}(x, y) = \begin{cases}
        3/4, & \text{either $x, y \le 1/2 + m$ or $x, y \ge 1/2 + m$}, \\
        1/4, & \text{otherwise}.
    \end{cases}
\end{align*}
We choose $U$ to be $W_0$ and $V$ to be $W_{\Delta}$ for $\Delta$ specified in the sequel. First, we bound the difference between $W_0$ and $W_{\Delta}$ in the operator norm.
\begin{prop}
    \label{prop:risk lower bound}
    We have
    \begin{align*}
        \opnorm{W_0 - W_{\Delta}} = \frac{1}{2} \sqrt{\Delta (1 - \Delta)}.
    \end{align*}
\end{prop}

\begin{proof}
    Due to symmetry, we may assume that $\Delta > 0$. Note that a graphon $W' = W_{\Delta} - W_0$ is an SMB with three communities:
    \begin{align*}
        W'(x, y) & = \frac{1}{2} \indicator{x \in (1/2; 1/2 + \Delta)} \indicator{y \in (0; 1/2)} + \frac{1}{2} \indicator{x \in (0; 1/2)} \indicator{y \in (1/2; 1/2 + \Delta)} \\ 
         &\hskip -0.7 cm \quad - \frac{1}{2} \indicator{x \in (1/2 + \Delta; 1)} \indicator{y \in (1/2; 1/2 + \Delta)} - \frac{1}{2} \indicator{x \in (1/2; 1/2 + \Delta)} \indicator{y \in (1/2 + \Delta; 1)}.
    \end{align*}
    Since for any function $f$, $\bbW' f$ can be represented as
    \begin{align*}
        (\bbW' f)(x) = f_1 \indicator{x \in (0; 1/2)} + f_2 \indicator{x \in (1/2; \Delta)} + f_2 \indicator{x \in (1/2 + \Delta; 1)}
    \end{align*}
    for some numbers $f_1, f_2, f_3$,
    any eigenfunction $\bv$ of $\bbW'$ has the form:
    \begin{align*}
        \bv(x) = v_1 \indicator{x \in (0; 1/2)} + v_2 \indicator{x \in (1/2; {1/2 +}\Delta)} + v_3 \indicator{x \in (1/2 + \Delta; 1)}.
    \end{align*}
    Note that, the matrix of connection probabilities between communities of the SBM $W'$ has rank $2$  and $\bbW'$ has zero trace, so the graphon $W'$ has two different non-zero eigenvalues $\lambda$ and $-\lambda$, $\lambda > 0$. We find $\lambda$ from the system $\bbW' \bv = \lambda \bv$. We have
    \begin{align*}
        \begin{cases}
            \frac{\Delta}{2} v_2 & = \lambda v_1 \\
            \frac{1}{2} v_1 \cdot \frac{1}{2}- \frac{1}{2} v_3  \left ( \frac{1}{2} - \Delta \right )& = \lambda v_2 \\
            - \frac{\Delta}{2} v_2 & = \lambda v_3.
        \end{cases}
    \end{align*}
    From the first and the third equations, we deduce $v_1 =  - v_3$. Thus, we have
    \begin{align*}
        \begin{cases}
            \frac{\Delta}{2} v_2 & = \lambda v_1 \\
            \frac{1}{2} v_1 (1 - \Delta) & = \lambda v_2.
        \end{cases}
    \end{align*}
We multiply these equations, and get $\lambda^2 = \Delta (1 - \Delta) / 4$. Therefore, the proposition follows.
\end{proof}

Next, we bound the KL-divergence between distributions $\P_0$ and $\P_{\Delta}$ corresponding to graphons models $\rho_n W_0$ and $\rho_n W_{\Delta}$ respectively.
\begin{prop}
    \label{prop:KL bound}
    We have that $\KL{\P_{\Delta}}{\P_0} \le 16 n \Delta^2$.
\end{prop}
\begin{proof}
By the definition, we have
\begin{align*}
    \KL{\P_{\Delta}}{\P_0} = \sum_{A} \P_{\Delta}(\bA = A) \log \frac{\P_{\Delta}(\bA = A)}{\P_0(\bA = A)},
\end{align*}
where the sum is taken over all adjacency matrix $A$.
We introduce a new random variable $\theta = \max \{ k \mid \xi_{k} \text{ belongs to the first community}  \}$, with the convention that $\theta=0$ if $\xi_{k}$ belongs to the second community.  Subject to $\theta$, the distributions $\P_{\Delta} (\cdot \mid \theta)$ and $\P_0 (\cdot \mid \theta)$ are identical, since both of them are random graph models with the same matrix of connection probabilities $\bQ$:
\begin{align*}
    \bQ_{ij} = 
    \begin{cases}
        3\rho_n/4 , & \text{if $i, j \le \theta$ or $i, j > \theta$,} \\
        \rho_n/4 , & \text{otherwise.}
    \end{cases}
\end{align*}
For brevity, we define $p_{\bA | \theta}(A \mid t) = \P_0 (\bA = A \mid \theta = t) = \P_{\Delta} (\bA = A \mid \theta = t)$, $q(t) = \P_{\Delta}(\theta = t)$ and $p(t) = \P_{0}(\theta = t)$. Then, we get
\begin{align*}
    \KL{\P_{\Delta}}{\P_0} = \sum_{A} \sum_t p_{\bA \mid \theta}(A \mid t) q(t) \log \frac{\sum_{t} p_{\bA \mid \theta}(A \mid t) q(t)}{\sum_t p_{\bA \mid \theta} (A \mid t) p(t)}.
\end{align*}
Since the function $(x, y) \to x \log (x / y)$ is convex, due to the Jensen inequality, we have
\begin{align*}
    \KL {\P_{\Delta}}{\P_0} \le \sum_{t} q(t) \log \frac{q(t)}{p(t)}.
\end{align*}
The latter is the KL-divergence between two binomial distributions $Binom(n, 1/2 + \Delta)$ and \\$Binom(n, 1/2)$. We have
\begin{align*}
    \sum_{t} q(t) \log \frac{q(t)}{p(t)} & = \sum_{t = 0}^n \binom{n}{t}(1/2 + \Delta)^t (1/2 - \Delta)^{n - t} \log \frac{(1/2 + \Delta)^t (1/2 - \Delta)^{n - t}}{2^{-n}}  \\
    & = \KL{Bern^{\otimes n}(1/2 + \Delta)}{Bern^{\otimes n}(1/2)} \\
    & = n \KL{Bern(1/2 + \Delta)}{Bern(1/2)} \le {4} n \Delta^2,
\end{align*}
where we use the fact that KL-divergence does not exceed chi-square divergence.
\end{proof}

We choose $\Delta = ({2} \sqrt{2n})^{-1}$, so $\KL{\P_{\Delta}}{\P_0} \le 1/2$ and $\opnorm{\rho_n W_0 - \rho_n W_{\Delta}} \ge \rho_n n^{-1/4} / 8$. Application of Theorem~\ref{theorem:minimax bounds} finishes the proof.

\section{Proof of Theorem~\ref{theorem: Holder averaging}}
\label{section: upper bound proof - smooth graphons}

\begin{proof}
We analyze the bias part $\cS_m(\bQ) - \bQ$ and the variance  part $\cS_m(\bA) - \cS_m(\bQ)$ of the error matrix $\cS_m(\bA) - \bQ$ separately.

\noindent \textbf{Step 1. Bias.} We have
\begin{align}
    \Vert \cS_m(\bQ) - \bQ \Vert_{\op} & \le \max_i \sum_{j} |(\cS_m(\bQ))_{ij} - \bQ_{ij}|. \label{eq: L1 bound for bias}
\end{align}
By the definition of $\cS_m(\cdot)$ and the Jensen inequality, we have
\begin{align*}
     |\cS_m(\bQ)_{ij} - \bQ_{ij} | & \le \frac{1}{|\cB(i, j)|} \sum_{(i', j') \in \cB(i, j)} |\bQ_{i' j'} - \bQ_{i j}| \\
    & = \frac{1}{|\cB(i, j)|} \sum_{(i', j') \in \cB(i, j)} |\rho_n W(\xi_{(i')}, \xi_{(j')}) - \rho_n W(\xi_{(i)}, \xi_{(j)})|.
\end{align*}
Since the graphon is $(H, \alpha)$-H\"older, there exists a constant $C(H, \alpha) \ge 1$, such that for any $x, y, x', y'$, we have
\begin{align*}
    |W(x, y) - W(x', y')| \le C(H, \alpha) \left (|x - x'|^{\alpha \wedge 1} + |y - y'|^{\alpha \wedge 1} \right ).
\end{align*}
By the triangle inequality, it implies
\begin{align}
    |W(\xi_{(i)}, \xi_{(j)}) - W(\xi_{(i')}, \xi_{j'})| & \le C(H, \alpha) |\xi_{(i)} - \xi_{(i')}|^{\beta} + C(H, \alpha) |\xi_{(j)} - \xi_{(j')}|^{\beta} \nonumber
\end{align}
Using
\begin{align*}
    |\xi_{(i)} - \xi_{(i')}|^{\beta} & \le |\xi_{(i)} - i/n|^{\beta} + |\xi_{(i')} - i'/n|^\beta + |i - i'|^\beta / n^\beta, \\
    |\xi_{(j)} - \xi_{(j')}|^{\beta} & \le |\xi_{(j)} - j/n|^{\beta} + |\xi_{(j')} - j'/n|^\beta + |j - j'|^\beta / n^\beta, 
\end{align*}
and setting $\Delta = \max_{i} |\xi_{(i)} - i/n|$, we deduce
\begin{align*}
    |W(\xi_{(i)}, \xi_{(j)}) - W(\xi_{(i')}, \xi_{j'})| & \le C(H, \alpha) \left (4 \Delta^{\beta} + \left | \frac{i}{n} - \frac{i'}{n} \right |^{\beta}+ \left | \frac{j}{n} - \frac{j'}{n}\right |^{\beta} \right ).
\end{align*}
Since $(i', j') \in \cB(i, j)$, we have $|i - i'| \le 2m, |j - j'|\le 2m$ and so
\begin{align}
    |W(\xi_{(i)}, \xi_{(j)}) - W(\xi_{(i')}, \xi_{j'})| & \le 4 C(H, \alpha) \left ( \Delta^{\beta} +  \left (\frac{m}{n} \right )^{\beta} \right ). \label{eq: point difference in backets}
\end{align}
Then, by Corollary~\ref{corollary: order statistics unifrom concentration}, we have
\begin{align*}
    \max_i |\xi_{(i)} - i/n|^{\beta} \le \left ( \frac{\log (2/\delta)}{2n} \right )^{\beta/2}
\end{align*}
with probability $1 - \delta$.
Substituting to~\eqref{eq: point difference in backets} implies
\begin{align*}
    |W(\xi_{(i)}, \xi_{(j)}) - W(\xi_{(i')}, \xi_{j'})| \le 4 C(H, \alpha)   \left ( \frac{\log(2/\delta)}{2n} \right )^{\beta/2} + 4 C(H, \alpha)  (m/n)^{\beta}
\end{align*}
with probability at least $1 - \delta$. With the same probability, we have
\begin{align*}
    |\cS_m(\bQ)_{ij} - \bQ_{ij} | \le 4 C(H, \alpha)  \rho_n  \left ( \frac{\log(2/\delta)}{2n} \right )^{\beta/2} +  C(H, \alpha)\rho_n (m/n)^{\beta}
\end{align*}
uniformly over $i, j$. 
Applying the above to~\eqref{eq: L1 bound for bias}, we deduce
\begin{align}
    \Vert \cS_m(\bQ) - \bQ \Vert_{\op} \le 4 C(H, \alpha)  \rho_n n \cdot \left ( \frac{\log(2/\delta)}{2 n} \right )^{\beta/2} + 4 C(H, \alpha) \rho_n n \cdot (m/n)^{\beta} \label{eq: bias leading term}
\end{align}
with probability at least $1 - \delta$. 

\noindent \textbf{Step 2. Variance.} For brevity, define $\bE = \bA - \bQ - \bD_{\bQ}$, where $\bD_{\bQ}$ is the diagonal of $\bQ$. Then, we define a matrix $\overline{\bE} \in \bbR^{s \times s}$ of averages over $\cB(i, j)$ as follows:
\begin{align*}
    \avgMatrix_{a,b} = \frac{1}{|\cB_{(a - 1)m + 1}||\cB_{(b - 1)m + 1}|} \sum_{\substack{i' \in \cB_{(a -1)m + 1} \\ j' \in \cB_{(b - 1)m + 1}}} \bE_{i' j'}.
\end{align*}
Then, we have
\begin{align*}
    \cS_m(\bE) = Z \avgMatrix Z^\top
\end{align*}
for the matrix $Z_{ia} = \indicator{i \in \cB_{(a - 1)m + 1}}$. Since $Z^\top Z$ is the diagonal matrix with entries $|\cB_{(a - 1)m + 1}|, a = 1, \ldots, s$, we have
\begin{align*}
    \Vert \cS_m(\bE) \Vert_{\op} \le \Vert Z \Vert_{\op}^2 \Vert \avgMatrix \Vert_{\op} = \Vert Z^\top Z \Vert_{\op} \Vert \avgMatrix \Vert_{\op} \le \max_{a \in [s]} |\cB_{(a - 1)m + 1}| \cdot \Vert \avgMatrix \Vert_{\op} \le 2m \Vert \avgMatrix \Vert_{\op}.
\end{align*}

To bound $\Vert \avgMatrix \Vert_{\op}$, we use the following lemma, which proof is given in Section~\ref{section: proof of averages operator norm concentration}.
\begin{lemma}
\label{lemma: central averages lemma}
    For any $\delta \in (0, 1)$, we have
    \begin{align*}
        \Vert \avgMatrix\Vert_{\op} \le    6 \sqrt{\frac{s \rho_n}{m^2}} + 4 \left ( \sqrt{\frac{32 \rho_n \log(8 s^2 m^2/\delta)}{m^2}} + \frac{12 \log(8 s^2 m^2/\delta)}{m^2}\right ) \sqrt{\constantBandeira \log(2s/\delta)}.
    \end{align*}
with probability at least $1 - \delta$.
\end{lemma}
It implies
\begin{align}
     \Vert \cS_m(\bE) \Vert_{\op} \le 12 \sqrt{s \rho_n} + 192 \left ( \sqrt{\rho_n \log(3 n/\delta)} + \frac{\log(3 n/\delta)}{m}\right ) \sqrt{\constantBandeira \log(2s/\delta)}\label{eq: bound on E0}
\end{align}
with probability at least $1 - \delta$. It implies
\begin{align*}
    \Vert \cS_m(\bA) - \cS_m(\bQ) \Vert_{\op} & = \Vert \cS_m(\bE + \bD_{\bQ}) \Vert_{\op} \le \Vert \cS_m(\bE) \Vert_{\op} + \Vert \cS_m(\bD_{\bQ})\Vert_{\op}  \\
    & \le \Vert \cS_m(\bE) \Vert_{\op} + \rho_n \\
    & \le 13 \sqrt{s \rho_n} + 192 \left ( \sqrt{\rho_n \log(3 n/\delta)} + \frac{\log(3 n/\delta)}{m}\right ) \sqrt{\constantBandeira \log(2s/\delta)}.
\end{align*}

\noindent \textbf{Step 3. Balancing bias and variance.} Combining the above with~\eqref{eq: bias leading term}, we obtain
\begin{align*}
    \Vert \cS_m(\bA) - \bQ \Vert_{\op} & \le 4 C(H, \alpha)  \rho_n n \cdot \left ( \frac{\log(4/\delta)}{2 n} \right )^{\beta/2} + 4 C(H, \alpha) \rho_n n \cdot (m/n)^{\beta} \\
    & \quad + 13 \sqrt{\frac{n \rho_n}{m}} + 192 \left ( \sqrt{\rho_n \log(6 n/\delta)} + \frac{\log(6n/\delta)}{m}\right ) \sqrt{\constantBandeira \log \frac{4 n}{m\delta}}
\end{align*}
with probability at least $1 - \delta$.
For $m = \left \lceil n^{(2 \beta - 1)/(2 \beta + 1)} \rho_n^{-1/(2 \beta + 1)} \right \rceil$, we obtain
\begin{align}
    \Vert \cS_m(\bA) - \bQ \Vert_{\op} & \le 4 C(H, \alpha)  \rho_n n \cdot \left ( \frac{\log(4/\delta)}{2 n} \right )^{\beta/2} \nonumber \\
    & \quad + 4 C(H, \alpha) \rho_n n \cdot \left (\frac{n^{(2 \beta - 1)/(2 \beta + 1)} \rho_n^{-1/(2 \beta + 1)} +1}{n} \right )^{\beta}  \label{eq: bias leading term in total sum} \\
    & \quad + 13 \sqrt{\frac{n \rho_n}{n^{(2 \beta - 1)/(2 \beta + 1)} \rho_n^{-1/(2 \beta + 1)}}} \label{eq: variance leading term in total sum} \\
    & \quad + 192 \left ( \sqrt{\rho_n \log(6 n/\delta)} + \frac{\log(6n/\delta)}{n^{(2 \beta - 1)/(2 \beta + 1)} \rho_n^{-1/(2 \beta + 1)}}\right ) \sqrt{\constantBandeira \log \frac{4 n}{m\delta}} \nonumber
\end{align}
with probability at least $1 - \delta$.
Using $n^{2\beta - 1} \ge \rho_n$ due to assumption~\eqref{eq: m is larger 1 assumption}, we upper bound the sum of terms~\eqref{eq: bias leading term in total sum} and~
\eqref{eq: variance leading term in total sum} by $(8 C(H, \alpha) + 13)\rho_n^{\frac{\beta + 1}{2 \beta + 1}} n^{1/(2 \beta + 1)}$. Then, we get
\begin{align}
     \Vert \cS_m(\bA) - \bQ \Vert_{\op} & \le 4 C(H, \alpha)  \rho_n n \cdot \left ( \frac{\log(4/\delta)}{2n} \right )^{\beta/2} \nonumber \\
    & \quad + (8 C(H, \alpha) + 13)\rho_n^{\frac{\beta + 1}{2 \beta + 1}} n^{1/(2 \beta + 1)
    } \label{eq: varaince term - simplified sum} \\
    & \quad + 192 \left ( \sqrt{\rho_n \log(6 n/\delta)} + \frac{\log(9n/\delta)}{n^{(2 \beta - 1)/(2 \beta + 1)} \rho_n^{-1/(2 \beta + 1)}}\right ) \sqrt{\constantBandeira \log \frac{4 n}{m\delta}} \label{eq: remainder - simplified sum}.
\end{align}
Next, by the first condition of~\eqref{eq: assumption sqrt-rho}, we have
\begin{align*}
    \rho_n^{(\beta + 1)/(2 \beta + 1)} n^{1/(2\beta + 1)} \ge \rho_n^{1/2} \log \frac{6n}{\delta} \ge \rho_n^{1/2} \sqrt{\log \frac{6 n}{\delta} \cdot \log \frac{4n}{m \delta}}
\end{align*}
and, by the second condition of~\eqref{eq: assumption sqrt-rho}, we have
\begin{align*}
    (n \rho_n^{\beta + 1})^{1/(2 \beta + 1)} \ge \left ( \frac{\rho_n}{n^{2 \beta - 1}} \right )^{\frac{1}{2 \beta + 1}} \log^{3/2} \frac{9n}{\delta} \ge \left ( \frac{\rho_n}{n^{2 \beta - 1}} \right )^{\frac{1}{2 \beta + 1}} \log \frac{9n}{\delta} \sqrt{\log \frac{4n}{m \delta}},
\end{align*}
so~\eqref{eq: varaince term - simplified sum} dominates remainder~\eqref{eq: remainder - simplified sum}. It implies
\begin{align*}
    \Vert \cS_m(\bA) - \bQ \Vert_{\op} & \le4 C(H, \alpha)   \rho_n n \cdot \left ( \frac{\log(4/\delta)}{n} \right )^{\beta/2} \\
    & \quad + (8 C(H, \alpha) + 13 + 384 \constantBandeira^{1/2})\rho_n^{\frac{\beta + 1}{2 \beta + 1}} n^{1/(2 \beta + 1) 
    }
\end{align*}
with probability at least $1- \delta$.
\end{proof}

\subsection{Proof of Lemma~\ref{lemma: central averages lemma}}
\label{section: proof of averages operator norm concentration}

Through the proof of the lemma, we consider the distribution of $\bE$ conditional on $\xi_1, \ldots, \xi_n$, so its upper diagonal entries are independent Bernoulli random variables with parameters $\bQ_{ij}$. 

\noindent \textbf{Step 1. Choosing a threshold.} A naive application of Corollary~\ref{corollary: symmetric matrix concentration} with $\sigma_* = 1$ to $\overline{\bE}$ would lead to suboptimal concentration bounds, so we choose a threshold $\tau$ and decompose $\overline{\bE}$ into the sum of two matrices $\avgLessT$ and $\avgLarge$ with entries smaller and larger than $\tau$, respectively: 
$\avgMatrix = \avgLessT+ \avgLarge$, where
\begin{align*}
    (\avgLessT)_{ab} & = \avgMatrix_{ab} \indicator{|\avgMatrix_{ab}| \le \tau}, \\
    (\avgLarge)_{ab} & = \avgMatrix_{ab} \indicator{|\avgMatrix_{ab}| > \tau}.
\end{align*}
It will be convenient to choose $\tau$ such that
\begin{align}
    \bbP \left ( |\avgMatrix_{a,b}| > \tau \right ) \le \frac{\delta}{2s^2 m^2} \label{eq: large probability bound for matrix concentration}
\end{align}
holds for any $a, b \in [s]$. 
By the Bernstein inequality, if $a \neq b$, we have
\begin{align*}
    \bbP \left ( |\avgMatrix_{a,b}| > \tau \right ) \le 2 \exp \left ( - \frac{m^2 \tau^2/2}{\rho_n+ \tau/3 } \right ).
\end{align*}
If $a = b$, then set
\begin{align*}
    p = \begin{cases}
        m & \text{ if } a \le s - 1 \\
        m + \ell & \text{if } a = s.
    \end{cases}
\end{align*}
Then, we have
\begin{align*}
     \bbP \left ( |\avgMatrix_{ab}| > \tau \right ) & \le \bbP \left ( \frac{1}{p^2} \left |\sum_{r = 1}^p \bE_{(a - 1) m + r, (a - 1)m + r} \right | > \tau/2  \right ) \\
     & \qquad + \bbP \left ( \frac{2}{p^2} \left |\sum_{r = 1}^p \sum_{r' = r + 1}^p \bE_{(a - 1) m + r, (a - 1)m + r'} \right | > \tau/2  \right ) \\
     & \le 2 \exp \left ( - \frac{p^4 \tau^2/8}{(\rho_n p + \tau/6p^2)} \right ) + 2 \exp \left (- \frac{p^2 \tau^2/32}{\rho_n /2  + \tau/12} \right )
\end{align*}
if $a = b$. In either case, we have
\begin{align*}
    \bbP \left ( |\overline{\bE}_{ab}| > \tau \right ) \le 4 \exp \left ( - \frac{m^2 \tau^2/32}{\rho_n + \tau/3}\right ).
\end{align*}
So, we set
\begin{align*}
    \tau = \sqrt{\frac{32 \rho_n \log(8 s^2 m^2/\delta)}{m^2}} + \frac{11 \log(8 s^2 m^2/\delta)}{m^2},
\end{align*}
for which~\eqref{eq: large probability bound for matrix concentration} is satisfied.

\noindent \textbf{Step 2. Applying matrix concentration.} Next, we apply Corollary~\ref{corollary: symmetric matrix concentration} to $\avgLessT - \bbE \avgLessT$. We have
\begin{align*}
    \max_a \sum_b \bbE (\avgLessT - \bbE \avgLessT)_{ab}^2 \le \max_a \sum_b \bbE (\avgLessT)_{ab}^2 \le \max_a \sum_b \bbE (\overline{\bE}_0)_{ab}^2 \le s \rho_n / m^2,
\end{align*}
and
\begin{align*}
    \max_{a, b} \left | (\avgLessT - \bbE \avgLessT)_{ab} \right | \le 2\tau.
\end{align*}
Therefore, by Corollary~\ref{corollary: symmetric matrix concentration}, we obtain
\begin{align*}
    \Vert \avgLessT - \bbE \avgLessT \Vert_{\op} \le 6 \sqrt{s \rho_n / m^2} + 4\tau \sqrt{\constantBandeira \log(2s/\delta)}
\end{align*}
with probability at least $1 - \delta/2$,
which implies
\begin{align*}
    \Vert \avgLessT \Vert_{\op} \le \Vert  \bbE \avgLessT \Vert_{\op} + 6 \sqrt{s \rho_n / m^2} + 4\tau \sqrt{\constantBandeira \log(2s/\delta)}
\end{align*}
with probability at least $1 - \delta/2$. Using
\begin{align*}
    0 = \bbE \overline{\bE}_0 = \bbE \avgLessT + \bbE \avgLarge,
\end{align*}
we obtain
\begin{align*}
    \Vert \avgLessT \Vert_{\op} \le \Vert \bbE \avgLarge \Vert_{\op} + 6 \sqrt{s \rho_n / m^2} + 4 \tau \sqrt{\constantBandeira \log(2s/\delta)}
\end{align*}
with probability at least $1 - \delta/2$.

\noindent \textbf{Step 3. Bounding $ \Vert \bbE \avgLarge \Vert_{\op}$ and $\Vert \avgLarge \Vert_{\op}$.} We start with bounding $\Vert \bbE \avgLarge \Vert_{\op}$. We have
\begin{align*}
    \Vert \bbE \avgLarge \Vert_{\op} & \le s \max_{a, b} \bbE |\avgMatrix_{ab}| \le s \max_{a, b} \bbP(|\avgMatrix_{ab}| > \tau) \le \frac{1}{m^2}
\end{align*}
due to~\eqref{eq: large probability bound for matrix concentration}.

Then, we bound $\Vert \avgLarge \Vert_{\op}$. We have
\begin{align*}
    \Vert \avgLarge \Vert_{\op} & \le \max_a \sum_{b = 1}^s |(\avgLarge)_{ab}| \le \max_{a} \sum_{b = 1}^s \indicator{|(\overline{\bE}_0)_{ab}| > \tau}.
\end{align*}
The union bound implies
\begin{align*}
    \bbP \left (\max_{b \in [s]} \indicator{|\avgMatrix_{ab}| > \tau} = 1 \right ) \le \sum_{b = 1}^s \bbP \left ( |\avgMatrix_{ab}| > \tau \right ) \le \frac{\delta}{2s m^2}.
\end{align*}
In particular, 
\begin{align*}
    \max_{a} \sum_{b = 1}^{s - 1} \indicator{|\avgMatrix_{ab}| > \tau} = 0
\end{align*}
with probability at least $1 - \delta/2$. Therefore, with probability at least $1 - \delta$, we have
\begin{align*}
    \Vert \avgMatrix\Vert_{\op} & \le \frac{1}{m^2} + 6 \sqrt{\frac{s \rho_n}{m^2}} + 4 \tau \sqrt{\constantBandeira \log(2s/\delta)} \\
    & \le  6 \sqrt{\frac{s \rho_n}{m^2}} + 4 \left ( \sqrt{\frac{32 \rho_n \log(8 s^2 m^2/\delta)}{m^2}} + \frac{12 \log(8 s^2 m^2/\delta)}{m^2}\right ) \sqrt{\constantBandeira \log(2s/\delta)}.
\end{align*}
Since the latter bound holds for any realization of $\xi_1, \ldots, \xi_n$, it also holds unconditionally.

\section{Proof of Theorem~\ref{theorem: holder convergence rate}}
\label{appendix: proof of Holder convergence rate}

i) \textbf{Upper bound on the convergence rate.} We build the proof on that of Proposition~3.5 in~\citep{klopp2017}. The difference is that we bound the error in probability instead of bounding in expectation. We first bound the difference in the Frobenius norm and then, using that, bound the difference in the spectral norm. 


By the definition of the empirical graphon 
\begin{align*}
	W_{\bQ}(x,y) = \rho_n\sum_{i j} W(\xi_{(i)}, \xi_{(j)}) \cdot I\{x \in [(i - 1)/n, i/n]\} \cdot I\{y \in [(j - 1)/n, j/n)\}.
\end{align*}
Define $\Delta_{ij} = [(i-1)/n, i/n) \times [(j - 1)/n, j/n)$. Then for any $(x, y) \in \Delta_{ij}$ we have
\begin{align*}
	|W(x, y) - W(\xi_{(i)}, \xi_{(j)})| \le \left |W(x, y) - W \left ( i/n, j/n \right ) \right | + \left |  W \left ( i/n, j/n \right ) - W(\xi_{(i)}, \xi_{(j)})\right |.
\end{align*}
Since the graphon $W$ is $(H, \alpha)$-H\"older there exists a constant $C(H, \alpha)$ such that
\begin{align*}
	\left |W(x, y) - W \left ( i/n, j/n \right ) \right | & \le C(H, \alpha) \left ( |x - i/n|^{\alpha \land 1} + |y - j/n|^{\alpha \land 1} \right ) \\
	& \le 2 C(H, \alpha) n^{-(\alpha \land 1)},
\end{align*}
where $\alpha \land 1=\min\{\alpha,1\}.$
Analogously, we have
\begin{align*}
	\left |  W \left ( i/n, j/n \right ) - W(\xi_{(i)}, \xi_{(j)})\right | \le C(H, \alpha) \left (|\xi_{(i)} - i/n|^{\alpha \land 1} +  |\xi_{(j)} - j/n|^{\alpha \land 1}\right ).
\end{align*}
Thus, we obtain
\begin{align}
    \|\rho_n \bbW - \bbW_{\bQ}\|_F^2 &= \int_{[0, 1]^2} (\rho_n W(x, y) - W_{\bQ}(x, y))^2 dx dy \nonumber \\
    & \le 2\rho_n^2 \left(4 C^2(H, \alpha) n^{-2 (\alpha \land 1)} \nonumber\right. \\
	&\left. \quad \quad + C^2(H, \alpha) \sum_{i j} \frac{1}{n^2} \left (|\xi_{(i)} - i/n|^{\alpha \land 1} +  |\xi_{(j)} - j/n|^{\alpha \land 1}\right )^2\right) \nonumber \\
	& \le 2\rho_n^2 \left(4 C^2(H, \alpha) n^{-2 (\alpha \wedge 1)} + \frac{4 C^2(H, \alpha)}{n} \sum_{i = 1}^n |\xi_{(i)} - i/n|^{2 (\alpha \wedge 1)}\right). \label{eq: L2 decomposition for Holder graphon}
\end{align}
Due to Corollary~\ref{corollary: order statistics unifrom concentration}, we get that with probability at least $1 - \delta/4$
\begin{align*}
	\frac{1}{n} \sum_{i = 1}^n \left | \xi_{(i)} - i/n \right |^{2 (\alpha \wedge 1)}  \le \frac{1}{2} \left ( \frac{\log (8/\delta)}{n} \right )^{(\alpha \wedge 1)}.
\end{align*}
Substituting the above into~\eqref{eq: L2 decomposition for Holder graphon}, we get
\begin{align*}
    \Vert \rho_nW - W_{\bQ} \Vert_{F}^2 & \le 2\rho_n^2 \left(4 C^2(H, \alpha) n^{-2 (\alpha \wedge 1)} + 2 C^2(H, \alpha) \left ( \frac{\log (8/\delta)}{n} \right )^{(\alpha \wedge 1)}\right)  \\
    & \le 16\rho_n^2 C^2(H, \alpha) \left ( \frac{\log \frac{8}{\delta}}{n} \right )^{(\alpha \wedge 1)},
\end{align*}
which implies
 $   \Vert \rho_n\bbW - \bbW_{\bQ} \Vert_{F} \leq4\rho_n C(H, \alpha) \left ( \frac{\log \frac{8}{\delta}}{n} \right )^{\frac{(\alpha \wedge 1)}{2}}$
with probability at least $1 - \delta/4$. Since the Frobenius norm bounds the spectral norm, the above implies
\begin{align}
	\opnorm{\rho_n \bbW - \bbW_{\bQ}} \le 4 \rho_n C(H, \alpha) \left ( \frac{\log (8/\delta)}{n} \right )^{\frac{(\alpha \wedge 1)}{2}}. \label{eq: holder matrix of connection probabilities difference}
\end{align}

Then, we bound $\Vert \bQ - \bQhat^*_{\beta}\Vert_{\op}$. If $n^{2 \beta - 1} > \rho_n$, then, due to Theorem~\ref{theorem: Holder averaging}, we have
\begin{align*}
    \Vert \bQ - \bQhat^*_{\beta} \Vert_{\op} & \le \lambda' + \Vert \cS_m(\bA) - \bQ \Vert_{\op} \\
    & \le \max \left \{\rho_n n \cdot \left (\frac{\log (16/\delta)}{n} \right )^{\beta/2}, \rho_n^{\frac{\beta + 1}{2 \beta + 1}} n^{1/(2 \beta + 1)} \right \} \\
    & \qquad +  C_1 \rho_n n \cdot \left ( \frac{\log(16/\delta)}{n} \right )^{\beta/2}  + C_2\rho_n^{\frac{\beta + 1}{2 \beta + 1}} n^{1/(2 \beta + 1)} \\
    & \le (C_1 + 1) \rho_n n \cdot \left ( \frac{\log(16/\delta)}{n} \right )^{\beta/2}  + (C_2 + 1)\rho_n^{\frac{\beta + 1}{2 \beta + 1}} n^{1/(2 \beta + 1)}
\end{align*}
with probability at least $1 - \delta/4$. 
If $n^{2 \beta - 1} \le \rho_n$, then $\bQhat^*_{\beta} = \widehat{\bQ}_{\lambda}$ and
\begin{align*}
    \Vert \bQ - \bQhat^*_{\beta} \Vert_{\op} & \le 16 \sqrt{n \rho_n} + \sqrt{\constantBandeira \log \frac{4n}{\delta}} 
\end{align*}
due to Theorem~\ref{theorem: UHTE performance} with probability at least $1 - \delta/4$. Using the assumption $n \rho_n \ge \log(4n/\delta)$, we get
\begin{align*}
    \Vert \bQ - \bQhat^*_{\beta} \Vert_{\op} & \le (16 + \constantBandeira^{1/2}) \sqrt{ n \rho_n}.
\end{align*}

Finally, note that
\begin{align*}
    \min \left \{ \sqrt{n \rho_n},  \rho_n^{\frac{\beta + 1}{2 \beta + 1}} n^{1/(2 \beta + 1)}\right \} = \begin{cases}
        \sqrt{n \rho_n} & \text{ if } n^{2 \beta - 1} \le \rho_n, \\
        \rho_n^{\frac{\beta + 1}{2 \beta + 1}} n^{1/(2 \beta + 1)} & \text{ if } n^{2 \beta - 1} > \rho_n,
    \end{cases}
\end{align*}
which, due to Lemma~\ref{lemma: empirical graphon norm}, implies
\begin{align}
    \Vert \rho_n \bbW - \bbW_{\bQhat^*_{\beta}} \Vert_{\op}&  \le \Vert \rho_n \bbW - \bbW_{\bQhat^*_{\beta}} \Vert_{\op} + \frac{\Vert \bQ - \bQhat^*_{\beta} \Vert_{\op}}{n} \nonumber \\
    & \le C \rho_n n \cdot \left ( \frac{\log(16/\delta)}{n} \right )^{\beta/2}  + C' \min \left \{\sqrt{\frac{\rho_n}{n}},  \left (\rho_n^{\beta + 1}/n^{2 \beta} \right)^{1/(2 \beta + 1)}\right \}
\label{eq: final bound in the proof of Holder graphon rates}
\end{align}
with probability at least $1 - \delta/2$ for some constants $C, C' > 0$ depending only on $H$ and $\alpha$.

ii) \textbf{Upper bound on the rank.} We start by recalling the following result from approximation theory. 
\begin{lemma}
\label{lemma: Holder function approximation}
	Fix $k \in \mathbb{N}$. Define $\Delta_{ij} = [(i - 1)/k, i/k) \times [(j - 1)/k, j/k)$ for $1 \le i, j \le k$. Define $\ell = \lfloor \alpha \rfloor$. Then for any $(H, \alpha)$-H\"older graphon $W$ there exists a function $P_k: [0, 1]^2 \to \bbR$ such that
	\begin{enumerate}
		\item for any $\Delta_{i j}$, the function $P_k|_{\Delta_{ij}}$ is polynomial of degree at most $\ell$,
		\item the error of approximation satisfies
		\begin{align*}
			\sup_{x, y} |W(x, y) - P_k(x, y)| \le 2 H k^{-\alpha}.
		\end{align*}
	\end{enumerate}
\end{lemma}
For the proof, see, for example, Lemma 3 in~\citep{xu2017}.

First, we claim that $P_k(x, y)$ has rank at most $k \cdot (\ell + 1)^2$. Indeed, it can be represented as a sum of functions
\begin{align*}
	P_k(x, y) = \sum_{a, b, a + b \le \ell} W_{\bS^{ab}}(x,y) x^a y^b,
\end{align*}
where functions $W_{\bS^{ab}}$ are constants on each $\Delta_{ij}$, i.e. $W_{\bS^{ab}}|_{\Delta_{ij}} = \bS^{ab}_{ij}$ for some matrix $\bS^{ab} \in \bbR^{k \times k}$. Thus, each $W_{\bS^{ab}}$ has rank at most $k$, and, consequently, $P_k$ has rank at most $k \cdot (\ell + 1)^2$ (see Appendix \ref{sec:rank_function}). \\

Second, we state that $\bP_k = \{P_k(\xi_{(i)}, \xi_{(j)})\}_{ij}$ has the same rank as $P_k$. Indeed, if $P_k(x, y)$ can be represented as a sum $\sum_{r = 1}^{k (\ell + 1)^2} \phi_r(x) \psi_r(y)$ of functions of rank $1$, then $\bP_k$ has decomposition $\sum_{r = 1}^{k (\ell + 1)^2} \bv_r \bg_r^\top$, where $\bv_r = (\phi_r(\xi_{(i)}))_{i = 1}^n$, $\bg_r = (\psi_r(\xi_{(i)}))_{i = 1}^n$.

We aim to show that for some constant $C(H,\alpha)$, we have $\sigma_k(\bQ) \le C(H, \alpha) n \rho_n k^{-(\alpha + 1/2)}$ for all $k \in \bbN$. First, we bound $\sigma_{k} (\bQ)$ for each $k \ge (\ell + 1)^3$. Define $k_1 = \left \lceil \frac{k}{(\ell + 1)^3} \right \rceil$ and $k_2 = k - k_1 (\ell + 1)^2$.
Eckart - Young - Mirsky theorem (see Proposition~\ref{proposition: eckart-yang theorem}) for $\bA=\bQ$, $r=k_1 (\ell+1)^2$, $\bB=\rho_n \bP_{k_1}$  and
Lemma~\ref{lemma: Holder function approximation}, imply
\begin{align*}
	k_2 \sigma^2_k(\bQ) \le \sum_{r = k_1 (l + 1)^2 + 1}^n \sigma_r^2(\bQ) \le  \Vert \bQ - \rho_n \bP_{k_1} \Vert_{\mathrm{F}}^2 \le (2 H n \rho_n k_1^{-\alpha})^2 ,
\end{align*}
and, consequently, $\sigma_k(\bQ) \le 2 H n \rho_n k_2^{-1/2} k_1^{-\alpha}\leq C'(H, \alpha) n \rho_n k^{-(\alpha + 1/2)}$ where $C'(H, \alpha)$ is a constant depending only on $H$ and $\alpha$. 
On the other hand, as $\Vert \bQ \Vert_{\op} \le n \rho_n$ taking $C'(H, \alpha)$ large enough the bound holds also for  $k<(l+1)$.

Finally, we apply Proposition~\ref{proposition: singular value rank}. If $n^{2\beta - 1} > \rho_n$, then $\widehat{\bQ}^*_{\beta}$ is obtained from $\cS_m(\bA)$ by hard-thresholding with threshold $\lambda'$. Hence, 
\begin{align*}
    \rank(\widehat{\bQ}^*_{\beta}) & \le \max \{k: \sigma_k(\cS_m(\bA)) > \lambda' \} \\
    & \le \max \{k: \sigma_k(\bQ) > \lambda' - \Vert \cS_m(\bA) - \bQ \Vert_{\op} \} \\
    & \le  \max \left \{k: C(H, \alpha) n \rho_n k^{-(\alpha + 1/2)} >  C_1\rho_n n \cdot \left ( \frac{\log(8/\delta)}{n} \right )^{\beta/2}  + C_2\rho_n^{\frac{\beta + 1}{2 \beta + 1}} n^{1/(2 \beta + 1)} \right \},
\end{align*}
where the last inequality holds with probability at least $1 - \delta/2$ due to the bound on $\Vert \cS_m(\bA) - \bQ \Vert_{\op}$ obtained in Theorem~\ref{theorem: Holder averaging}. Hence, we get
\begin{align*}
    \rank(\widehat{\bQ}^*_{\beta}) & \le C'(H, \alpha) \left (  \left (\frac{\log \frac{8}{\delta}}{n} \right )^{\beta/2} + \left ( \frac{1}{n^2 \rho_n}\right )^{\frac{\beta}{2 \beta + 1}} \right )^{- \frac{1}{\alpha + 1/2}} \\
    & \le  C'(H, \alpha) \cdot \min \Big \{\frac{n^{\beta/(2 \alpha + 1)}}{\log^{\beta/(2 \alpha + 1)}(8/\delta)}, \; (n^2 \rho_n)^{2\beta (2 \beta + 1)^{-1} / (2\alpha + 1)} \Big \}
\end{align*}
with probability at least $1 - \delta/2$. Combined with~\eqref{eq: final bound in the proof of Holder graphon rates}, the above implies the result of the theorem in the case $n^{2\beta - 1} > \rho_n$.

If $n^{2 \beta - 1} \le \rho_n$, then $\widehat{\bQ}^*_{\beta} = \widehat{\bQ}_\lambda$ and, due to Proposition~\ref{proposition: singular value rank}, we have
\begin{align*}
    \rank(\widehat{\bQ}_\lambda) \le \max \{k: \sigma_k(\bQ) > \lambda - \Vert \bQ - \bA \Vert_{\op} \}.
\end{align*}
Bounding $\Vert \bQ - \bA \Vert_{\op}$ by~\eqref{eq: Bandeira bound for Q minus A}, we obtain
\begin{align*}
    \rank(\widehat{\bQ}_\lambda) & \le \max \{k: \sigma_k(\bQ) > \sqrt{n \rho_n} \} \\
        & \le \max \{k: C(H, \alpha) n \rho_n k^{-(\alpha + 1/2)} > \sqrt{n \rho_n} \} \\
        & \le C''(H, \alpha) \left ( n \rho_n\right )^{\frac{1}{2 \alpha + 1}}
\end{align*}
with probability at least $1 - \delta/2$. Combined with~\eqref{eq: final bound in the proof of Holder graphon rates}, the above concludes the proof.

\section{Proof of Theorem~\ref{theorem: holder graphon estimation lower bound}}
\label{appendix: proof of holder graphon estimation lower bound}

\begin{proof}
    To prove the theorem, it is sufficient to prove the following two lemmas.
    \begin{lemma}
    \label{lemma: holder graphon estimation lower bound 1}
        Let $\Sigma(H, \alpha)$ be the class of $(H, \alpha)$-H\"older graphons for $\alpha \in (0; 1]$. Suppose that $n \ge (H^{1/\alpha} + 1)^2$. Then, there exists a finite set $\Sigma_1 \subset \Sigma(H, \alpha)$ such that for any estimator $\widehat{\bbW}$, we have 
        \begin{align*}           
            \inf_{\widehat{\bbW}} \max_{W \in \Sigma_1} \bbP_{\bA \sim \rho_n W} \left (   \Vert \rho_n \bbW -  \widehat{\bbW}(\bA) \Vert_{\op} \ge \frac{ H \rho_n}{2^{13} n^{\alpha/2}} \right ) \ge \frac{1}{6}.
        \end{align*}
    \end{lemma}

    \begin{lemma}
    \label{lemma: holder graphon estimation lower bound 2}
        Let $\Sigma(H, \alpha)$ be the class of $(H, \alpha)$-H\"older graphons for $\alpha \in (0; 1]$. Suppose that $n \rho_n \ge H^{1/(2 \alpha)} / 6$ and $H^2 n^2 \rho_n \ge 2^{20}$.  Then, there exists a finite set $\Sigma_2 \subset \Sigma(H, \alpha)$ such that for any estimator $\widehat{\bbW}$, we have
        \begin{align*}
            \inf_{\widehat{\bbW}} \max_{W \in \Sigma_2} \bbP_{\bA \sim \rho_n W} \left ( \Vert \rho_n \bbW - \widehat{ \bbW     }(\bA) \Vert_{\op} \ge \frac{1}{30} \left ( \frac{H \rho_n^{1 + \alpha}}{n^{2 \alpha}} \right )^{\frac{1}{(1 + 2 \alpha)}} \right ) \ge \frac{1}{6}.
        \end{align*}
    \end{lemma}
    Then, the lemmas imply
    \begin{align*}
       \inf_{\widehat{\bbW}} \max_{W \in \Sigma_1 \cup \Sigma_2} \bbP_{\bA \sim \rho_n W} \left ( \Vert \rho_n \bbW - \widehat{ \bbW     }(\bA) \Vert_{\op} \ge \max \left \{  \frac{1}{30} \left ( \frac{H \rho_n^{1 + \alpha}}{n^{2 \alpha}}\right )^{\frac{1}{(1 + 2 \alpha)}}, \frac{ H \rho_n}{2^{14} n^{\alpha/2}} \right \}  \right ) \ge \frac{1}{6}.
    \end{align*}
    It yields
    \begin{align*}
       \inf_{\widehat{\bbW}} \max_{W \in \Sigma_1 \cup \Sigma_2} \bbP_{\bA \sim \rho_n W} \left ( \Vert \rho_n \bbW - \widehat{ \bbW     }(\bA) \Vert_{\op} \ge \frac{1}{60} \left ( \frac{H \rho_n^{1 + \alpha}}{n^{2 \alpha}}\right )^{\frac{1}{(1 + 2 \alpha)}} + \frac{ H \rho_n}{2^{15} n^{\alpha/2}}   \right ) \ge \frac{1}{6}. & \qedhere
    \end{align*}
\end{proof}

\subsection{Proof of Lemma~\ref{lemma: holder graphon estimation lower bound 1}}

\begin{proof}
	\textbf{Step 1. Constructing hypotheses.}We start by the following standard construction. Let $\eta(x)$ be a function defined by
    \begin{align}
    \label{eq: eta definition}
        \eta(x) = (1 - 2|x|) \indicator{|x| \le 1/2}.
    \end{align}
	Fix $H>0$ and define a H\"older graphon $W^{(0)}$ by
	\begin{align*}
	& W^{(0)}(x,y) := \frac12 \\
    & \quad + \frac{H}{4} \sum_{a = 1}^{\ts} \sum_{b =1}^{\ts } \ts^{-\alpha} \eta \left (x\ts - a + \frac{1}{2} \right ) \eta \left ( y \ts - b + \frac{1}{2} \right ) \indicator{x \in \left  [\frac{(a - 1)}{\ts}; \frac{a}{\ts} \right ], y \in \left [\frac{(b - 1)}{\ts}; \frac{b}{\ts}\right ]},
	\end{align*}
    where $\ts \ge \max \{2, H^{1/\alpha}\}$ is an integer divisible by $4$ to be chosen later.
    The following proposition shows that $W^{(0)}(\cdot, \cdot)$ is $(H/2, \alpha)$-H\"older. 
    \begin{prop}
    \label{proposition: Holder graphons construction}
    For any integer $\ts$, any $G > 0$, any weights $\omega_{ab} \in [-1; 1]$, $1 \le a, b \le \ts$, and any $\alpha \in (0; 1]$, a graphon
    \begin{align*}
        & W(x, y) = \frac{1}{2} \\
        & \quad + \frac{G }{2 \ts^{\alpha}} \sum_{a = 1}^{\ts} \sum_{b = 1}^{\ts} \omega_{ab}  \eta \left (x\ts - a + \frac{1}{2} \right ) \eta \left ( y \ts - b + \frac{1}{2} \right ) \indicator{x \in \left  [\frac{(a - 1)}{\ts}; \frac{a}{\ts} \right ], y \in \left [\frac{(b - 1)}{\ts}; \frac{b}{\ts}\right ]},
    \end{align*}
    is $(G, \alpha)$-H\"older.
    \end{prop}
    The proof of Proposition~\ref{proposition: Holder graphons construction} is given in Section~\ref{section: proof of holder graphon construction}.

	Next define a smooth increasing diffeomorphism $\tau_t:[0,1]\to[0,1]$ by
	\[
	\tau_t(x) := x + t\,\varphi(x),
	\qquad \text{where}\quad \varphi(x):=x(1-x).
	\]
	For $|t|\le 1/2$, we have $\tau_t(0)=0$, $\tau_t(1)=1$, and
	$\tau_t'(x)=1+t(1-2x)\in[1-|t|,\,1+|t|]$, hence $\tau_t$ is strictly increasing.
	Define the second graphon as the reparameterization
	\[
	W^{(1)}(x,y) := W^{(0)}(\tau_t(x),\tau_t(y)).
	\]
    We claim that $W^{(1)}$ is $(H, \alpha)$-H\"older as well. Indeed, since $W^{0}$ is $(H/2, \alpha)$-H\"older, for $x,x',y,y'\in[0,1]$, we have
    \begin{align*}
        |W^{(1)}(x,y) - W^{(1)}(x',y')| & \le \frac{H}{2}  \left (|\tau_t(x) - \tau_t(x')|^\alpha + |\tau_t(y) - \tau_t(y')|^\alpha \right ) \\
        & \le \frac{H (1 + |t|)^\alpha}{2}  \left (|x - x'|^\alpha + |y - y'|^\alpha \right ) \\
         & \le H  \left (|x - x'|^\alpha + |y - y'|^\alpha \right ),
    \end{align*}
    where we used $\tau'_t(x)\in[1-|t|,\,1+|t|]$ for $x\in[0,1]$.

	Consider the difference $\Delta W := W^{(1)}-W^{(0)}$. We claim that $\Vert \Delta \bbW \Vert_{\op} \ge 2^{-13} H t^{\alpha}$, provided $\ts^{-1}/ 2 \le t \le \ts^{-1}$. Since $\Delta W(x, y)$ can be represented as
    \begin{align}
        \Delta W(x, y) = \frac{H \ts^{-\alpha}}{4} \psi(\tau_t(x)) \psi(\tau_t(y)) - \frac{H \ts^{-\alpha}}{4} \psi(x) \psi(y)
    \end{align}
    for the function
    \begin{align*}
        \psi(x) := \sum_{a = 1}^{\ts}  \eta \left (x\ts - a + \frac{1}{2} \right ) \indicator{x \in \left  [\frac{(a - 1)}{\ts}; \frac{a}{\ts} \right ]},
    \end{align*}
    the graphon $\Delta W(x, y)$ has rank 2, so
    \begin{align*}
        \Vert \Delta \bbW \Vert_{\op} \ge \frac{1}{\sqrt{2}}\Vert \Delta W \Vert_{\mathrm{HS}} = \frac{H\ts^{-\alpha}}{4\sqrt{2}} \sqrt{\int_0^1 \int_0^1 \left ( \psi(\tau_t(x)) \psi(\tau_t(y)) -  \psi(x) \psi(y) \right )^2 dx dy}. 
    \end{align*}
    To bound the latter integral from below, we represent it as 
    {\begin{align*}
        &\sum_{a = 1}^{\ts} \sum_{b = 1}^{\ts} \int_{\frac{a - 1}{\ts}}^{\frac{a}{\ts}} \int_{\frac{b - 1}{\ts}}^{\frac{b}{\ts}} \left ( \psi(\tau_t(x)) \psi(\tau_t(y)) -  \eta(\{\ts x\} - 1/2) \eta(\{\ts y\} - 1/2) \right )^2 dx dy \\
        & \quad \ge \sum_{a = \ts/4 + 1}^{3 \ts/4} \sum_{b = \ts/4 + 1}^{3 \ts/4} \int_{\frac{a - 1}{\ts}}^{\frac{a - 3/4}{\ts}} \int_{\frac{b -1}{\ts}}^{\frac{b- 3/4}{\ts}} \left ( \psi(\tau_t(x)) \psi(\tau_t(y)) -   \eta(\{\ts x\} - 1/2) \eta(\{\ts y\} - 1/2) \right )^2 dx dy
    \end{align*}}
    If $a \in [\ts/4 + 1; 3 \ts/4]$, $x \in [(a - 1)/\ts; (a - 3/4)/\ts]$, and $t \in (0; \ts^{-1})$, then
    \begin{align*}
        \phi(x) \in [3/16; 1/4] \quad \text{and} \quad \psi(\tau_t(x)) =   1 - 2 \left ( \frac{1}{2} - \{\ts x \}- \ts t \phi(x)\right ) \ge \eta(\{\ts x\} - 1/2) = \psi(x).
    \end{align*}
    Therefore, the above double sum is bounded from below by
    \begin{align*}
        & \sum_{a = \ts/4 + 1}^{3 \ts/4} \sum_{b = \ts/4 + 1}^{3 \ts/4} \int_{\frac{a - 1}{\ts}}^{\frac{a - 3/4}{\ts}} \int_{\frac{b -1}{\ts}}^{\frac{b- 3/4}{\ts}} \left ( \psi(\tau_t(x)) -   \eta(\{\ts x\} - 1/2)  \right )^2 \eta^2(\{\ts y\} - 1/2) dx dy \\
        & \qquad = \sum_{a = \ts/4 + 1}^{3 \ts/4}  \int_{\frac{a - 1}{\ts}}^{\frac{a - 3/4}{\ts}} \left ( \psi(\tau_t(x)) -   \eta(\{\ts x\} - 1/2)  \right )^2 dx \cdot \sum_{b = \ts/4 + 1}^{3 \ts/4}\int_{\frac{b -1}{\ts}}^{\frac{b- 3/4}{\ts}} \eta(\{\ts y\} - 1/2)^2 dy \\
        & \qquad = \sum_{a = \ts/4 + 1}^{3 \ts/4}  \int_{\frac{a - 1}{\ts}}^{\frac{a - 3/4}{\ts}} \left (2\ts t \phi(x) \right )^2 dx \cdot \sum_{b = \ts/4 + 1}^{3 \ts/4}\int_{\frac{b -1}{\ts}}^{\frac{b- 3/4}{\ts}} (2 \{\ts y \})^2 dy \\
        & \qquad \ge \frac{9 t^2 \ts^2}{2^{11}} \cdot \frac{\ts/2}{48 \ts} = \frac{3 t^2 \ts^2}{2^{16}}.
    \end{align*}
    Then, it implies
    \begin{align*}
        \Vert \Delta \bbW \Vert_{\op} \ge \frac{H t \ts^{1-\alpha} }{2^{11}}.
    \end{align*}
    Then, if $t \ge \ts^{-1}/2$, then, using $\ts^{-1} \ge t$, we get $\Vert \Delta \bbW \Vert_{\op} \ge 2^{-12} H t^{\alpha}$, as desired.

	\medskip
	\noindent \textbf{Step 2. Bounding the KL divergence.} Let $\bbP_0$ be a distribution of $\bA$ under the sampling model~\eqref{eq: Q sampling}-\eqref{eq: A sampling} with $W=W^{(0)}$, and let $\bbP_1$ be the distribution of $\bA$ under the same sampling model with $W=W^{(1)}$. To bound $\KL{\bbP_1}{\bbP_0}$, we use the following lemma, which is a consequence of the data-processing inequality for KL divergence. For the proof, see Section~\ref{section: proof of shifted graphon KL-divergence}.
    \begin{lemma}
    \label{lemma: shifted graphon KL-divergence}
    Set $\tau_t(x) = x + t x(1-x)$ for $t \in [-1/2; 1/2]$. Let $W(\cdot, \cdot)$ be an arbitrary graphon and define $W_t(x, y) := W(\tau_t(x), \tau_t(y))$. Then, for any $\rho_n \in (0; 1]$, we have
    \begin{align*}
        \KL{P_{W_t}}{P_W} \le 4 n t^2,
    \end{align*}
    where $P_W$ and $P_{W_t}$ are the distributions of $\bA$ under the sampling model~\eqref{eq: Q sampling}-\eqref{eq: A sampling} with graphons $W$ and $W_t$ respectively.
    \end{lemma}

   We choose $t = 1/4 \sqrt{n}$ and set $\ts =  4 \lfloor \sqrt{n} \rfloor$. Since $n \ge (1 + H^{1/\alpha})^2$, we have $\ts \ge \max \{2 , H^{1/\alpha}\}$, so the above construction is valid. Then, we have $\KL{\bbP_1}{\bbP_0} \le 1/4$.

	\medskip
	\noindent \textbf{Step 3. Applying Theorem~\ref{theorem:minimax bounds}.} Since $\Vert \rho_n \bbW^{(1)} - \rho_n \bbW^{(0)} \Vert_{\op} = \rho_n \Vert \Delta \bbW \Vert_{\op} \ge 2^{-12} H t^{\alpha} \rho_n$, by Theorem~\ref{theorem:minimax bounds}, we have
    \begin{align*}
        \inf_{\widehat{\bbW}} \max_{W \in \{W^{(0)}, W^{(1)}\}} \bbP_{\bA \sim \rho_n W} \left (   \Vert \rho_n \bbW -  \widehat{\bbW}(\bA) \Vert_{\op} \ge 2^{-13} H t^{\alpha} \rho_n \right ) \ge 1/6,
    \end{align*}
    which completes the proof.
\end{proof}

\subsubsection{Proof of Proposition~\ref{proposition: Holder graphons construction}}
\label{section: proof of holder graphon construction}

\begin{proof} The proof follows that of Lemma 2.1 from~\citep{gao2014rate}. For $x, x', y, y'$ consider two cases. If $|x - x'| \ge \ts^{-1}$ or $|y - y'| \ge \ts^{-1}$, we have
\begin{align*}
    W(x, y) - W(x', y') \le G \ts^{-\alpha} \le  G (|x - x'|^{\alpha} + |y - y'|^{\alpha}).
\end{align*}
Otherwise, let $a, a', b, b' \in [\ts]$ be such that $x \in [(a - 1)/\ts; a/\ts]$, $x' \in [(a' - 1)/\ts; a'/\ts]$, $y \in [(b - 1)/\ts; b/\ts]$, and $y' \in [(b' - 1)/\ts; b'/\ts]$. Then, by the triangle inequality, we have
\begin{align}
    W (x, y) - W(x', y') & \le \frac{G}{2} \ts^{-\alpha} |\omega_{ab} \eta\left (\{\ts x\} - \frac{1}{2} \right ) - \omega_{a' b'}\eta \left (\{\ts x'\} - \frac{1}{2} \right )| |\eta(\{\ts y\} - 1/2)| \nonumber \\
    & \qquad +  \frac{G}{2} \ts^{-\alpha} |\omega_{a'b'}| |\eta \left ( \{\ts x' \} - 1/2\right )| |\eta \left (\{\ts y\} - \frac{1}{2}\right ) - \eta \left  (\{\ts y'\} - \frac{1}{2} \right )|, \label{eq: triangle inequality in Holder graphon construction}
\end{align}
where $\{\cdot\}$ stands for the fractional part. To bound
\begin{align*}
    \left |\omega_{ab} \eta\left (\{\ts x\} - \frac{1}{2} \right ) - \omega_{a' b'}\eta \left (\{\ts x'\} - \frac{1}{2} \right ) \right |,
\end{align*}
consider two cases. If $a = a'$, then 
\begin{align*}
    \left |\omega_{ab} \eta\left (\{\ts x\} - \frac{1}{2} \right ) - \omega_{a' b'}\eta \left (\{\ts x'\} - \frac{1}{2} \right ) \right | \le 2 \ts |\omega_{ab}| |x - x'| \le 2 \ts |x - x'|.
\end{align*}
If $a \neq a'$, then without loss of generality, we can assume that $a < a'$, so, using $|x' - x| \le \ts^{-1}$, we deduce $a' = a + 1$. Then, we have
\begin{align*}
    \left |\omega_{ab} \eta\left (\{\ts x\} - \frac{1}{2} \right ) - \omega_{a' b'}\eta \left (\{\ts x'\} - \frac{1}{2} \right ) \right | & \le \left |\omega_{ab} \eta\left (\{\ts x\} - \frac{1}{2} \right ) \right | + \left |\omega_{a' b'}\eta \left (\{\ts x'\} - \frac{1}{2} \right ) \right |  \\
    & \le 2 \ts \left ( \frac{a}{\ts} - x + x' - \frac{a}{\ts} \right ) = 2 \ts |x - x'|.
\end{align*}
It implies
\begin{align*}
    \left |\omega_{ab} \eta\left (\{\ts x\} - \frac{1}{2} \right ) - \omega_{a' b'}\eta \left (\{\ts x'\} - \frac{1}{2} \right ) \right | \le 2 \ts |x - x'|.
\end{align*}
Similarly, we have
\begin{align*}
   |\omega_{a'b'}| |\eta \left ( \{\ts x' \} - 1/2\right )| |\eta \left (\{\ts y\} - \frac{1}{2}\right ) - \eta \left  (\{\ts y'\} - \frac{1}{2} \right )| \le 2 \ts |y - y'|.
\end{align*}
Substituting the above bounds into~\eqref{eq: triangle inequality in Holder graphon construction}, we get
\begin{align*}
     W (x, y) - W(x', y') & \le G \ts^{1-\alpha} (|x - x'| + |y - y'|) \\
     & \le G (|x - x'|^\alpha + |y - y'|^\alpha) \cdot \max \left \{ \ts |x - x'|, \ts |y - y'| \right \}^{1 - \alpha} \\
     & \le G (|x - x'|^\alpha + |y - y'|^\alpha),
\end{align*}
so $W$ is indeed $(G, \alpha)$-H\"older.
\end{proof}

\subsubsection{Proof of Lemma~\ref{lemma: shifted graphon KL-divergence}}
\label{section: proof of shifted graphon KL-divergence}

\begin{proof}
    Under $W_t(\cdot, \cdot)$, the sampling model~\eqref{eq: Q sampling}-\eqref{eq: A sampling} can be rewritten as follows:
	sample $\xi_1,\dots,\xi_n\stackrel{\text{i.i.d.}}{\sim}\mathrm{Uniform}[0,1]$ and define
	$\zeta_i := \tau_t(\xi_i)$. Since $\tau_t$ is increasing, the order is preserved, i.e.,
	$\zeta_{(i)}=\tau_t(\xi_{(i)})$ for every $i$.
	Then the edge probabilities satisfy
	\[
	\bQ^{(t)}_{ij} = \rho_n W_{t}(\xi_{(i)},\xi_{(j)})
	= \rho_n W(\tau_t(\xi_{(i)}),\tau_t(\xi_{(j)}))
	= \rho_n W(\zeta_{(i)},\zeta_{(j)}).
	\]
	In other words, \emph{conditional on the ordered latent points}, the distribution of $\mathbf A$ is obtained by the
	same Markov kernel (Bernoulli edges with probabilities $\rho_n W^{(0)}(\cdot,\cdot)$); only the distribution of the
	latent points differs between the two models ($\zeta$-points versus uniform points).
	
	Formally, let $K(\cdot\,|\,\mathbf u)$ denote the conditional law of $\mathbf A$ given ordered latent points
	$\mathbf u=(u_{(1)},\dots,u_{(n)})$ under the kernel $W^{(0)}$.
	Then
	\[
	P_{W^{(0)}}(\mathbf A\in\cdot)=\int K(\cdot\,|\,\mathbf u)\,d\mu_0(\mathbf u),
	\qquad
	P_{W^{(1)}}(\mathbf A\in\cdot)=\int K(\cdot\,|\,\mathbf u)\,d\mu_t(\mathbf u),
	\]
	where $\mu_0$ is the law of $(\xi_{(1)},\dots,\xi_{(n)})$ and $\mu_t$ is the law of $(\zeta_{(1)},\dots,\zeta_{(n)})$.
	By the data-processing inequality for KL divergence (see Lemma~\ref{lem:dp_kl_markov_kernel}), we have
	\[
	\mathrm{KL}\!\left(P_{W_t} \,\|\, P_{W}\right)
	\;\le\;
	\mathrm{KL}\!\left(\mu_t \,\|\, \mu_0\right).
	\]
	Since sorting is a deterministic map, another data-processing step yields
	\[
	\mathrm{KL}(\mu_t\|\mu_0)\le \mathrm{KL}\!\left(\mathcal L(\zeta_1,\dots,\zeta_n)\,\big\|\,\mathcal L(\xi_1,\dots,\xi_n)\right)
	= n\,\mathrm{KL}(g_t\|1),
	\]
	where $g_t$ is the density of $\zeta_1=\tau_t(\xi_1)$ with respect to Lebesgue measure and $\mathcal{L}(X)$ is the law of a random variable $X$.
	Recall $\xi_1\sim \mathrm{Uniform}[0,1]$ and $\tau_t(x)=x+t x(1-x)=(1+t)x-tx^2$.
	For $|t|\le 1/2$, $\tau_t$ is strictly increasing with $\tau_t'(x)=(1+t)-2tx>0$, hence invertible and  the density satisfies
	$g_t(z)=1/\tau_t'(\tau_t^{-1}(z))$.
	Solving $z=(1+t)x-tx^2$ gives
	\[
	\tau_t^{-1}(z)=\frac{(1+t)-\sqrt{(1+t)^2-4tz}}{2t}\quad (t\neq 0),
	\]
	and thus
	\[
	\tau_t'(\tau_t^{-1}(z))=\sqrt{(1+t)^2-4tz}
	\quad\Rightarrow\quad
	g_t(z)=\frac{1}{\sqrt{(1+t)^2-4tz}}.
	\]
	Moreover, since $(1-|t|)^2\le (1+t)^2-4tz\le (1+|t|)^2$, we have
	$\frac{1}{1+|t|}\le g_t(z)\le \frac{1}{1-|t|}$.
	Therefore the $\chi^2$ divergence is bounded by
	\[
	\chi^2(g_t, 1)=\int_0^1 (g_t(z)-1)^2\,dz \le \frac{t^2}{(1-|t|)^2},
	\]
	and using $\mathrm{KL}\le \chi^2$ we obtain, for $|t|\le 1/2$,
	\[
	\mathrm{KL}(g_t, 1)\le 4t^2
	\quad\Rightarrow\quad
	\mathrm{KL}\!\left(P_{W_t} \,\|\, P_{W}\right)\le 4n t^2.
	\]
	With $t=1/(4\sqrt{n})$, this gives $\mathrm{KL}(P_{W_t}\|P_{W})\le 1/4$.
\end{proof}

\subsection{Proof of Lemma~\ref{lemma: holder graphon estimation lower bound 2}}

\begin{proof}
\textbf{Step 1. Constructing hypotheses.} 
Let $\eta(\cdot)$ be a function defined by~\eqref{eq: eta definition} and $\ts \ge H^{1/\alpha}$ be an integer to be chosen later. Then, for a vector $\zeta \in \{-1, 1\}^\ts$ define
\begin{align*}
    & W_{\zeta}(x, y) = \frac{1}{2} \\
    &   + \frac{H}{2}\sum_{a = 1}^{\ts} \sum_{b =1}^{\ts } \ts^{-\alpha} \zeta_a \zeta_b \eta \left (x\ts - a + \frac{1}{2} \right ) \eta \left ( y \ts - b + \frac{1}{2} \right ) \indicator{x \in \left  [\frac{(a - 1)}{\ts}; \frac{a}{\ts} \right ], y \in \left [\frac{(b - 1)}{\ts}; \frac{b}{\ts}\right ]}
\end{align*}
and let $W_0(x, y) = 1/2$ be a constant graphon.

We will choose vectors $\zeta_i$ such that $|\langle \zeta_i, \zeta_j \rangle| \le \ts/2$ for any $i, j$ by the following lemma, whose proof is given in Section~\ref{section: vector packing lemma}.

\begin{lemma}
\label{lemma: almost orthogonal vectors}
For $\ts \ge 18 \cdot 17$, there exists a family $\Xi \subset \{-1, 1\}^{\ts}$ such that for any distinct $\zeta, \zeta' \in \Xi$ we have $|\langle \zeta, \zeta' \rangle | \le \ts/2$ and $|\Xi| \ge e^{\ts/18}$.
\end{lemma}

In what follows, we fix a family $\Xi$ provided by Lemma~\ref{lemma: almost orthogonal vectors}. Consider a family of graphons $\{W_{\zeta} \}_{\zeta \in \Xi}$. Then, we bound $\Vert W_{\zeta} - W_{\zeta'} \Vert_{\op}$ for each distinct $\zeta, \zeta' \in \Xi$. Choose a function $f_{\zeta}$ given by
\begin{align*}
    f_{\zeta}(x) = \sum_{a = 1}^{\ts} \zeta_a \indicator{x \in [(a - 1)/\ts; a/\ts)}.
\end{align*}
Obviously, $\Vert f_{\zeta} \Vert_{L_2} = 1$. Then, we have
\begin{align*}
     (\bbW_{\zeta} - \bbW_{\zeta'}) [f_{\zeta}](x) & =  \frac{H}{2}\ts^{-\alpha} \sum_{a = 1}^\ts \eta(\ts x - a + \frac{1}{2} ) \indicator{x \in \left [\frac{a - 1}{\ts}; \frac{a}{\ts} \right ]} \sum_{b = 1}^\ts (\zeta_a \zeta_b - \zeta'_a \zeta_b') \cdot \frac{\zeta_b}{2 \ts} \\
     & = \frac{H}{4}\ts^{-\alpha} \sum_{a = 1}^\ts \eta(\ts x - a + \frac{1}{2} ) \indicator{x \in \left [\frac{a - 1}{\ts}; \frac{a}{\ts} \right ]}  \left (\zeta_a  - \zeta'_a \frac{\langle \zeta,  \zeta' \rangle}{s} \right ).
\end{align*}
Therefore, we have
\begin{align*}
    \Vert (\bbW_{\zeta} - \bbW_{\zeta'}) [f_{\zeta}] \Vert_{L_2}^2 & = \frac{H^2 \ts^{-2 \alpha}}{48 \ts} \sum_{a = 1}^\ts \left (\zeta_a - \frac{\zeta_a' \langle \zeta, \zeta' \rangle }{\ts} \right )^2 \\ 
    & = \frac{H^2 \ts^{-2 \alpha}}{48 \ts} \sum_{a = 1}^\ts \left (1 - \frac{\zeta_a \zeta_a' \langle \zeta, \zeta' \rangle }{\ts} \right )^2 \ge H^2 \ts^{-2\alpha} / 225,
\end{align*}
where we used $|\langle \zeta, \zeta' \rangle| \le \ts/2$ for any distinct $\zeta, \zeta \in \Xi$.
It implies
\begin{align}
    \Vert \bbW_{\zeta} - \bbW_{\zeta'} \Vert_{\op} \ge H \ts^{-\alpha} / 15. \label{eq: packing in operator norm}
\end{align}
Analogously, we have
\begin{align}
\Vert \bbW_{\zeta} - \bbW_{0} \Vert_{\op} \ge H \ts^{-\alpha} / 15. \label{eq: boundness of packing from 0}
\end{align}

\noindent \textbf{Step 2. Bounding $\chi^2$-divergence.} Let $\bbP_{\zeta}$ be the distribution of the matrices $\bA$ defined by the sparse graphon $\rho_n W_{\zeta}$ and $\bbP_{0}$ be defined by the sparse graphon $\rho_n W_0$. For brevity, set $\Delta = W_{\zeta} - W_0$.

Let
\begin{align*}
    L(\bA) = \frac{d \bbP_{\zeta}}{d \bbP_{0}}(\bA) = \frac{\bbP_{\zeta}(\bA)}{\bbP_0(\bA)}
\end{align*}
be the likelihood ratio between $\bbP_{\zeta}$ and $\bbP_{0}$. Then, we have
\begin{align*}
    \chi^2(\bbP_{\zeta}, \bbP_{0}) = \bbE_{\bA \sim \bbP_{0}} L^2(\bA) - 1.
\end{align*}

Let $\bbQ_{\zeta}$ be the distribution of $\bQ$ generated by the sparse graphon $\rho_n W_{\zeta}$. Then, define
\begin{align*}
    L(\bA \mid \bQ) & = \frac{\bbP_{\zeta}(\bA \mid \bQ)}{\bbP_0(\bA)} =  \prod_{1 \le i < j \le n} \frac{\bQ_{ij}^{\bA_{ij}} (1 - \bQ_{ij})^{1 - \bA_{ij}}}{(\rho_n / 2)^{\bA_{ij}} (1 - \rho_n/2)^{1 -\bA'_{ij}}} \\
    & = \prod_{1 \le i < j \le n} \frac{(1 - \bQ_{ij})}{(1 - \rho_n/2)} \cdot \prod_{1 \le i < j \le n} \left ( \frac{\bQ_{ij} (1 - \rho_n/2)}{(1 - \bQ_{ij}) (\rho_n/2)} \right )^{\bA_{ij}}.
\end{align*}
so $L(\bA) = \bbE_{\bQ \sim \bbQ_{\zeta}} L(\bA \mid \bQ)$ and
\begin{align*}
    L^2(\bA) = \bbE_{\bQ, \bQ' \sim \bbQ_{\zeta}} L(\bA \mid \bQ) L(\bA \mid \bQ)
\end{align*}
for $\bQ, \bQ'$ independently sampled from $\bbQ_{\zeta}$. Then, 
\begin{align*}
    L^2(\bA) = \bbE_{\bQ, \bQ' \sim \bbQ_{\zeta}}\prod_{1 \le i < j \le n} \frac{(1 - \bQ_{ij}) (1 - \bQ_{ij}')}{(1 - \rho_n/2)^2} \cdot \prod_{1 \le i < j \le n} \left ( \frac{\bQ_{ij} \bQ'_{ij} (1 - \rho_n/2)^2}{(1 - \bQ_{ij}) (1 - \bQ'_{ij}) (\rho_n/2)^2} \right )^{\bA_{ij}}.
\end{align*}
Then, we have
\begin{align*}
    & \chi^2(\bbP_\zeta, \bbP_0) \le \bbE_{\bA \sim \bbP_0} L^2(\bA) \\
    & \quad = \bbE_{\bQ, \bQ' \sim \bbQ_{\zeta}}\prod_{1 \le i < j \le n} \frac{(1 - \bQ_{ij}) (1 - \bQ_{ij}')}{(1 - \rho_n/2)^2} \cdot \prod_{1 \le i < j \le n} \bbE_{\bA_{ij} \sim \bbP_0} \left ( \frac{\bQ_{ij} \bQ'_{ij}  (1 - \rho_n/2)^2}{(1 - \bQ_{ij}) (1 - \bQ'_{ij}) (\rho_n/2)^2} \right )^{\bA_{ij}},
\end{align*}
where we used independence of $\bA_{ij}$ for $1 \le i \le j \le n$ under $\bbP_0$. Computing expectations w.r.t $\bA_{ij}$ explicitly, we obtain 
\begin{align*}
    & \chi^2(\bbP_\zeta, \bbP_0) \\
    & \quad \le \bbE_{\bQ, \bQ' \sim \bbQ_{\zeta}} \left ( \prod_{1 \le i < j \le n} \frac{(1 - \bQ_{ij}) (1 - \bQ_{ij}')}{(1 - \rho_n/2)^2} \prod_{1 \le i < j \le n} \left ( 1 - \rho_n /2 + \frac{\bQ_{ij} \bQ_{ij}' (1 - \rho_n/2)^2}{(1 - \bQ_{ij}) (1 - \bQ_{ij}') (\rho_n/2)} \right ) \right ) \\
    & \quad = \bbE_{\bQ, \bQ' \sim \bbQ_{\zeta}} \prod_{1 \le i < j \le n} \left ( \frac{(1 - \bQ_{ij}) (1 - \bQ_{ij}')}{ (1 - \rho_n/2)} + \frac{\bQ_{ij} \bQ_{ij}'}{(\rho_n/2)} \right ).
\end{align*}
Simplifying the latter product, we get
\begin{align*}
    \chi^2(\bbP_\zeta, \bbP_0) & \le \bbE_{\bQ, \bQ' \sim \bbQ_{\zeta}} \prod_{1 \le i < j \le n} \left (1 + \frac{(\bQ_{ij} - \rho_n/2) (\bQ'_{ij} - \rho_n/2)}{\rho_n/2(1 - \rho_n/2)} \right ) \\
    & = \bbE \prod_{1 \le i < j \le n} \left (1 + \frac{\rho_n^2 \Delta(\xi_{(i)}, \xi_{(j)})\Delta(\xi'_{(i)}, \xi'_{(j)})}{\rho_n/2(1 - \rho_n/2)} \right ),
\end{align*}
where $\xi_1', \ldots, \xi'_n$ are independent copies of $\xi_1, \ldots, \xi_n \sim \operatorname{Uniform}([0;1])$. Since $|\Delta(x, y)| \le H \ts^{-\alpha} / 2$, the above implies
\begin{align*}
    \chi^2(\bbP_{\zeta}, \bbP_0) \le \prod_{1 \le i < j \le n} \left ( 1 + H^2 \rho_n \ts^{-2 \alpha} \right ) \le \exp \left ( H^2 n (n - 1)/2 \rho_n^2 \ts^{- 2 \alpha}\right ) \le e^{H^2 n^2 \rho_n \ts^{-2 \alpha}}.
\end{align*}
Then, if $|\Xi| \ge 3\chi^2(\bbP_{\zeta}, \bbP_0)$ for any $\zeta \in \Xi$, we may apply Theorem~\ref{theorem: chi-square lower bound}.

\noindent \textbf{Step 3. Applying Theorem~\ref{theorem: chi-square lower bound}.} For $\ts \ge 17 \cdot 18$, we have
\begin{align*}
    |\Xi|/3 \ge e^{\ts / 36},
\end{align*}
so to ensure $|\Xi| \ge 3 \chi^2(\bbP_{\zeta}, \bbP_0)$ for any $\zeta \in \Xi$, it is enough to choose $\ts$ such that
\begin{align*}
    \ts^{1 + 2\alpha} \ge 36 \cdot H^2 n^2 \rho_n.
\end{align*}
Under the assumptions of the theorem, $H^2 n^2 \rho_n \ge 2^{20}$, so it is enough to set $$\ts = \lceil (36 H^2 n^2 \rho_n)^{1/(1 + 2 \alpha)} \rceil \ge 17 \cdot 18$$ for $\alpha \in (0; 1]$. The condition $\ts \ge H^{1/\alpha}$ is implied by $6 n \rho_n \ge H^{1/(2 \alpha)}$. Then, since
\begin{align*}
    \Vert \rho_n \bbW_{\zeta} - \rho_n \bbW_{\zeta'} \Vert_{\op} \ge H\rho_n \ts^{-\alpha} / 15, \\
    \Vert \rho_n \bbW_{\zeta} - \rho_n \bbW_0 \Vert_{\op} \ge H \rho_n \ts^{-\alpha} / 15 
\end{align*}
due to~\eqref{eq: packing in operator norm} and~\eqref{eq: boundness of packing from 0}, Theorem~\ref{theorem: chi-square lower bound} implies
\begin{align*}
    \inf_{\widehat{\bbW}} \sup_{\bbW \in \Sigma(H, \alpha)} \bbP_{\bA \sim \rho_n \bbW} \left ( \Vert \rho_n \bbW - \widehat{ \bbW}(\bA) \Vert_{\op} \ge H \rho_n \ts^{-\alpha} / 30 \right ) \ge \frac{1}{6}.
\end{align*}
For our choice of $\ts$ we have
\begin{align*}
    H \rho_n \ts^{-\alpha} \ge  H^{\frac{1}{1 + 2 \alpha}} \rho_n^{(1 + \alpha)/(1 + 2 \alpha)} n^{- 2 \alpha/(1 + 2 \alpha)},
\end{align*}
so
\begin{align*}
    \inf_{\widehat{\bbW}} \sup_{\bbW \in \Sigma(H, \alpha)} \bbP_{\bA \sim \rho_n \bbW} \left ( \Vert \rho_n \bbW - \widehat{ \bbW}(\bA) \Vert_{\op} \ge \frac{1}{30} \left ( \frac{H \rho_n^{1 + \alpha}}{n^{2 \alpha}} \right )^{\frac{1}{(1 + 2 \alpha)}} \right ) \ge \frac{1}{6}. & \qedhere
\end{align*}
\end{proof}

\subsubsection{Proof of Lemma~\ref{lemma: almost orthogonal vectors}}
\label{section: vector packing lemma}

\begin{proof}
    Let $\zeta_1, \ldots, \zeta_M \in \{-1, 1\}^{\ts}$ be independent random vectors, each consists of independent Rademacher random variables. Then, $\langle \zeta_i, \zeta_j \rangle$ is a zero-mean sum of $\ts$ independent Rademacher random variables, so
    \begin{align*}
        \bbP \left ( |\langle \zeta_i, \zeta_j \rangle|  \ge t \right ) \le 2 \exp \left ( - \frac{t^2}{2\ts} \right )
    \end{align*}
    for any $t > 0$ by the Hoeffding inequality. In particular, for $t = \ts/2$, we have
    \begin{align*}
        \bbP \left (  |\langle \zeta_i, \zeta_j \rangle | \ge \ts/2 \right ) \le 2 e^{- \ts/8}.
    \end{align*}
    It implies
    \begin{align*}
        \bbP \left ( \exists \text{ distinct } i, j  \text{ such that } |\langle \zeta_i, \zeta_j \rangle | \ge \ts/2 \right ) \le 2 \binom{M}{2} e^{-\ts / 8} \le M^2 e^{-\ts / 8}.
    \end{align*}
    Choosing $M = \lfloor e^{\ts/17} \rfloor $, we get
    \begin{align*}
        \bbP \left ( \exists \text{ distinct } i, j  \text{ such that } |\langle \zeta_i, \zeta_j \rangle | \ge \ts/2 \right ) < 1,
    \end{align*}
    so there exists a family of $M \ge e^{\ts/ 18}$ vectors $\zeta_1, \ldots, \zeta_M$ such that $|\langle \zeta_i, \zeta_j \rangle| \le \ts/2$ for any distinct $i, j$, provided $\ts \ge 17 \cdot 18$.
\end{proof}

\section{Proof of Theorem~\ref{theorem: analityc convergence rate}}
\label{appendix: proof of analytic convergence rate}

i) \textbf{Upper bound on the convergence rate.} We claim that Definition~\ref{condition: M-r analiticity} guarantees that $W(x, y)$ is $(H, 1)$-H\"older ($H$-Lipschitz) for some $H$. If $|x_2 - x_1| > r/2$ or $|y_1 - y_2| > r/2$, we can bound
\begin{align*}
	|W(x_1, y_1) - W(x_2, y_2)| \le 2 \le \frac{2}{r} (|x_2 - x_1| + |y_2 - y_1|).
\end{align*}
Otherwise, we have
\begin{align*}
	|W(x_1, y_1) - W(x_2, y_2)| & \le |W(x_1, y_1) - W(x_2, y_1)| + |W(x_2, y_1) - W(x_2, y_2)| \\
	& \le \sum_{i = 1}^{\infty} |c_i(x_1, y_1)| |x_2 - x_1|^i + \sum_{i = 1}^{\infty} |c_i(x_2, y_1)| |y_2 - y_1|^i.
\end{align*}
Since $\sup_{x, y} |c_k(x, y)| \le M r^{-k}$, we obtain
\begin{align*}
	|W(x_1, y_1) - W(x_2, y_2)| & \le M \left ( \sum_{i = 1}^\infty \frac{|x_2 - x_1|^i}{r^i} + \sum_{i = 1}^\infty  \frac{|y_2 - y_1|^i}{r^i} \right ) \\
	& \le M \left ( \frac{r^{-1}|x_2 - x_1|}{1 - r^{-1} |x_2 - x_1|} + \frac{r^{-1}|y_2 - y_1|}{1 - r^{-1}|y_2 - y_1|} \right ) \\
	& \le 4 r^{-1} M (|x_2 - x_1| + |y_2 - y_1|).
\end{align*}
Consequently, graphon $W$ is indeed $(H, 1)$-H\"older for $H = \max(4 M r^{-1}, 2r^{-1})$. Applying~\eqref{eq: final bound in the proof of Holder graphon rates} from the proof of Theorem~\ref{theorem: holder convergence rate}, we obtain the desired rate of convergence with probability at least $1 - \delta/2$.

ii) \textbf{Upper bound on the rank.} Choose some integer $\ell$ such that $\ell \ge 2 / r$. For $a,b\in \{1,\dots,l\}$ consider $\Delta_{a b} = [(a - 1)/\ell, a/\ell) \times [(b - 1)/\ell, b/\ell)$.  Let  $ p\geq 1$ be an integer.  For each $\Delta_{ab}$  define function
\begin{align*}
	P_p^{ab}(x, y) = \sum_{i = 0}^{p - 1} c_i \left ( \frac{a - 1/2}{\ell}, y \right ) \left (x - \frac{a - 1/2}{\ell} \right )^i.
\end{align*}
The function $P_p^{ab}$ has rank at most $p$.
Define   function $P_p: [0, 1]^2 \to \bbR$ as $P_{p}|_{\Delta_{ab}} = P^{ab}_{p}$.
Then $P_p$ has rank at most $\ell^2 p$. Moreover, we have  
\begin{align*}
	\sup_{x, y} |W(x, y) - P_p(x, y)| = \sup_{a, b} \sup_{(x, y) \in \Delta_{ab}} |W(x, y) - P_p^{ab}(x, y)| \le M \sum_{i = p}^{\infty} r^{-i} l^{-i} \le M \sum_{i = p}^{\infty} 2^{-i}  \le M 2^{1-p}.
\end{align*}
Next, define $\bP_p = (P_p(\xi_i, \xi_j))_{1 \le i, j \le n}$ and recall that $\bQ = \rho_n(W(\xi_i, \xi_j))_{1 \le i, j \le n}$. Then $\bP_p$ has the same rank as $P_p$. Indeed, if $P_p(x, y)$ can be represented as a sum $\sum_{r = 1}^{\ell^2 p} \phi_r(x) \psi_r(y)$ of functions of rank $1$, then $\bP_k$ has decomposition $\sum_{r = 1}^{\ell^2 p} \bv_r \bg_r^\top$, where $\bv_r = (\phi_r(\xi_i))_{i = 1}^n$, $\bg_r = (\psi_r(\xi_i))_{i = 1}^n$. 
We have, 
\begin{align*}
   \sigma_{\ell^2 p}(\bQ) \le \Vert \bQ - \rho_n \bP_p \Vert_{op}\le \Vert \bQ - \rho_n \bP_p \Vert_{F} \le M n \rho_n 2^{1-p}.
\end{align*}
For any $k \ge \ell^2$, by choosing $p = \left \lfloor \frac{k}{\ell^2} \right \rfloor$ we can conclude that 
\begin{align*}
	\sigma_k(\bQ) \le M n \rho_n 2^{1 - p} \leq 4 M n \rho_n 2^{-\frac{k}{\ell^2}.} 
\end{align*}
Since $\Vert \bQ \Vert_{\op} \le n \rho_n$, there exists constant $C(M, r)$ such that $\sigma_k(\bQ) \le C(M, r) n \rho_n 2^{- k/\ell^2}$ for all positive $k$. Since $\bQhat^*_1$ is obtained from $\cS_m(\bA)$ by hard-thresholding with threshold $\lambda'$ for
\begin{align*}
    m = \lceil n^{1/3} / \rho_n^{1/3} \rceil, \quad \lambda' = 2C_1 \rho_n n \cdot \left ( \frac{\log(8/\delta)}{n} \right )^{1/2}  + 2C_2\rho_n^{2/3} n^{1/3},
\end{align*}
we can apply Proposition~\ref{proposition: singular value rank}, and obtain
\begin{align*}
    \rank(\widehat{\bQ}^*_1) & \le \max \{k: \sigma_k(\bQ) > \lambda' - \Vert \cS_m(\bA) - \bQ \Vert_{\op} \} \\
    & \le \max \left \{ k : C(M, r) n \rho_n 2^{- k / \ell^2} > C_1 \rho_n n \cdot \left ( \frac{\log(8/\delta)}{n} \right )^{1/2}  + C_2\rho_n^{2/3} n^{1/3} \right \}
\end{align*}
with probability at least $1 - \delta/2$ due to Theorem~\ref{theorem: Holder averaging}. Choosing $\ell = \lceil 2 / r \rceil$, we get
\begin{align*}
	\rank(\widehat{\bQ}_\lambda) & \leq C' \left ( \log (n \rho_n) + \log \log \frac{8}{\delta} \right )
\end{align*}
for some constant $C'$ that depends on $M$ and $r$ only with probability at least $1 - \delta/2$. Together with the conclusion of the first part, the above implies the result.

\section{Proof of Theorem~\ref{theorem: analytic graphon estimation lower bound}}
\label{section: analytic graphon lower bound proof}

First, we will lower bound $\Vert \bbW^{(1)} - \bbW^{(0)} \Vert_{\op}$. Since both graphons has rank at most $2$, we have
\begin{align*}
    \Vert \bbW^{(1)} - \bbW^{(0)} \Vert_{\op} \ge \frac{1}{\sqrt{2}} \Vert \bbW^{(1)} - \bbW^{(0)} \Vert_{F}. 
\end{align*}
Then, we have
\begin{align*}
    \Vert \bbW^{(1)} - \bbW^{(0)} \Vert_{F}^2 & = \int_0^1 \int_0^1 (W^{(1)}(x, y) - W^{(0)}(x, y))^2 dx dy = \int_0^1 \int_0^1 (\tau_t(x) \tau_t(y) - xy )^2 dx dy \\
    & = \int_0^1 \int_0^1 (t x y (1 - x) + t xy (1 - y) + t^2 xy(1 -x )(1 - y))^2 dx dy \\
    & \ge t^2 \int_0^1 \int_0^1 ( x y (1 - x) +  xy (1 - y))^2 dx dy.
\end{align*}
Computing the latter integral using symbolic algebra Python library \texttt{sympy}, we get
\begin{align*}
     \Vert \rho_n \bbW^{(1)} - \rho_n \bbW^{(0)} \Vert_{\op} \ge \frac{\rho_n}{\sqrt{2}}\Vert \bbW^{(1)} - \bbW^{(0)} \Vert_{F} \ge \rho_n t \cdot \sqrt{\frac{13}{520}} \ge \rho_n t/7 = \frac{\rho_n}{28 \sqrt{n}},
\end{align*}
where we used $t = 1/(4 \sqrt{n})$ from the definition of $W^{(1)}$.

Next, note that Lemma~\ref{lemma: shifted graphon KL-divergence} implies
\begin{align}
    \KL{\bbP_{W^{(1)}}}{\bbP_{W^{(0)}}} \le 4 n t^2 = 1/4, \label{eq: KL divergence bound for smooth graphons}
\end{align}
so by Theorem~\ref{theorem:minimax bounds}, we have
\begin{align*}
    \inf_{\widehat{\bbW}} \max_{i \in \{0, 1\}} \bbP_{\bA \sim \rho_n \bbW} \left ( \Vert \rho_n \bbW^{(i)} - \widehat{\bbW}(\bA) \Vert_{\op} \ge \frac{\rho_n}{56\sqrt{n}} \right ) \ge 1/6,
\end{align*}
which concludes the proof.

\section{Proof of Lemma~\ref{lemma: regret upper bound spectral gap}}

\noindent \textbf{Step 1. Taking perturbation of $\hat{\theta}_1$ into account.} To prove the second part of the theorem, we define
\begin{align*}
    \Delta_{\cR^2} = \cR^2(\gamma, \bbW_1) - \cR^2(\gamma, \bbW_2).
\end{align*}
Then, lower bound~\eqref{eq: lower bound on T with approximate theta} can be rewritten as
\begin{align*}
    T_1(\hat{\theta}_2) \ge T_1(\hat{\theta}_1) + \frac{1}{2} \langle \Delta_{\cR^2} [\theta + \hat{\theta}_2], \theta + \hat{\theta}_2 \rangle - \frac{1}{2} \langle \Delta_{\cR^2} [\theta + \hat{\theta}_1], \theta + \hat{\theta}_1 \rangle.
\end{align*}
Then, introduce $\ttd = \hat{\theta}_2 - \hat{\theta}_1$. We have
\begin{align*}
    \langle \Delta_{\cR^2} [\theta + \hat{\theta}_2], \theta + \hat{\theta}_2 \rangle & = \langle \Delta_{\cR^2} [\theta + \hat{\theta}_1 + \ttd], \theta + \hat{\theta}_1 + \ttd \rangle \\
    & = \langle \Delta_{\cR^2} [\theta + \hat{\theta}_1], \theta + \hat{\theta}_1 \rangle + 2 \langle \Delta_{\cR^2} [\theta + \hat{\theta}_1], \ttd \rangle + \langle \Delta_{\cR^2} [\ttd], \ttd \rangle,
\end{align*}
which implies
\begin{align*}
    T_1(\hat{\theta}_2) & \ge T_1(\hat{\theta}_1) + \langle \Delta_{\cR^2} [\theta + \hat{\theta}_1], \ttd \rangle + \frac{1}{2} \langle \Delta_{\cR^2} [\ttd], \ttd \rangle \\
    & \ge T_1(\hat{\theta}_1) - \Vert \Delta_{\cR^2} \Vert_{\op} \cdot \Vert \theta + \hat{\theta}_1 \Vert_{L_2} \cdot \Vert \ttd \Vert_{L_2} - \frac{1}{2} \Vert \Delta_{\cR^2} \Vert_{\op} \cdot \Vert \ttd \Vert_{L_2}^2.
\end{align*}
We simplify the above by using the bound $\Vert \hat{\theta}_i \Vert_{L_2}^2 \le B$, $i = 1, 2$,so $\Vert \ttd \Vert_{L_2} \le 2\sqrt{B}$ and $\Vert \theta + \hat{\theta}_1 \Vert_{L_2} \le (\Vert \theta \Vert_{L_2} + \sqrt{B})$, so we have
\begin{align*}
    T_1(\hat{\theta}_2) \ge T_1(\hat{\theta}_1) -  \Vert \Delta_{\cR^2} \Vert_{\op} \cdot \Vert \ttd \Vert_{L_2} \cdot (\Vert \theta \Vert_{L_2} + 2 \sqrt{B}).
\end{align*}
Then, applying the bound~\eqref{eq: resolvent different upper bound} on $\Vert \Delta_{\cR^2} \Vert_{\op}$, we obtain
\begin{align}
\label{eq: perturbation bound with theta difference}
    T_1(\hat{\theta}_2) \ge T_1(\hat{\theta}_1) - \frac{2 \gamma \Vert \bbW_1 - \bbW_2 \Vert_{\op} \cdot \Vert \ttd \Vert_{L_2} \cdot (\Vert \theta \Vert_{L_2} + 2 \sqrt{B})}{(1 - \gamma [\Vert \bbW_1 \Vert_{\op} \vee \Vert \bbW_2 \Vert_{\op}])^3}.
\end{align}
Then, we bound $\Vert \ttd \Vert_{L_2}$ under the spectral gap assumption, and using $\theta(x) \equiv 0$.

\medskip

\noindent \textbf{Step 2. Bounding $\Vert \hat{\theta}_2 - \hat{\theta}_1 \Vert_{L_2}$ under the spectral gap assumption.} When $\theta(x) \equiv 0$, the optimal intervention $\hat{\theta}_i$ are eigenfunctions of the operator $\bbW_i$ corresponding to its largest eigenvalue, $i = 1, 2$. We may choose them in a way such that $\langle \hat{\theta}_1, \hat{\theta}_2 \rangle \ge 0$. We use the following lemma, which can be seen as a variant of the Davis--Kahan theorem~\citep{davis1970rotation} for integral operators.

\begin{lemma}[Lemma 7 of~\citep{avella-medina_centrality_2018}]
\label{lemma: operator Davis-Kahan}
    Consider two linear integral self-adjoint operators $\bbW_1, \bbW_2$ with symmetric bounded kernels $W_1, W_2 \in L_2([0, 1]^2)$. Let $\lambda_1(\bbW_i)$ be the largest eigenvalue of $\bbW_i$, and $u_i(\cdot)$ be the corresponding unit eigenfunction, $i = 1, 2$ such that $\langle u_1, u_2 \rangle \ge 0$. If $\lambda_1(\bbW_1) - \lambda_2(\bbW_1) > |\lambda_1(\bbW_1) - \lambda_1(\bbW_2)|$, then
    \begin{align*}
        \Vert u_1 - u_2 \Vert_{L_2} \le \frac{\sqrt{2} \Vert \bbW_1 - \bbW_2 \Vert_{\op}}{\lambda_1(\bbW_1) - \lambda_2(\bbW_1) - |\lambda_1(\bbW_1) - \lambda_1(\bbW_2)|}.
    \end{align*}
\end{lemma}

Noting that
\begin{align*}
    \lambda_1(\bbW_1) - \lambda_1(\bbW_2) & = \langle \bbW_1[u_1], u_1 \rangle - \langle \bbW_2[u_2], u_2 \rangle \\
    & \le \langle \bbW_1[u_1], u_1 \rangle - \langle \bbW_2[u_1], u_1 \rangle = \langle (\bbW_1 - \bbW_2)[u_1], u_1 \rangle \le \Vert \bbW_1 - \bbW_2 \Vert_{\op},
\end{align*}
and
\begin{align*}
    \lambda_1(\bbW_2) - \lambda_1(\bbW_1) \le \langle \bbW_2[u_2], u_2 \rangle - \langle \bbW_1[u_2], u_2 \rangle = \langle (\bbW_2 - \bbW_1)[u_2], u_2 \rangle \le \Vert \bbW_1 - \bbW_2 \Vert_{\op},
\end{align*}
we deduce that $| \lambda_1(\bbW_1) - \lambda_1(\bbW_2)| \le \Vert\bbW_1 - \bbW_2 \Vert_{\op}$. Then, we consider two cases. If $\ttg \ge 2 \Vert \bbW_1 - \bbW_2 \Vert_{\op}$, then $\lambda_1(\bbW_1) - \lambda_2(\bbW_1) > |\lambda_1(\bbW_1) - \lambda_1(\bbW_2)|$, and we can apply Lemma~\ref{lemma: operator Davis-Kahan} to obtain
\begin{align*}
    \frac{1}{\sqrt{B}} \Vert \hat{\theta}_2 - \hat{\theta}_1 \Vert_{L_2} \le \frac{\sqrt{2} \Vert \bbW_1 - \bbW_2 \Vert_{\op}}{\ttg - \Vert \bbW_1 - \bbW_2 \Vert_{\op}} \le \frac{2 \sqrt{2} \Vert \bbW_1 - \bbW_2 \Vert_{\op}}{\ttg}.
\end{align*}
If $\ttg < 2 \Vert \bbW_1 - \bbW_2 \Vert_{\op}$, then we have 
\begin{align*}
    \frac{1}{\sqrt{B}} \Vert \hat{\theta}_2 - \hat{\theta}_1 \Vert_{L_2} \le 2 \le \frac{4\Vert \bbW_1 - \bbW_2 \Vert_{\op}}{\ttg}.
\end{align*}
In either case, we have $\Vert \ttd \Vert_{L_2} \le \frac{4 \sqrt{B} \Vert \bbW_1 - \bbW_2 \Vert_{\op}}{\ttg}$, which implies
\begin{align*}
    T_1(\hat{\theta}_2) \ge T_1(\hat{\theta}_1) - \frac{8 \gamma \sqrt{B} (\Vert \theta \Vert_{L_2} + 2 \sqrt{B}) \cdot \Vert \bbW_1 - \bbW_2 \Vert_{\op}^2}{\ttg (1 - \gamma [\Vert \bbW_1 \Vert_{\op} \vee \Vert \bbW_  2 \Vert_{\op}])^3}
\end{align*}
due to~\eqref{eq: perturbation bound with theta difference}.

\section{Proof of Theorem~\ref{theorem: spectral norms deficit}}
\label{appendix: proof of deficit lower bound theorem}

Let $\bj^*$ be uniformly distributed on $\{0, 1\}$. For a matrix $\bA$ sampled from the sparse graphon model $\rho_n W_{\bj^*}$ and an estimator $\widehat{\theta}(\bA)$ define
\begin{align*}
    \widehat{j} = \argmin_{j \in \{0, 1\}} \regret_j (\widehat{\theta}(\bA)),
\end{align*}
breaking ties by setting $\widehat{j} = 0$.
If $\widehat{j} \neq \bj^*$, we have
\begin{align}
\label{eq: missclassification regret}
    \regret_{\bj^*}(\widehat{\theta}) \ge \regret_{1 - \bj^*}(\widehat{\theta}).
\end{align}
Meanwhile, for the special case when $\theta(x) \equiv 0$, we have
\begin{align*}
    \regret_0(\widehat{\theta}) + \regret_1(\widehat{\theta}) = T_{(\rho_n W_0, 0)}(\theta^*_1) + T_{(\rho_n W_1, 0)}(\theta^*_2) - T_{(\rho_n W_0, 0)}(\widehat{\theta}) - T_{(\rho_n W_1, 0)}(\widehat{\theta}),
\end{align*}
where $\theta^*_j$ is a maximizer of $T_{(\rho_n W_j, 0)}(\cdot)$. Using
\begin{align*}
    T_{(\rho_n W_j, 0)}(\tilde{\theta}) = \Vert (\bI - \gamma \rho_n \bbW_j)^{-1} [\tilde{\theta}] \Vert_{2}^2,
\end{align*}
we deduce that $\theta^*_j$ is the leading eigenfunction of $\bbW_j$, so $T_{(\rho_n W_j, 0)}(\theta^*_j) = B \Vert (\bI - \gamma \rho_n \bbW_j)^{-2} \Vert_{\op}^2$. Hence, we have
\begin{align*}
    \regret_0(\widehat{\theta}) + \regret_1(\widehat{\theta}) & \ge B \Vert (\bI - \gamma \rho_n \bbW_0)^{-2} \Vert_{\op} + B \Vert (\bI - \gamma \rho_n \bbW_1)^{-2} \Vert_{\op} \\
    & \qquad - \langle \widehat{\theta}, \left ( (\bI - \gamma \rho_n \bbW_0)^{-2} + (\bI - \gamma \rho_n \bbW_1)^{-2} \right ) [\widehat{\theta}] \rangle \\
    & \ge B \Vert (\bI - \gamma \rho_n \bbW_0)^{-2} \Vert_{\op} + B \Vert (\bI - \gamma \rho_n \bbW_1)^{-2} \Vert_{\op} \\
    & \qquad - B \Vert (\bI - \gamma \rho_n \bbW_0)^{-2} + (\bI - \gamma \rho_n \bbW_1)^{-2}  \Vert_{\op} \\
    & \ge B s.
\end{align*}
Together with~\eqref{eq: missclassification regret}, it implies
\begin{align*}
    2 \regret_{\bj^*}(\widehat{\theta}) \ge \regret_0(\widehat{\theta}) + \regret_1(\widehat{\theta}) \ge B s,
\end{align*}
provided $\widehat{j} \neq \bj^*$. Therefore, we have
\begin{align*}
    \bbP \left (\regret_{\bj^*}(\widehat{\theta}) \ge B s/ 2 \right ) \ge \bbP (\widehat{j} \neq \bj^*).
\end{align*}
The latter is bounded from below by Theorem~\ref{theorem: fano's lemma}, so
\begin{align*}
    \max_{j = 0, 1} \bbP \left (\regret_{j}(\widehat{\theta}) \ge B s/ 2 \right ) \ge \bbP \left (\regret_{\bj^*}(\widehat{\theta}) \ge B s/ 2 \right ) \ge \max \left \{ e^{-\alpha}/4, \frac{1 - \sqrt{\alpha}}{2}\right \}.
\end{align*}
Since the latter holds for any estimator $\widehat{\theta}$, the theorem follows.

\section{Proof of Proposition~\ref{proposition: graphons power upper bound}}
\label{section: graphons deficit power lower bound proof}

Since operators $\bbW_0, \bbW_1$ are  Hilbert–Schmidt, we have
    \begin{align*}
        \Vert (\bbI - \gamma \rho_n \bbW_i)^{-2} \Vert_{\op} = (1 - \gamma \rho_n \lambda_{\max}(\bbW_i))^{-2} = \sum_{k = 0}^{\infty} (k + 1) (\gamma \rho_n \lambda_{\max}(\bbW_i))^k
    \end{align*}
    for each $i = 0,1$. On the other hand, we have
    \begin{align*}
        \Vert (\bbI - \gamma \rho_n \bbW_0)^{-2} + (\bbI - \gamma \rho_n \bbW_1)^{-2} \Vert_{\op} & = \left \Vert \sum_{k = 0}^{\infty} (k + 1) (\gamma \rho_n)^k (\bbW_0^k + \bbW_1^k) \right \Vert_{\op} \\
        & \le \sum_{k = 0}^{\infty} (k + 1) (\gamma \rho_n)^k \Vert \bbW_0^k + \bbW_1^k \Vert_{\op}.
    \end{align*}
    Therefore, we have
    \begin{align*}
        & \Vert (\bbI - \gamma \rho_n \bbW_0)^{-2} \Vert_{\op} +  \Vert (\bbI - \gamma \rho_n \bbW_1)^{-2} \Vert_{\op} - \Vert (\bbI - \gamma \rho_n \bbW_0)^{-2} + (\bbI - \gamma \rho_n \bbW_1)^{-2} \Vert_{\op} \\
        & \qquad \ge \sum_{k = 0}^{\infty} (k + 1) \gamma^k \rho_n^k (\lambda_{\max} (\bbW_0)^k + \lambda_{\max} (\bbW_1)^k - \Vert \bbW_0^k + \bbW_1^k \Vert_{\op})  \\
        & \qquad \ge (\gamma \rho_n)^{k_0} s', 
    \end{align*}
    which concludes the proof.

\section{Proof of Theorem~\ref{theorem: regret lower bound smooth graphons}}

    In the proof, we choose graphons $\bbW^{(0)}$, $\bbW^{(1)}$ from the statement of Theorem~\ref{theorem: analytic graphon estimation lower bound}.

   \noindent \textbf{Step 1. Reduction to matrices.} Note that images of $\bbW^{(i)}$, $i = 0, 1$, belong to 
   \begin{align*}
        \mathfrak{I} = \mathrm{span}(1, x, \tau_t(x)) = \mathrm{span}(1, x, x + t x (1 - x)) = \mathrm{span}(1, x, x^2).
   \end{align*}
   We will choose the following orthonormal basis in $\mathfrak{I}$:
   \begin{align*}
        e_1(x) & = 1 \\
        e_2(x) & = \sqrt{3} (2x - 1) \\
        e_3(x) & = \sqrt{5}(6 x^2 - 6x + 1).
   \end{align*}
   Since $\bbW^{(i)}$ are self-adjoint, we have $\mathfrak{J}^{\perp} \subset \mathrm{Ker}(\bbW^{(i)})$, so, if maximum eigenvalues of $\bbW^{(i)}$ are positive, they coincide with maximum eigenvalues of matrices
   \begin{align*}
    \ttW_0 = \left ( \int_{0}^1 \int_{0}^{1} W^{0}(x, y) e_i(x) e_j(y) dx dy \right )_{i = 1, j = 1}^{3, 3} = \begin{pmatrix}
        \frac{5}{8} & \frac{\sqrt{3}}{24} & 0 \\
        \frac{\sqrt{3}}{24} & \frac{1}{24} & 0 \\
        0 & 0  & 0
    \end{pmatrix}
   \end{align*}
   and
   \begin{align}
        \ttW_t & = \left ( \int_{0}^1 \int_{0}^{1} W^{(1)}(x, y) e_i(x) e_j(y) dx dy \right )_{i = 1, j = 1}^{3, 3} \nonumber \\
        & = \begin{pmatrix}
        \frac{5}{8} & \frac{\sqrt{3}}{24} & 0 \\
        \frac{\sqrt{3}}{24} & \frac{1}{24} & 0 \\
        0 & 0  & 0
       \end{pmatrix} + t \begin{pmatrix}
        \frac{1}{12} & \frac{\sqrt{3}}{72} & - \frac{\sqrt{5}}{120} \\
        \frac{\sqrt{3}}{72} & 0 & - \frac{\sqrt{15}}{360} \\
        - \frac{\sqrt{5}}{120} & - \frac{\sqrt{15}}{360} & 0 
       \end{pmatrix} + t^2 \begin{pmatrix}
        \frac{1}{72} & 0 & - \frac{\sqrt{5}}{360} \\
        0 & 0 & 0 \\
        - \frac{\sqrt{5}}{360} & 0 & \frac{1}{360}
       \end{pmatrix} \nonumber \\
       & := \ttW_0 + t \cdot \ttD_1 + t^2 \cdot \ttD_2 \label{eq: matrix polynom coefficients}
   \end{align}
   It will be more convenient to study eigenvalues of $\ttW_t$ in terms of $t$ instead of $n = (16 t^2)^{-1}$, assuming $t$ is smaller than some absolute constant. Using symbolic algebra Python library \texttt{sympy}, we compute eigenvalues of $\ttW_0$ explicitly:
   \begin{align}
    \label{eq: eigenvalues of W0}
        \lambda_{1}(\ttW_0) = \frac{\sqrt{13}}{12} + \frac{1}{3} > 0, \quad \lambda_2(\ttW_0) = \frac{1}{3} - \frac{\sqrt{13}}{12}, \quad \lambda_{3}(\ttW_0) = 0,
   \end{align}
   so $\lambda_{1}(\ttW_0) = \lambda_{\max}(\bbW^{(0)})$.
   To locate eigenvalues of $\ttW_t$, we bound norms matrices $\ttD_1, \ttD_2$ defined by~\eqref{eq: matrix polynom coefficients}. We have $\Vert \ttD_1\Vert_{\op} \le 1$ and $\Vert \ttD_2 \Vert_{\op} \le 1$. Since $t \le 1/12$, we have
   \begin{align*}
    |\lambda_i(\ttW_0) - \lambda_i (\ttW_t)| \le t \Vert \ttD_1 \Vert_{\op} + t^2 \Vert \ttD_2 \Vert_{\op} \le \frac{1}{6},
   \end{align*}
   so
   \begin{align*}
        \lambda_1(\ttW_t) \in \left [\frac{1}{6} + \frac{\sqrt{13}}{12}; \frac{1}{2} + \frac{\sqrt{13}}{12} \right ], \quad \lambda_2(\ttW_t) \in \left [ \frac{1}{6} - \frac{\sqrt{13}}{12} ; \frac{1}{2} - \frac{\sqrt{13}}{12} \right ], \quad \lambda_3(\ttW_t) \in \left [- \frac{1}{6}; \frac{1}{6} \right ].
   \end{align*}
   In particular, $\lambda_1(\ttW_t) = \lambda_{\max}(\bbW^{(1)})$ and $\lambda_1(\ttW_t) - \lambda_2(\ttW_t) \ge \sqrt{13}/6 - 1/3$ for $t \in [0; 0.5]$.

   For $z \in [0, t]$, define
   \begin{align*}
        f(z) & = \lambda_{\max}(\ttW_0) + \lambda_{\max}(\ttW_0 + z \ttD_1 + z^2 \ttD_2) - \Vert 2 \ttW_0 + z \ttD_1 + z^2 \ttD_2 \Vert_{\op} \\
        & = \lambda_{\max}(\ttW_0) + \lambda_{\max}(\ttW_0 + z \ttD_1 + z^2 \ttD_2) - \lambda_{\max}(2\ttW_0 + z \ttD_1 + z^2 \ttD_2)
   \end{align*}
   so any lower bound $f(t)$ provides a minimax lower bound for regret through Theorem~\ref{theorem: spectral norms deficit} and Proposition~\ref{proposition: graphons power upper bound}.

   Consider the characteristic polynomials
   \begin{align*}
        p_1(z, \lambda) & = \det (\ttW_0 + z \ttD_1 + z^2 \ttD_2 - \lambda \bI_3), \\
        p_2(z, \lambda) & = \det( 2 \ttW_0 + z \ttD_1 + z^2 \ttD_2 - \lambda \bI_3).
   \end{align*} 
   Analogously to the analysis of $\lambda_{\max}(\ttW_t)$, $\lambda_{\max}(\ttW_0 + z \ttD_1 + z^2 \ttD_2)$ is a simple root of $p_1(z, \lambda)$ for $z \in [0, 0.5]$. Hence, we have
   \begin{align*}
    (\partial_{\lambda} p_1)(z, \lambda) \bigg |_{\lambda = \lambda_{\max}(\ttW_0 + z \ttD_1 + z^2 \ttD_2)} \neq 0.
   \end{align*}
   Analogously, we have
   \begin{align*}
        (\partial_{\lambda} p_2)(z, \lambda) \bigg |_{\lambda = \lambda_{\max}(2\ttW_0 + z \ttD_1 + z^2 \ttD_2)} \neq 0.
   \end{align*}
   Then, by the implicit function theorem, we have
   \begin{align}
        \frac{d}{dz} \lambda(\ttW_0 + z \ttD_1 + z^2 \ttD_2)&  = - \frac{\partial_z p_1(z, \lambda)}{\partial_{\lambda} p_1(z, \lambda)} \bigg |_{\lambda = \lambda_{\max}(\ttW_0 + z \ttD_1 + z^2 \ttD_2)}, \label{eq: first lambda derivative}\\ 
        \frac{d}{dz} \lambda(2\ttW_0 + z \ttD_1 + z^2 \ttD_2)&  = - \frac{\partial_z p_2(z, \lambda)}{\partial_{\lambda} p_2(z, \lambda)} \bigg |_{\lambda = \lambda_{\max}(2\ttW_0 + z \ttD_1 + z^2 \ttD_2)}. \label{eq: second lambda derivative}
   \end{align}
   It implies
   \begin{align*}
        f'(z) = - \frac{\partial_z p_1(z, \lambda)}{\partial_{\lambda} p_1(z, \lambda)} \bigg |_{\lambda = \lambda_{\max}(\ttW_0 + z \ttD_1 + z^2 \ttD_2)} + \frac{\partial_z p_2(z, \lambda)}{\partial_{\lambda} p_2(z, \lambda)} \bigg |_{\lambda = \lambda_{\max}(2 \ttW_0 + z \ttD_1 + z^2 \ttD_2)}.
   \end{align*}
   In particular, we have
   \begin{align*}
        f'(0) & = \frac{(\partial_z p_2)(0, 2\lambda_0)}{(\partial_{\lambda}p_2)(0, 2\lambda_0)} - \frac{(\partial_z p_1)(0, \lambda_0)}{(\partial_{\lambda}p_1)(0, \lambda_0)} = 0,
   \end{align*}
   which we simplified using \texttt{sympy} again. 

   Differentiating $f'(z)$, we obtain
   \begin{align*}
        \frac{d}{dz} f'(z) & = \frac{\partial_{\lambda z} p_1(z, \lambda)}{(\partial_\lambda p_1(z, \lambda))^2} - \frac{\partial_{zz} p_1(z, \lambda)}{\partial_{\lambda} p_1(z, \lambda)} \bigg |_{\lambda = \lambda_{\max}(\ttW_0 + z \ttD_1 + z^2 \ttD_2)} \\
        & \quad + \frac{\partial_z p_1(z, \lambda) \partial_{\lambda \lambda} p_1(z, \lambda)}{(\partial_\lambda p_1(z, \lambda))^2} \bigg |_{\lambda = \lambda_{\max}(\ttW_0 + z \ttD_1 + z^2 \ttD_2)} \cdot \frac{d}{dz} \lambda_{\max}(\ttW_0 + z \ttD_1 + z^2 \ttD_2) \\
        & \quad + \frac{\partial_{\lambda z} p_2(z, \lambda)}{(\partial_\lambda p_2(z, \lambda))^2} - \frac{\partial_{zz} p_2(z, \lambda)}{\partial_{\lambda} p_2(z, \lambda)} \bigg |_{\lambda = \lambda_{\max}(2\ttW_0 + z \ttD_1 + z^2 \ttD_2)} \\
        & \quad + \frac{\partial_z p_2(z, \lambda) \partial_{\lambda \lambda} p_2(z, \lambda)}{(\partial_\lambda p_2(z, \lambda))^2} \bigg |_{\lambda = \lambda_{\max}(2\ttW_0 + z \ttD_1 + z^2 \ttD_2)} \cdot \frac{d}{dz} \lambda_{\max}(2\ttW_0 + z \ttD_1 + z^2 \ttD_2) 
   \end{align*}
   
   Replacing the derivatives of eigenvalues with~\eqref{eq: first lambda derivative} and~\eqref{eq: second lambda derivative}, and setting $z = 0$, we obtain
   \begin{align*}
    f''(0) = \frac{1}{720} - \frac{\sqrt{13}}{760} > 9 \cdot 10^{-4}.
   \end{align*}
   Since $f''(z)$ is continuous due continuity of characteristic polynomials and eigenvalues as functions of $z$, we have
   \begin{align*}
        f''(z) \ge 8 \cdot 10^{-4}
   \end{align*}
   for any $z \in [0; z_0]$ for some absolute number $z_0$. Therefore, provided $t = 1/(4 \sqrt{n}) \le z_0$, we have
   \begin{align*}
    f(t) = f(0) + t f'(0) + t^2/2 f''(\tilde{z}_t)
   \end{align*}
   for some $\tilde{z}_t \in [0; t]$, so
   \begin{align}
    \label{eq: deficit lower bound}
        \lambda_{\max}(\ttW_0) + \lambda_{\max}(\ttW_0 + z \ttD_1 + z^2 \ttD_2) - \Vert 2 \ttW_0 + z \ttD_1 + z^2 \ttD_2 \Vert_{\op} \ge 4 t^2 \cdot 10^{-4} = \frac{10^{-4}}{4n}.
   \end{align}
   Let $\bbP_{i}$ be the distribution of $\bA$ sampled from the sparse graphon model~\eqref{eq: Q sampling}-\eqref{eq: A sampling} for the sparse graphon $\rho_n W^{(i)}$. Then, by the lines of the proof of Theorem~\ref{theorem: analytic graphon estimation lower bound}, we have $\KL{\bbP_1}{\bbP_0}) \le 1/4$, see~\eqref{eq: KL divergence bound for smooth graphons}.

   Then, Proposition~\ref{proposition: graphons power upper bound} and Theorem~\ref{theorem: spectral norms deficit} lead to the lower bound
   \begin{align*}
     \inf_{\widehat{\theta}} \max_{i = 0, 1} \bbP_{\bA \sim \rho_n W^{(i)}} \left ( \regret_i(\widehat{\theta}) \ge \frac{\rho_n}{4 \cdot 10^4 \cdot n} \right ) \ge \frac{1}{6}.
   \end{align*}

\section{Computational aspects of LQ games}
\label{section: computational aspects}

\begin{lemma}
\label{lemma: computational aspects}
For a network $\bA$ of size $n$ and heterogeneities $\btheta$, let $G(\bA, \btheta)$ be a network LQ game with peer-effect parameter $\gamma$ and budget $B$. Suppose that $\btheta$ has non-zero projection on the eigenspace of the largest eigenvalue $\lambda_1(\bA)$ of $\bA$. Then, there exists a number $L \in \left (-\infty; (1 - \gamma \Vert \bA \Vert_{\op} / n)^{-2} \right )$, such that the optimal interventions can be computed as follows:
\begin{align*}
    \widehat{\btheta} = - \left [\bI + L (\bI - \gamma \bA / n)^2 \right ]^{-1} \btheta, 
\end{align*}
and $\Vert \widehat{\btheta} \Vert^2 = B$.
\end{lemma}

\begin{proof}
    The optimal intervention problem can be restated as follows:
    \begin{align*}
        \max \Vert (\bI - \gamma \bA / n)^{-1} (\btheta + \widehat{\btheta}) \Vert^2 \\
        \text{subject to } \Vert \widehat{\btheta} \Vert^2 \le B,
    \end{align*}
    see the proof of Proposition 1 in~\citep{Parise_Ozdaglar_Econometrica}. Expanding the squared norm, we obtain the following equivalent problem
    \begin{align*}
        \min_{\widehat{\btheta}} \quad \widehat{\btheta}^\top \left [ - (\bI - \gamma \bA / n)^{-2}\right ] \widehat{\btheta} - 2 \widehat{\btheta}^\top (\bI - \gamma \bA/n)^{-2} \btheta, \\
        \text{ subject to } \Vert \widehat{\btheta} \Vert^2_2 \le B.
    \end{align*}
    A problem of the form
    \begin{align*}
        \min_{x} x^\top \bC x - 2 \bb^\top x,  \quad \text{subject to } \Vert x \Vert^2_2 \le \beta,
    \end{align*}
    where $\bC$, $\bb$ and $\beta$ are arbitrary matrix, vector and positive scalar respectively,
    was widely studied in optimization literature, see, for example, papers~\citep{opt_forsythe1965stationary, gander1989constrained, opt_more2006levenberg}. It admits solution in the following form:
    \begin{align*}
        x = (\bC - L \bI)^{-1} \bb, \quad \Vert x \Vert^2_2 = \beta,
    \end{align*}
    where $L \in (-\infty; \lambda_{\min}(\bC))$, assuming $\bb$ has non-zero projection on the eigenspace of $\lambda_{\min}(\bC)$. In our case, $\bC = - (\bI - \gamma \bA / n)^{-2}$ and $\bb = (\bI - \gamma \bA/n)^{-2} \btheta$. Simplifying, we obtain
    \begin{align*}
        \widehat{\btheta} = - \left [ (\bI + L(\bI - \gamma \bA /n)^2 \right ]^{-1} \btheta, \quad \Vert \widehat{\btheta} \Vert^2_2 = B,
    \end{align*}
    where $L \in \left (-\infty; (1 - \gamma \Vert \bA \Vert_{\op} / n)^{-2} \right)$.
\end{proof}

\section{Proof of Lemma \ref{lemma: empirical graphon norm}}
\label{appendix: proof of empirical graphon lemma}

For $f, g \in [0, 1]$ such that $\Vert f \Vert_{L_2} = \Vert g \Vert_2 = 1$ consider
\begin{align*}
	\langle g, \bbW_{\bT}[f] \rangle & = \sum_{i, j} \bT_{ij} \int_{[(i - 1)/n, i/n)} g(y) dy \cdot \int_{[(j - 1)/n, j/n)} f(x) dx \\ 
	& = \bv^\top_g \bT \bv_f \\
	& \le \Vert \bv_g \Vert_2 \cdot \Vert \bT \Vert_{\op} \cdot \Vert \bv_f \Vert_2.
\end{align*}
where $\bv_f$ and $\bv_g$ are vectors defined as
\begin{align*}
	(\bv_f)_i & = \int_{[(i - 1)/n, i/n)} f(x) dx, \\
	(\bv_g)_i & = \int_{[(i - 1)/n, i/n)} g(x) dx.
\end{align*}
Next we bound $\Vert \bv_f \Vert_2$, the norm $\Vert \bv_g \Vert$ can be bounded analogously. Define a Hilbert space $L_2^i := L_2 ([(i - 1)/n, i/n))$. Then
\begin{align*}
	|(\bv_f)_i| = |\langle 1, f \rangle_{L_2^i} | \le \Vert 1 \Vert_{L_2^i} \cdot \Vert f \Vert_{L_2^i}.
\end{align*}
Meanwhile, we have
\begin{align*}
	\sum_{i = 1}^n \Vert f \Vert_{L_2^i}^2 = \int_0^1 f^2(x) dx = 1.
\end{align*}
Since $\Vert 1 \Vert_{L_2^i} = \sqrt{\mu ([(i - 1)/n, i/n))} = 1/\sqrt{n}$, we obtain 
\begin{align*}
    \Vert \bv_f \Vert_2 \leq \sqrt{\sum_{i = 1}^n |(\bv_f)_i|^2} \leq \sqrt{ \frac{1}{n} \sum_{i = 1}^n \|f\|^2_{L_2^i}} \le \frac{1}{\sqrt{n}}
\end{align*}
Thus,
\begin{align*}
	\langle g, \bbW_{\bT}[f] \rangle \le \frac{1}{n} \Vert \bT \Vert_{\op}.
\end{align*}
The dual of $L_2[0, 1]$ coincides with $L_2[0, 1]$, so 
\begin{align*}
	\opnorm{\bbW_{\bT}} = \sup_{f, g, \Vert f \Vert_{L_2} = \Vert g \Vert_{L_2} = 1} 
	\langle g, \bbW_{\bT} [f] \rangle \le \frac{1}{n} \Vert \bT \Vert_{\op}.
\end{align*}
The upper bound can be attained by taking $f, g$ equal to
\begin{align*}
	f & = \sum_{i = 1}^n \bu_i I{[(i - 1)/n, i/n)}, \\
	g & = \sum_{i = 1}^n \bv_i I{[(i - 1)/n, i/n)},
\end{align*} 
where $\bu$ and $\bv$ are the right and left singular vectors, respectively, corresponding to the largest singular value of $\bT$.

\section{Proof of Corollary~\ref{corollary: convergene rate targeted interventions}}
\label{section: convergence rate targeted interventions}

Instead of the network game $G(\bA, \btheta)$, we consider the corresponding graphon game $G(W_{\bA}, \tilde{\theta})$, where $W_{\bA}$ is the empirical graphon of the network $\bA$. Define function $\hat{t}^\circ, \hat{t}$ as follows:
\begin{align*}
    \hat{t}(x) & = \sum_{i = 1}^N \hat{\btheta}_i \cdot \indicator{x \in [(i- 1)/N; i/N)}, \\
    \hat{t}^\circ(x) &= \sum_{i = 1}^N \hat{\btheta}'_i \cdot \indicator{x \in [(i- 1)/N; i/N)}.
\end{align*}
Then, we have
\begin{align*}
    T_{(\bA, \btheta)}(\hat{\btheta}) - T_{(\bA, \btheta)}(\hat{\btheta}') = T_{(W_{\bA}, \tilde{\theta})}(\hat{t}) - T_{(W_{\bA}, \tilde{\theta})}(\hat{t}^\circ),
\end{align*}
where $T_{(W_{\bA}, \tilde{\theta})}$ is the target function of the corresponding graphon game $G(W_{\bA}, \tilde{\theta})$ with budget $B$. Then, we rewrite the above as
\begin{align}
    T_{(W_{\bA}, \tilde{\theta})}(\hat{t}) - T_{(W_{\bA}, \tilde{\theta})}(\hat{t}^\circ) = \left (T_{(W_{\bA}, \tilde{\theta})}(\hat{t}) - T_{(W_{\bA}, \tilde{\theta})}(\hat{t}') \right ) + \left (T_{(W_{\bA}, \tilde{\theta})}(\hat{t}') - T_{(W_{\bA}, \tilde{\theta})}(\hat{t}^\circ) \right ), \label{eq: regret decomposition into discretezation and graphon error}
\end{align}
To bound the second term, we use
\begin{align}
\label{eq: discritization error - setup}
    T_{(W_{\bA}, \tilde{\theta})}(\hat{t}') - T_{(W_{\bA}, \tilde{\theta})}(\hat{t}^\circ) & = \frac{1}{2} \left \langle \cR^2\gamma, \bbW_{\bA})[\tilde{\theta} + \hat{t}'], \tilde{\theta} + \hat{t}' \right \rangle - \frac{1}{2} \left \langle \cR^2\gamma, \bbW_{\bA})[\tilde{\theta} + \hat{t}^\circ], \tilde{\theta} + \hat{t}^\circ \right \rangle \nonumber \\
    & =  \frac{1}{2}\left \langle \cR^2(\gamma, \bbW_{\bA})[\hat{t}' - \hat{t}^\circ], \tilde{\theta} + \hat{t}' \right \rangle + \frac{1}{2} \left \langle \cR^2\gamma, \bbW_{\bA})[\tilde{\theta} + \hat{t}^\circ], \hat{t}' - \hat{t}^\circ \right \rangle \nonumber \\
    & \le \frac{\Vert \hat{t}' - \hat{t}^\circ \Vert_{L_2} (\Vert \tilde{\theta} \Vert_{L_2} + \sqrt{B})}{(1 - \gamma \Vert \bbW_{\bA} \Vert_{\op})^2} = \frac{\Vert \hat{t}' - \hat{t}^\circ \Vert_{L_2} (\Vert \tilde{\theta} \Vert_{L_2} + \sqrt{B})}{(1 - \gamma \Vert \bA \Vert_{\op} / N)^2}.
\end{align}

Then, we bound $\Vert \hat{t}' - \hat{t}^\circ \Vert_{L_2}$ as follows:
\begin{align*}
    \Vert \hat{t}' - \hat{t}^\circ \Vert^2_{L_2} & = \sum_{i = 1}^N \int_{(i - 1)/N}^{i/N} \left (\hat{t}'(z) - N \int_{(i - 1)/N}^{i/N} \hat{t}'(x) dx\right )^2 dz.
\end{align*}
The optimizer $\hat{t}'$ belongs to $\operatorname{Im}{\bbW}_{\widehat{\bQ}'} + \operatorname{span}(\tilde{\theta})$. Let $g \in \operatorname{Im}{\bbW}_{\widehat{\bQ}'}$ be a function such that $\hat{t}' - g$ is the projection of $\hat{t}'$ on $\operatorname{span}(\tilde{\theta})$, so $\hat{t}' - g$ is constant on $((i - 1)/N; i/N)$ for $i = 1, \ldots, N$. It implies
\begin{align*}
    \Vert \hat{t}' - \hat{t}^\circ \Vert^2_{L_2} & = \sum_{i = 1}^N \int_{(i - 1)/N}^{i/N} \left (g(z) - N \int_{(i - 1)/N}^{i/N} g(x) dx\right )^2 dz.
\end{align*}

For $j = 0, \ldots, n - 1$, define $i_j \in [N]$ as such integer that $[(i_j - 1)/N; i_j/N)$ contains $j/n$. Note that since $N \ge n$, $i_j$ is defined correctly, i.e. there is no interval $[(i - 1)/N; i/N)$ that contains $j/n$ and $(j + 1)/n$ for some $j$. Since $g(x)$ is constant on intervals $(j/n; (j + 1)/n)$ for $j = 0, \ldots, n$, $g(x)$ is constant outside $S_{N} = \cup_{j = 0}^{n - 1} [(i_j - 1)/N; i_j/N]$, so
\begin{align*}
    \Vert \hat{t}' - \hat{t}^\circ \Vert^2_{L_2} & = \sum_{j = 0}^{n - 1} \int_{(i_j - 1)/N}^{i_j / N} \left (g(z) - N \int_{(i - 1)/N}^{i/N} g(x) dx \right )^2 dz \le  \sum_{j = 0}^{n - 1} \int_{(i_j - 1)/N}^{i_j / N}  (g)^2(x) dx \\
    & \le \sum_{j = 0}^{n - 1} \int_{(i_j - 1)/N}^{i_j / N}  (\hat{t}')^2(x) dx.
\end{align*}
Using again the fact that $\hat{t}'(\cdot)$ is constant on $(j/n; (j + 1)/n)$, we deduce that
\begin{align*}
   \Vert \hat{t}' - \hat{t}^\circ \Vert^2_{L_2} & \le \sum_{j = 0}^{n - 1} \int_{(i_j - 1)/N}^{i_j / N}  (\hat{t}')^2(x) dx  = \int_{S_N} (\hat{t}')^2(x) dx = \sum_{j = 0}^{n - 1} \int_{(j/n; (j + 1)/n) \cap S_N} (\hat{t}'(x))^2 dx \\
   & = \sum_{j = 0}^{n - 1} \mu \left ( (j/n; (j + 1)/n) \cap S_N\right )  \int_{j/n}^{(j + 1)/n} (\hat{t}'(x))^2 dx\\
    & \le \sum_{j = 0}^{n - 1} \frac{2}{N}  \int_{j/n}^{(j + 1)/n} (\hat{t}'(x))^2 dx = \frac{2 \Vert \hat{t}' \Vert_{L_2}^2}{N} = \frac{2 B}{N}.
\end{align*}
Substituting the above back to~\eqref{eq: discritization error - setup}, we obtain
\begin{align}
\label{eq: discritezation bound - final}
    T_{(W_{\bA}, \tilde{\theta})}(\hat{t}') - T_{(W_{\bA}, \tilde{\theta})}(\hat{t}^\circ) \le \sqrt{\frac{2B}{N}} \cdot \frac{(\Vert \tilde{\theta} \Vert_{L_2} + \sqrt{B})}{(1 - \gamma \Vert \bA \Vert_{\op} / N)^2} = \ttb.
\end{align}

To bound $T_{(W_{\bA}, \tilde{\theta})}(\hat{t}) - T_{(W_{\bA, \tilde{\theta}})}(\hat{t}')$, we use Theorem~\ref{theorem: graphon perturbed welfare problem}, that implies
\begin{align}
    & T_{(W_{\bA}, \tilde{\theta})}(\hat{t}) - T_{(W_{\bA, \tilde{\theta}})}(\hat{t}')  \le \frac{4 \gamma \Vert W_{\bA} - \rho_N/\rho_n W_{\widehat{\bQ}'} \Vert_{\op}}{(1 - \gamma [\Vert \bbW_{\bA} \Vert_{\op} \vee \Vert \rho_N / \rho_n \bbW_{\widehat{\bQ}'} \Vert_{\op}])^3} \nonumber \\
    & \qquad \qquad \le 4 \gamma  \cdot \frac{\Vert \bbW_{\bA} - \bbW_{\bQ} \Vert_{\op} + \Vert \bbW_{\bQ} - \rho_N \bbW \Vert_{\op} + \rho_N/ \rho_n \Vert \rho_n \bbW - \bbW_{\widehat{\bQ}'} \Vert_{\op}}{(1 - \gamma [\Vert \bbW_{\bA} \Vert_{\op} \vee \Vert \rho_N / \rho_n \bbW_{\widehat{\bQ}'} \Vert_{\op}])^3} \nonumber \\
    & \qquad \qquad = \tau \cdot \left ( \Vert \bbW_{\bA} - \bbW_{\bQ} \Vert_{\op} + \Vert \bbW_{\bQ} - \rho_N \bbW \Vert_{\op} + \rho_N/ \rho_n \Vert \rho_n \bbW - \bbW_{\widehat{\bQ}'} \Vert_{\op}\right ) \label{eq: graphon regret}
\end{align}
where $\bQ$ is the matrix of connection probabilities for the matrix $\bA$.

Denoting the diagonal of $\bQ$ by $\bD_{\bQ}$, we bound
\begin{align*}
    \Vert \bbW_{\bA} - \bbW_{\bQ} \Vert_{\op} & = \frac{\Vert \bA - \bQ \Vert_{\op}}{N} \le \frac{\Vert \bA - \bQ + \bD_{\bQ} \Vert_{\op}}{N} + \frac{\Vert \bD_{\bQ} \Vert_{\op}}{N} \\
    & \le 6 \sqrt{\frac{\rho_N}{N}} + \frac{\sqrt{\constantBandeira \log \frac{6N}{\delta}}}{N}+ \frac{\rho_N}{N} \le 8 \sqrt{\frac{\rho_N}{N}}
\end{align*}
with probability at least $1 - \delta/3$ due to Corollary~\ref{corollary: symmetric matrix concentration} and assumption $N \rho_N \ge \constantBandeira \log \frac{6 N}{\delta}$. Then, combining the above, \eqref{eq: discritezation bound - final} and~\eqref{eq: graphon regret} with~\eqref{eq: regret decomposition into discretezation and graphon error}, we obtain the following bound on $T_{(\bA, \btheta)}(\hat{\btheta}) - T_{(\bA, \btheta)}(\hat{\btheta}')$:
\begin{align}
    T_{(\bA, \btheta)}(\hat{\btheta}) - T_{(\bA, \btheta)}(\hat{\btheta}') & \le \ttb + 8 \tau \sqrt{\frac{\rho_N}{N}} \nonumber \\
    &\quad + \tau \cdot \left (  \Vert \bbW_{\bQ} - \rho_N \bbW \Vert_{\op} + \rho_N/ \rho_n \Vert \rho_n \bbW - \bbW_{\widehat{\bQ}'} \Vert_{\op} \right ). \label{eq: bound up to graphon type}
\end{align}

To bound $\Vert \bbW_{\bQ} - \rho_N \bbW \Vert_{\op}$ and $\Vert \rho_n \bbW - \bbW_{\widehat{\bQ}} \Vert_{\op}$, we consider three different cases.

\noindent \textbf{Case 1. The graphon $W$ is an SBM with $k$ communities.} By the lines of the proof of Theorem~\ref{theorem: SBM convergence rate}, we have
\begin{align*}
    \Vert \rho_N \bbW - \bbW_{\bQ} \Vert_{\op} \le \rho_N \left ( \frac{32k \sum_{a = 1}^{k - 1}e_a}{N} \log \frac{6k}{\delta} \right )^{1/4}
\end{align*}
with probability at least $1 - \delta/3$, see~\eqref{eq: error of connection probabilities graphon}. 

Finally, in this case, $\widehat{\bQ}' = (\widehat{\bQ}_{\lambda} + \widehat{\bQ}^T_{\lambda})/2$ for the hard-thresholding estimator $\widehat{\bQ}_{\lambda}$ of $\bQ'$. By conclusion of this theorem, we have
\begin{align*}
    \frac{\rho_N}{\rho_n} \Vert \rho_n \bbW - \widehat{\bbW} \Vert_{\op} \le \rho_N \left ( \frac{32k \sum_{a = 1}^{k - 1}e_a}{n} \log \left ( \frac{12k}{\delta} \right) \right )^{1/4} + 11 \sqrt{\frac{\rho_N}{n \rho_n}}
\end{align*}
and $\rank(\widehat{\bbW})  \le 2 \cdot \rank(\widehat{\bQ}_{\lambda}) \le 2 k$ with probability at least $1 - \delta/3$, provided $n \rho_n \ge \constantBandeira \log (12 n/\delta)$. Due to~\eqref{eq: bound up to graphon type}, it yields
\begin{align*}
    T_{(\bA, \btheta)}(\hat{\btheta}) - T_{(\bA, \btheta)}(\hat{\btheta}') & \le \ttb + \tau \cdot \left (8\sqrt{\frac{\rho_N}{N}}+ 8 \rho_N \left ( \frac{k \sum_{a = 1}^{k - 1}e_a}{n} \log \left ( \frac{12k}{\delta} \right) \right )^{1/4} +  11 \sqrt{\frac{\rho_N}{n \rho_n}}\right ) \\
    & \le \ttb + 24\tau \cdot \left (\sqrt{\frac{\rho_N}{n \rho_n}}+  \rho_N \left ( \frac{k \sum_{a = 1}^{k - 1}e_a}{n} \log \left ( \frac{12k}{\delta} \right) \right )^{1/4} \right )
\end{align*}
with probability at least $1 - \delta$,
where we used $n \rho_n \le n \le N$.

\noindent \textbf{Case 2. The graphon $W$ is $(H, \alpha)$-H\"older.} By the lines of the proof of Theorem~\ref{theorem: holder convergence rate}, we have
\begin{align*}
    \Vert \rho_N \bbW - \bbW_{\bQ} \Vert_{\op} \le  \rho_N C_1(H, \alpha) \left (  \frac{\log(12/\delta)}{N}\right )^{\beta/2},
\end{align*}
see~\eqref{eq: holder matrix of connection probabilities difference} with probability at least $1 - \delta/3$. We have $\widehat{\bQ}' = (\widehat{\bQ}^*_{\beta} + (\widehat{\bQ}^*_{\beta})^{T})/2$ for the estimator $\widehat{\bQ}^*_{\beta}$ applied to $\bA'$, so by the conclusion of this theorem, we have
\begin{align*}
    \Vert \rho_n \bbW - \widehat{\bbW} \Vert_{\op} \le \rho_n  C \cdot \left ( \frac{\log (24 /\delta)}{n} \right )^{\beta/2} + C \min \left \{\sqrt{\frac{\rho_n}{n}}, \; \left ( \frac{\rho_n^{\beta + 1}}{n^{2 \beta}} \right )^{1/(2 \beta + 1)} \right \},
\end{align*}
and
\begin{align*}
    \rank(\widehat{\bbW}) \le C' \cdot \begin{cases}
            \min \Big \{\frac{n^{\beta/(2 \alpha + 1)}}{\log^{\beta/(2 \alpha + 1)}(24/\delta)}, \; (n^2 \rho_n)^{2\beta (2 \beta + 1)^{-1} / (2\alpha + 1)} \Big \}, & \text{ if } n^{2 \beta - 1} > \rho_n, \\
            (n \rho_n)^{1/(2 \alpha + 1)}, & \text{ otherwise.}
        \end{cases}
\end{align*}
with probability at least $1 - \delta/3$, provided the conditions of Theorem~\ref{theorem: holder convergence rate} hold with $\delta/3$ in place of $\delta$.  Together with bound~\eqref{eq: bound up to graphon type}, the above implies
\begin{align*}
    & T_{(\bA, \btheta)}(\hat{\btheta}) - T_{(\bA, \btheta)}(\hat{\btheta}') \le \ttb  \\
    & \qquad \qquad + \tau \cdot \left ( 8 \sqrt{\frac{\rho_N}{N}} + C \rho_N \left ( \frac{\log (24 /\delta)}{n} \right )^{\beta/2} + C \min \left \{\sqrt{\frac{\rho_N}{n \rho_n}}, \; \frac{\rho_N}{\rho_n}\left ( \frac{\rho_n^{\beta + 1}}{n^{2 \beta}} \right )^{1/(2 \beta + 1)} \right \} \right )
\end{align*}

\noindent \textbf{Case 3. The graphon $W$ is $(M, r)$-analytic.} By the lines of the proof of Theorem~\ref{theorem: analityc convergence rate}, part i), the graphon $W$ is $(H, 1)$-H\"older for $H = \max \{4 M r^{-1}, 2 r^{-1}\}$, so, repeating the proof of Case 2, we conclude that, with probability at least $1 - \delta$, we have
\begin{align*}
     T_{(\bA, \btheta)}(\hat{\btheta}) - T_{(\bA, \btheta)}(\hat{\btheta}') \le \ttb  + \tau \cdot \left ( 8 \sqrt{\frac{\rho_N}{N}} + C \rho_N \left ( \frac{\log (24 /\delta)}{n} \right )^{1/2} + C \frac{\rho_N}{n^{2/3} \rho_n^{1/3}} \right ),
\end{align*}
and $\rank(\widehat{\bbW}) \le C_2(M, r) (\log (n \rho_n) + \log \log \frac{24}{\delta})$ by the conclusion of Theorem~\ref{theorem: analityc convergence rate}.

\section{Tools}\label{appendix: tools}

In this section, we review established results from the literature that are utilized in the proofs.

\subsection{Concentration inequalities}

First, we state a standard concentration inequality for sums of independent random variables.

\begin{prop}[Bernstein's inequality, see Equation (2.10) in \citep{boucheron_concentration_inequalities}]
    Let $X_1, \ldots, X_n$ be independent random variables such that $\max_i |X_i| \le M$. Then
    \begin{align*}
    	\bbP \left (\sum_{i = 1}^n X_i - \sum_{i = 1}^n \bbE (X_i) \ge  t
    	\right )
    	\le \exp \left (
    		- \frac{t^2/2}{\left(\sum_{i = 1}^n \Var (X_i)\right) + M t/3}
    	\right ), 
    \end{align*}
    and
    \begin{align*}
    	\bbP \left (
    		\left|\sum_{i = 1}^n X_i - \sum_{i = 1}^n \bbE (X_i) \right|\ge  t
    	\right )
    	\le 2\exp \left (
    		- \frac{t^2/2}{\left(\sum_{i = 1}^n \Var (X_i)\right) + M t/3}
    	\right ).
    \end{align*}
\end{prop}

Then, we introduce a concentration inequality for the operator norm of random matrices with independent entries.

\begin{prop}[Matrix concentration inequality, see Corollary 3.12, Remark 3.13 in~\citep{bandeira2016}]
\label{proposition: matrix concentration}
	Let $\bX$ be an $n \times n$ matrix whose entries $\bX_{i j}$ are independent, centered and bounded random variables. Then, for any $0 < \epsilon \le 1/2$,  there exists a universal constant $c_\epsilon$ such that for every $t > 0$
	\begin{align*}
		\bbP \left ( \Vert \bX \Vert_{\op} \ge (1 + \epsilon) 2 \sqrt{2} \sigma + t \right ) \le n e^{-t^2 / c_\epsilon \sigma_*^2},
	\end{align*}
	where we have defined
	\begin{align*}
		\sigma : = \max_i \sqrt{\sum_j \bbE (\bX_{i j}^2}), \quad \sigma_* : = \max_{i j}  \Vert \bX_{i j} \Vert_{\infty}.
	\end{align*}
\end{prop}

The above proposition yields the following corollary for symmetric matrices.

\begin{corollary}
\label{corollary: symmetric matrix concentration}
    Let $\bX$ be a symmetric $n \times n$ matrix whose entries $\bX_{i j}, 1 \le i \le j \le n,$ are independent, centered and bounded random variables. Then, for any $0 < \epsilon \le 1/2$, there exists a universal constant $c_\epsilon$ such that for every $t > 0$
    \begin{align*}
        \bbP \left ( \Vert \bX \Vert_{\op} \ge (1 + \epsilon) 4 \sqrt{2} \sigma + t \right ) \le 2n e^{-t^2 / 4c_\epsilon \sigma_*^2},
    \end{align*}
    where we have defined
    \begin{align*}
        \sigma : = \max_i \sqrt{\sum_j \bbE (\bX_{i j}^2}), \quad \sigma_* : = \max_{i j}  \Vert \bX_{i j} \Vert_{\infty}.
    \end{align*}
\end{corollary}

\begin{proof}
    Set $\bX^+$ to be the upper triangular part of $\bX$, i.e., $\bX^+_{i j} = \bX_{i j}$ for $1 \le i < j \le n$ and $\bX^+_{i j} = 0$ otherwise. Set $\bX^- = \bX - \bX^+$. By Proposition~\ref{proposition: matrix concentration}, we have
    \begin{align*}
        \bbP \left ( \Vert \bX^+ \Vert_{\op} \ge (1 + \epsilon) 2 \sqrt{2} \sigma + t/2 \right ) \le n e^{-t^2 / 4 c_\epsilon \sigma_*^2}, \\
        \bbP \left ( \Vert \bX^- \Vert_{\op} \ge (1 + \epsilon) 2 \sqrt{2} \sigma + t/2 \right ) \le n e^{-t^2 / 4 c_\epsilon \sigma_*^2}.
    \end{align*}
    Then, by the triangle inequality, we have
    \begin{align*}
        \bbP \left ( \Vert \bX \Vert_{\op} \ge  (1 + \epsilon) 4 \sqrt{2} \sigma + t \right ) & \le \bbP \left ( \Vert \bX^+ \Vert_{\op} + \Vert \bX^- \Vert_{\op} \ge (1 + \epsilon) 4x \sqrt{2} \sigma + t \right ) \\
        & \le \bbP \left ( \Vert \bX^+ \Vert_{\op} \ge (1 + \epsilon) 2 \sqrt{2} \sigma + t/2 \right ) \\
        & \quad + \bbP \left ( \Vert \bX^- \Vert_{\op} \ge (1 + \epsilon) 2 \sqrt{2} \sigma + t/2 \right ) \\
        & \le 2 n e^{-t^2 / 4 c_\epsilon \sigma_*^2}.  \qedhere
    \end{align*}
\end{proof}

In our paper, we choose $\epsilon = \frac{3 \sqrt{2}}{4} - 1 \approx 0.06$. Then, $(1 + \epsilon) 4 \sqrt{2} = 6$. We denote the constant $c_\epsilon \vee 1$ by $\constantBandeira$. 

To derive concentration inequality for order statistics of uniform random variables, we use the following celebrated result of Massart, who extended earlier theorem of Dvoretzky, Kiefer and Wolfowitz~\citep{dvoretzky1956asymptotic}.

\begin{theorem}[\citep{massart1990tight}]
\label{theorem: massart DKW}
    Let $X_1, \ldots, X_n$ be i.i.d. random samples from a distribution with defined by CDF $F(\cdot)$. Consider the empirical CDF function $\widehat{F}(\cdot)$:
    \begin{align*}
        \widehat{F}(x) = \frac{1}{n} \sum_{i = 1}^n \indicator{X_i \le x}.
    \end{align*}
    Then, for any $t > 0$, we have
    \begin{align*}
        \bbP \left (\sqrt{n} \sup_{x \in \bbR} |F(x) - \widehat{F}(x) | \ge t \right ) \le 2 e^{-2t^2}.
    \end{align*}
\end{theorem}

Applying the above theorem to $\xi_1, \ldots, \xi_n$ uniformly distributed on $[0, 1]$ and noting that $F(\xi_{(i)}) = \xi_{(i)}$ in this case, we obtain the following corollary.

\begin{corollary}
\label{corollary: order statistics unifrom concentration}
Let $\xi_1, \ldots, \xi_n$ be i.i.d. random variables uniformly distributed on $[0, 1]$. Then, for any $t > 0$, we have
\begin{align*}
    \bbP \left ( \max_{i = 1, \ldots, n} |\xi_{(i)} - i/n|  \ge \frac{t}{\sqrt{n}} \right ) \le 2 e^{- 2 t^2}.
\end{align*}
\end{corollary}

\subsection{Norms and singular values of matrices}

\begin{prop}[Weyl's inequality, see \citep{bisgard2020}]
\label{proposition: weyl inequality}
	Let $\bA, \bB \in \bbR^{n \times n}$ be arbitrary matrices, and let $\sigma_k(\bA)$ and $\sigma_k(\bB)$, $k \in [n]$  be their singular values in descending order. Then
	\begin{align*}
		\sigma_{i + j - 1}(\bA + \bB) \le \sigma_{i}(\bA) + \sigma_j(\bB)
	\end{align*}
	for $1 \le i, j, i+ j - 1 \le  n$. If $\bA, \bB$ are symmetric, then the same inequality holds for their eigenvalues.
\end{prop}

\begin{prop}[Eckart-Young-Mirsky theorem, see \citep{bisgard2020}]
\label{proposition: eckart-yang theorem}
	Let $\bA$ be an arbitrary matrix and $\sigma_1 \ge \ldots \ge \sigma_n$ be its singular values. Then, we have
	\begin{align*}
		\min_{\rank(\bB) \le r} \Vert \bA - \bB \Vert_{F} = \sqrt{\sum_{i > r} \sigma_i^2},
	\end{align*}
	and the minimizer is unique.
\end{prop}

\subsection{Norms of matrices and graphons}
The cut norm of a matrix $\bA \in [0,1]^{n \times n}$ is defined by
\begin{align}
    \|\bA\|_\square := \frac{1}{n^2} \max_{S, T \subseteq [n]} \left|\sum_{i \in S, j \in T} A_{ij}\right|. 
\end{align} 
We have that  $\Vert \bA \Vert_{\op} \ge n \Vert \bA \Vert_{\square}.$
\begin{proof}
    Let $S^*, T^*$ be the sets that provide the maximum  in  the definition of $\|\bA\|_\square$, i.e., 
    \begin{align*}
        \|\bA\|_\square = \frac{1}{n^2} \left|\sum_{i \in S^*, j \in T^*} A_{ij}\right|.
    \end{align*}
    Define $\bu^{S^*} \in \bbR^n$ such that $\bu^{S^*}_i = 1$ if $i \in S^*$, $\bu^{S^*}_i = 0$ if $i \notin S^*$.  Similarly, define $\bu^{T^*}$. Then,
    \begin{align*}
        \|\bA\|_\square = \frac{1}{n^2} \langle \bA \bu^{T^*}, \bu^{S^*} \rangle \leq \frac{1}{n^2} \| \bA \bu^{T^*} \| \| \bu^{S^*} \| \leq \frac{1}{n^2} \| \bA \|_{\op} \| \bu^{T^*} \| \| \bu^{S^*} \| \leq \frac{1}{n^2} \| \bA \|_{\op} \sqrt{n} \sqrt{n} = \frac{\|\bA\|_{\op}}{n},  
    \end{align*}
    where the first inequality is the Cauchy-Schwartz inequality and the second inequality is due to submultiplicativity of the $2$-norm. 
\end{proof}


\subsection{Rank of kernels} \label{sec:rank_function}

Let $\bbW$ be the integral operator associated with the graphon $W$ defined as
$[\bbW f](x)= \int_0^1 W(x, y) f(y) dy$ for all $x \in [0,1]$, for all $f \in L_2[0,1]$. 
The operator $\bbW$ is a Hilbert-Schmidt operator. In particular, $\bbW$ is a compact operator, meaning it possesses a discrete spectrum. Hence, $\bbW$ admits a spectral decomposition.
\begin{equation*}
W(x, y) \sim \sum_{k} \lambda_{k} f_{k}(x) f_{k}(y) 
\end{equation*}
The rank of a kernel $W$ is determined by the rank of its associated kernel operator $\bbW$. While this rank is typically infinite, our focus lies on cases where it's finite. 
If the rank $r$ of $W$ is finite, then  the spectral decomposition  of $W$ will be finite. In particular, 
\begin{equation*}
W(x, y)=\sum_{k=1}^{r} \lambda_{k} f_{k}(x) f_{k}(y). 
\end{equation*}
holds almost everywhere. 
Note that every SBM graphon (Definition \ref{definition: SBM}) has a finite rank. Indeed, if $W$ is a stepfunction with $m$ steps, then every function in the range of $\bbW$ is a stepfunction with the same $m$ steps, and so the dimension of this range is at most $m$.
It is also easy to see that the sum and product of
two kernels with finite rank have a finite rank (see \citep[Sections 7.5, 14.4]{lovasz}).

\subsection{Minimax bounds}
\label{sec:minimax bounds}

Here, we briefly describe the simplest framework for establishing the minimax lower bounds on a risk. For deeper introduction, see~\citep{tsybakov_introduction_2009}. Suppose that we have some family of distributions $\P_{\theta}, \theta \in \Theta,$ over a sample space $\cX$ parametrized by parametric space $\Theta$. Let $\cR : \Theta \times \Theta \to \bbR_+$ be a semi-distance on $\Theta$. Any function \(\widehat{\theta}: \cX \to \Theta\) is referred to as an estimator of the parameter \(\theta\). Given a sample \(X\) drawn from the distribution \(\P_\theta\), the quantity \(\cR(\widehat{\theta}(X), \theta)\) serves as a natural measure of the estimator’s quality. A fundamental question in statistical estimation is determining the minimax lower bound on the risk \(\cR(\widehat{\theta}(X), \theta)\), representing the best achievable performance across all estimators.

If two distributions \(\P_{\theta_1}\) and \(\P_{\theta_2}\) are close in some sense, they generate similar samples \(X\) from the sample space \(\cX\). When \(X\) is low-dimensional, distinguishing between \(\theta_1\) and \(\theta_2\) becomes challenging, particularly if both parameters are high-dimensional, making it difficult for any estimator to separate them effectively. To formalize this intuition, we examine the Kullback-Leibler (KL) divergence between two discrete distributions \(\P_{\theta_1}\) and \(\P_{\theta_2}\), defined as follows:
\begin{align*}
    \KL{\P_{\theta_2}}{\P_{\theta_1}} = \sum_{x \in \cX} \P_{\theta_2}(X = x) \log \frac{\P_{\theta_2}(X = x)}{\P_{\theta_1}(X =x)}.
\end{align*}
It admits several basic properties:
\begin{enumerate}
    \item it is non-negative and equals zero if and only if $\P_{\theta_1}$ and $\P_{\theta_2}$ coincide;
    \item the KL-divergence does not exceed the chi-square divergence:
    \begin{align*}
        \KL{\P_{\theta_2}}{\P_{\theta_1}} \le \sum_{x \in \cX} \frac{\left (\P_{\theta_2}(X = x) -  \P_{\theta_1}(X = x) \right )^2}{\P_{\theta_1}(X = x)};
    \end{align*}
    \item for tensor product of measures, we have
    \begin{align*}
        \KL{\P_1 \otimes \P_1'}{\P_2 \otimes \P_2'} = \KL{\P_1}{\P_2} + \KL{\P_1'}{\P_2'}.
    \end{align*}
\end{enumerate}
If \(\KL{\P_{\theta_1}}{\P_{\theta_2}}\) is small, any estimator \(\widehat{\theta}\) is expected to struggle to distinguish between \(\theta_1\) and \(\theta_2\). More precisely, the following theorem establishes that if \(X \sim \P_{\theta}\) with \(\theta \in \{ \theta_1, \theta_2 \}\), then the risk \(\cR(\widehat{\theta}(X), \theta)\) is at least \(\cR(\theta_1, \theta_2) / 2\), with a probability that increases as \(\KL{\P_{\theta_1}}{\P_{\theta_2}} \to 0\).
\begin{theorem}
    \label{theorem:minimax bounds}
    Let $\theta_1, \theta_2$ be any two elements of parametric space $\Theta$ such that $\KL{\P_{\theta_1}}{\P_{\theta_2}} \le \alpha$. Then
    \begin{align*}
        \inf_{\widehat \theta} \max_{\theta \in \{\theta_1, \theta_2\}} \P_{\theta} \left\{\cR(\widehat{\theta}, \theta) \ge \cR(\theta_1, \theta_2) / 2 \right \} \ge \max \left \{ e^{-\alpha} / 4, \frac{1 - \sqrt{\alpha/2}}{2} \right \}.
    \end{align*}
\end{theorem}
For proofs, see Chapter~2 from~\citep{tsybakov_introduction_2009}.
For general distributions such that $\bbP_1 \ll \bbP_0$, the KL-divergence can be computed as follows:
\begin{align*}
    \KL{\bbP_1}{\bbP_0} = \int \log \left ( \frac{d \bbP_1}{d \bbP_0} \right ) d \bbP_1 = \int \log \left ( \frac{d \bbP_1}{d \bbP_0} \right ) \frac{d \bbP_1}{d \bbP_0} d \bbP_0.
\end{align*}
It can be useful to upper bound the KL-divergence by the chi-square divergence, that is defined as follows:
\begin{align*}
    \chi^2(\bbP_1, \bbP_0) = \int \left ( \frac{d \bbP_1}{d \bbP_0} - 1 \right )^2 d \bbP_0 = \int \left ( \frac{d \bbP_1}{d \bbP_0} \right )^2 d \bbP_0 - 1,
\end{align*}
which may be easier to compute in some cases. In particular, the following theorem provides a lower bound on the risk of any estimator $\widehat{\theta}$ based on the chi-square divergence between $\bbP_{\theta_j}$ and $\bbP_{\theta_0}$ for $j = 1, \ldots, M$.

\begin{theorem}[\citep{tsybakov_introduction_2009},Theorem 2.6]
\label{theorem: chi-square lower bound} Assume that $M \ge 2$ and suppose that the parameter space $\Theta$ contains elements $\theta_0, \theta_1, \ldots, \theta_M$ defining distributions $\bbP_0, \ldots, \bbP_M$ such that
\begin{enumerate}
    \item $d(\theta_i, \theta_j) \ge 2 s$ for all $0 \le i < j \le M$,
    \item $\bbP_j \ll \bbP_0$ for any $j = 1, \ldots, M$ and
    \begin{align*}
        \frac{1}{M} \sum_{i = 1}^M \chi^2(\bbP_j, \bbP_0) \le \alpha M
    \end{align*}
    with $\alpha \le 1/2$. 
\end{enumerate}
Then
\begin{align*}
    \inf_{\widehat{\theta}} \sup_{\theta \in \Theta} \bbP_{\theta}(d(\widehat{\theta}, \theta) \ge s) \ge \frac{1}{2} \left (1 - \alpha - \frac{1}{M} \right ).
\end{align*}
\end{theorem}

Theorems~\ref{theorem:minimax bounds} and~\ref{theorem: chi-square lower bound} are based on Fano's lemma which variant we also require in our proofs. 

\begin{theorem}[\citep{tsybakov_introduction_2009}, Theorem 2.2]
\label{theorem: fano's lemma}
Let $\bbP_1$, $\bbP_0$ be distributions over a sample space $\cX$. Suppose that $\KL{\bbP_1}{\bbP_0} \le \alpha$. Let $\bj^*$ be a random variable uniformly distributed over $\{0, 1\}$. Consider a random variable $X$ generated by first sampling $\bj^*$ and then sampling $X$ from $\bbP_{\bj^*}$. Then, for any estimator $\widehat{j} : \cX \to \{0, 1\}$ of $\bj^*$, we have
\begin{align*}
    \P\left (\widehat{j}(X) \neq \bj^* \right ) \ge \max \left \{ e^{-\alpha}/4, \frac{1 - \sqrt{\alpha}}{2} \right \}.
\end{align*}
\end{theorem}

\subsection{Data-processing inequality}

\begin{lemma}[Data-processing for KL under a Markov kernel]
	\label{lem:dp_kl_markov_kernel}
	Let $(\mathcal U,\mathcal F)$ and $(\mathcal A,\mathcal G)$ be measurable spaces and
	let $K(\cdot\mid u)$ be a Markov kernel from $\mathcal U$ to $\mathcal A$.
	For $i\in\{0,1\}$ let $\mu_i$ be probability measures on $\mathcal U$ and define the induced
	marginal laws on $\mathcal A$ by
	\[
	P_i(B) := \int_{\mathcal U} K(B\mid u)\,\mu_i(du),
	\qquad B\in\mathcal G,
	\]
	i.e.\ $P_i = \mu_i K$.
	Assume $\mu_1\ll \mu_0$ so that $\mathrm{KL}(\mu_1\|\mu_0)<\infty$.
	Then
	\[
	\mathrm{KL}(P_1\|P_0)\ \le\ \mathrm{KL}(\mu_1\|\mu_0).
	\]
\end{lemma}

\begin{proof}
	Define joint probability measures on $\mathcal U\times \mathcal A$ by
	\[
	\widetilde P_i(du,da) := \mu_i(du)\,K(da\mid u),
	\qquad i\in\{0,1\}.
	\]
	By construction, the $\mathcal A$-marginal of $\widetilde P_i$ is $P_i$.
	
	\smallskip
	\noindent\textbf{(1) KL increases under adding variables (equivalently, decreases under marginalization).}
	Using the chain rule for KL divergence (factorizing by $A$),
	\[
	\mathrm{KL}(\widetilde P_1\|\widetilde P_0)
	=
	\mathrm{KL}(P_1\|P_0)
	\;+\;
	\mathbb E_{A\sim P_1}\!\left[
	\mathrm{KL}\big(\widetilde P_1(U\mid A)\ \|\ \widetilde P_0(U\mid A)\big)
	\right].
	\]
	The conditional KL term is nonnegative, hence
	\[
	\mathrm{KL}(P_1\|P_0)\ \le\ \mathrm{KL}(\widetilde P_1\|\widetilde P_0).
	\]
	
	\smallskip
	\noindent\textbf{(2) Compute the joint KL: the kernel cancels.}
	Since $\widetilde P_i(du,da)=\mu_i(du)K(da\mid u)$ share the same conditional law $K(\cdot\mid u)$,
	\[
	\frac{d\widetilde P_1}{d\widetilde P_0}(u,a)=\frac{d\mu_1}{d\mu_0}(u)
	\quad\text{for }\widetilde P_1\text{-a.e. }(u,a),
	\]
	and thus
	\[
	\mathrm{KL}(\widetilde P_1\|\widetilde P_0)
	=
	\int_{\mathcal U\times\mathcal A}
	\log\!\left(\frac{d\mu_1}{d\mu_0}(u)\right)\,\widetilde P_1(du,da)
	=
	\int_{\mathcal U}
	\log\!\left(\frac{d\mu_1}{d\mu_0}(u)\right)\,\mu_1(du)
	=
	\mathrm{KL}(\mu_1\|\mu_0).
	\]
	
	\smallskip
	Combining (1) and (2) gives $\mathrm{KL}(P_1\|P_0)\le \mathrm{KL}(\mu_1\|\mu_0)$.
\end{proof}

\begin{lemma}[KL contraction under a measurable map]
	Let $S:\mathcal X\to\mathcal Y$ be measurable and let $P,Q$ be probability measures on $\mathcal X$.
	Then
	\[
	\mathrm{KL}(P\circ S^{-1}\,\|\,Q\circ S^{-1}) \le \mathrm{KL}(P\|Q).
	\]
\end{lemma}

\begin{proof}
	This is the data-processing inequality for KL.
	Define the Markov kernel $K(B\mid x):=\mathbf 1\{S(x)\in B\}$ (deterministic channel).
	Then $P\circ S^{-1}=PK$ and $Q\circ S^{-1}=QK$, hence
	$\mathrm{KL}(P\circ S^{-1}\|Q\circ S^{-1})=\mathrm{KL}(PK\|QK)\le \mathrm{KL}(P\|Q)$.
\end{proof}

\end{document}